\documentclass[12pt]{amsart}
\usepackage{amsmath,amssymb,amsfonts,amsthm,amsopn}

\usepackage{mathrsfs}
\usepackage{graphicx,tikz}
\usepackage[all]{xy}
\usepackage{amsmath}
\usepackage{multirow, longtable, makecell, caption, array}
\usepackage{amssymb}
\usepackage{mathtools}
\usepackage{pb-diagram}
\usepackage{cellspace} %
\newtheorem{theorem}{Theorem}[section]
\newtheorem{corollary}[theorem]{Corollary}
\newtheorem{lemma}[theorem]{Lemma}
\newtheorem{definition}[theorem]{Definition}

\newtheorem{proposition}[theorem]{Proposition}
\newtheorem{remark}[theorem]{Remark}
\newcommand{\beqa}{\begin{eqnarray*}}
	\newcommand{\eeqa}{\end{eqnarray*}}
\DeclareMathOperator*{\supp}{supp}
\DeclareMathOperator*{\essupp}{ess\,sup\,}

\newcommand{\field}[1]{\mathbb{#1}}
\newcommand{\bR}{\field{R}}
\newcommand{\bN}{\field{N}}
\newcommand{\bZ}{\field{Z}}
\newcommand{\bC}{\field{C}}

\newcommand{\bT}{\field{T}}
\newcommand{\bH}{\field{H}}
\newcommand{\bA}{\field{A}}
\newcommand{\bG}{\field{G}}
\newcommand{\bK}{\field{K}}
\newcommand{\bS}{\field{S}}
\newcommand{\bU}{\field{U}}

\def\al{\alpha}
\def\be{\beta}
\def\ga{\gamma}
\def\Ga{\Gamma}

\def\la{\lambda}
\def\La{\Lambda}

\def\vp{\varphi}

\def\om{\omega}

\def\si{\sigma}

\def\cF{\mathcal{F}}
\def\cS{\mathcal{S}}

\def\cH{\mathcal{H}}

\def\cG{\mathcal{G}}
\def\hcG{\widehat{\mathcal{G}}}
\def\cM{\mathcal{M}}
\def\cK{\mathcal{K}}
\def\cU{\mathcal{U}}
\def\cA{\mathcal{A}}
\def\cJ{\mathcal{J}}

\def\cC{\mathcal{C}}

\def\cR{\mathcal{R}}
\def\cT{\mathcal{T}}

\def\cA{\mathcal{A}}

\def\sA{\mathscr{A}}
\def\sD{\mathscr{D}}
\def\sfT{\mathsf{T}}
\def\sfM{\mathsf{M}}
\def\Tb{\mathfrak{T}_b}
\def\a{\aleph}
\def\hf{\hat{f}}
\def\hg{\hat{g}}

\def\rd{\bR^d}

\def\rdd{{\bR^{2d}}}
\def\zdd{{\bZ^{2d}}}

\def\zd{\bZ^d}

\def\<{\left<}
\def\>{\right>}

\def\mv1{M_v^1}

\newcommand{\abs}[1]{\left|#1\right|}
\newcommand{\norm}[1]{\left\|#1\right\|}
\newcommand{\opnorm}[1]{
	\left|\mkern-1.5mu\left|\mkern-1.5mu\left|
	#1
	\right|\mkern-1.5mu\right|\mkern-1.5mu\right|}
\def\o{\omega}
\def\a{\alpha}

\def\i{\infty}

\def\Ren{\mathbb{R}^d}

\def\f{\varphi}

\def\Sn2{S_{2}(L^{2}(\Ren))}
\def\S1{S_{1}(L^{2}(\Ren))}
\def\sig00{\sigma_{0,0}}

\def\la{\langle}
\def\ra{\rangle}

\newcommand{\vr}{\varrho}

\newcommand{\A}{\mathcal{A}}

\newcommand{\set}[1]{\left\{\, #1\,\right\}}

\usepackage{pifont}

\def\p{\psi}
\def\A{A^{\p_1,\p_2}_a}
\DeclareMathOperator*{\Opz}{Op_{0}}
\DeclareMathOperator*{\Co}{\textsf{ Co}}
\DeclareMathOperator*{\CoFG}{\textsf{Co}_{\small FG}}
\def\bfx{\textnormal{\textbf{x}}}
\def\bfy{\textnormal{\textbf{y}}}
\def\bfu{\textnormal{\textbf{u}}}
\def\bfw{\textnormal{\textbf{w}}}

\def\bfn{\textnormal{\textbf{n}}}
\def\bfz{\textnormal{\textbf{z}}}
\usepackage{upgreek}
\usepackage[]{bm}

\def\bfth{\boldsymbol{\theta}}

\def\bfdel{\boldsymbol{\delta}}

\begin{document}
	\begin{abstract} We introduce new quasi-Banach modulation spaces on locally compact abelian (LCA) groups which coincide with the classical ones in the Banach setting and prove their main properties. Then we study   Gabor frames  on quasi-lattices, significantly extending the original theory introduced by Gr\"{o}chenig  and Strohmer. These issues are the key tools in   showing  boundedness results for Kohn-Nirenberg and localization operators on modulation spaces and studying their eigenfunctions' properties. In particular,  the results in the Euclidean space are recaptured.
	\end{abstract}
	
	\title[Modulation spaces and localization operators on LCA groups]{Quasi-Banach modulation spaces and localization operators on locally compact abelian groups}
	\author{Federico Bastianoni}
	\address{Dipartimento di Scienze Matematiche, Politecnico di Torino, corso
		Duca degli Abruzzi 24, 10129 Torino, Italy}
	\email{federico.bastianoni@polito.it}
	\thanks{}
	\author{Elena Cordero}
	\address{Dipartimento di Matematica, Universit\`a di Torino, via Carlo Alberto 10, 10123 Torino, Italy}
	\email{elena.cordero@unito.it}
	\thanks{}

	\date{}
\maketitle

\tableofcontents

\section{Introduction}

In the last decades time-frequency analysis and pseudo-differential calculus on locally compact groups acquired increasing interest  for both theoretical and practical reasons. The popularity gained in signal processing by time-frequency representations (see e.g. \cite{ClaMec-pt1-1980} and references therein and \cite{Strohmer2006}), led to the need for discrete versions of the techniques available on $\rd$. Many works have been done on $\zd$, finite 
abelian and elementary groups \cite{AnTol1998,ClaMec-pt2-1980,FeiStro1998,MecHla1997}, and for the $p$-adic groups $\mathbb{Q}_p$  \cite{EnsJakLue2019,Kut2003,GroStr2007}. Since the group laws of the $p$-adic numbers resembles the computer arithmetic, the $\mathbb{Q}_p$ groups appear to be the natural settings for problems in computer science. On the other hand,  the $p$-adic groups and the pseudo-differential calculus on them are essential for $p$-adic quantum theory \cite{Haran1993,RueThiVerWey1989,Vladimirov1990}. More generally,  group theory has caught  the attention of many authors in the last thirty years, as it is witnessed by the huge production on the topic, see for example \cite{AkyRuz2020,feichtinger-modulation,FeiGro1988,FeiGro1989-I,FeiGro1989-II,Gro-AspectsGabor-1998,Grochenig2020,ManRuz2017,Oussa2019,RuzTur-pseudo-2010,RuzTur-2011,RuzTurWir-2014,Tur2016,Tur20,WongLocalization}. \par
In this work we shall focus on topological, locally compact abelian (LCA) groups $\cG$.
The very first motivation that led to this manuscript was the study of eigenfuctions' properties  for localization operators on LCA groups, in the spirit of the results inferred in the Euclidean case \cite{BasCorNic20}. 
Despite the many contributions on pseudo-differential operators acting on groups (cf. e.g., \cite{Strohmer2006, RuzTurWir-2014,WongLocalization}), we believe this is  the first work in this direction. \par 
The  function spaces used for both eigenfunctions and symbols are modulation spaces.
For measuring the eigenfunctions' decay it becomes necessary to extend the Banach cases of modulation spaces $M^{p,q}_m(\cG)$ $1\leq p,q\leq\infty$, originally defined by Feichtinger in his pioneering work \cite{feichtinger-modulation}, to the quasi-Banach setting. This is the first contribution of this paper. 
Although there is a well-established theory for $M^{p,q}_m(\rd)$ including the quasi-Banach cases $0<p,q\leq\infty$ \cite{Galperin2004}, if we abandon the Euclidean setting for  a general group $\cG$ troubles arise. 
In fact, the Banach modulation spaces on groups  $\cG$ introduced in  \cite{feichtinger-modulation} cannot be adapted to the quasi-Banach case.  We overcome this difficulty by getting  inspiration from  the idea of Feichtinger and Gr\"{o}chenig in  \cite{FeiGro1988}, and view  modulation spaces on $\cG$ as particular coorbit spaces over the Heisenberg group $\cG\times\hcG\times\bT$ (cf. Definition \ref{defiHG} below). Again, the coorbit theory proposed by Feichtinger and Gr\"{o}chenig in their  works  \cite{FeiGro1988,FeiGro1989-I,FeiGro1989-II}  is not suitable for  the quasi-Banach case. The right construction is provided by the new coorbit theory started by Rauhut in \cite{Rauhut2007Coorbit} and developed by Voigtlaender in his Ph.D. thesis \cite{Voig2015}, see also \cite{VelthovenVoigtlaender2022}.

For a version of coorbit theory that does not need group representations, but only a continuous frame to start with,  we refer to \cite{FornasierRout2005,RauhutUllrich2011}.

Thanks to  this new theory (see a brief summary in the Appendix A below), we are able to give a definition of modulation spaces on LCA groups which recaptures Feichtinger's orginal one in  \cite{feichtinger-modulation} and deals with the quasi-Banach case.  
To explain the new modulation spaces, we  first need to introduce the main notations.

 We write $\hcG$ for the dual group of $\cG$. Latin letters such as $x$, $y$ and $u$ denote elements in $\cG$ whereas all the characters in $\hcG$, except the identity $\hat{e}$, are indicated by Greek letters like $\xi$, $\omega$ and $\eta$. For the evaluation of a character $\xi\in\hcG$ at a point $x\in\cG$ we write
\begin{equation*}
\la\xi,x\ra\coloneqq\xi(x).
\end{equation*}
For $x\in\cG, \xi\in\hcG$ and a function $f\colon\cG\to\bC$ we define the {\slshape translation operator} $T_x$, the {\slshape modulation operator} $M_\xi$ and the {\slshape time-frequency shift} $\pi(x,\xi)$ as
\begin{equation}\label{Eq-fundamental-operators}
T_xf(y)= f(y-x),\quad M_\xi f(y)= \la\xi,y\ra f(y),\quad \pi(x,\xi)= M_\xi T_x,
\end{equation}
$T_x$ and $M_\xi$ fulfil the so called commutation relations
\begin{equation}\label{Eq-commutation-relations}
M_\xi T_x=\la\xi,x\ra T_x M_\xi.
\end{equation}
For $f,g\in L^2(\cG)$, the {\slshape short-time Fourier transform (STFT) of $f$ with respect to $g$} is given by
\begin{equation}\label{Eq-STFT}
V_g f(x,\xi)= \la f,\pi(x,\xi)g\ra =\int_\cG f(y)\overline{\pi(x,\xi)g(y)}\,dy, \qquad (x,\xi)\in\cG\times\hcG.
\end{equation}
To define modulation spaces $M^{p,q}_m(\cG)$, instead of considering the mixed Lebesgue space $L^{p,q}_m$  to measure the (quasi-)norm of the STFT as in \cite{feichtinger-modulation}, that is
\begin{equation}\label{Eq-norm-Mpq-old}
\norm{f}_{M^{p,q}_m}=\norm{V_g f}_{L^{p,q}_m},
\end{equation}
we use the Wiener space $W(L^\infty,L^{p,q}_m)$-norm (see Definition \ref{definizione-Wiener} and subsequent comments):
\begin{equation}\label{Eq-norm-Mpq-new}
\norm{f}_{M^{p,q}_m}=\norm{V_g f}_{W(L^\infty,L^{p,q}_m)}.
\end{equation}
We develop a new general  theory  which coincides with the classical one when
\begin{itemize}
	\item[(i)] $p,q\geq1$ and $\cG$ is any LCA group;
	\item[(ii)] $0<p,q\leq\infty$ and $\cG=\rd$.
\end{itemize}
Frame expansions and new convolution relations for $M^{p,q}_m(\cG)$ are obtained as well, see Theorem \ref{Th-gabor-expansions-Mpq} and Proposition \ref{Pro-convolution-Mpq} below.\par
Galperin and Samarah proved in \cite[Lemma 3.2]{Galperin2004}, that for any $0<p,q\leq\infty$ there  exists  constant $C>0$ such that for every 
\begin{equation}\label{Eq-open-problem}
\norm{V_gf}_{W(L^\infty,L^{p,q}_m)}\leq C \norm{V_gf}_{L^{p,q}_m},\quad \forall\, f\in M^{p,q}_m(\rd),
\end{equation}
$g$ being the Gaussian. It is of course a natural question whether there are cases for which the (quasi-)norm \eqref{Eq-norm-Mpq-new} is equivalent to the more ``natural" \eqref{Eq-norm-Mpq-old}. In order to answer this question one has to verify \eqref{Eq-open-problem} for some suitable window function $g$. The techniques adopted in \cite{Galperin2004} to prove the above inequality rely on properties of entire functions on $\bC^d$, which cannot be adopted for a general LCA group $\cG$.  Whether the inequality in \eqref{Eq-open-problem} holds true whenever we replace $\rd$ by any LCA group $\cG$ is still an open problem and can be seen as a manifestation of a wider issue concerning coorbit theories, see Rauhut' observations in \cite[Section 6]{Rauhut2007Coorbit}.\par
In this work we are able to give a positive answer when $\cG$ is a discrete or compact group, see the subsequent Lemma \ref{Lem-Analogo-GalSam-Lem3.2}.

Next, we focus on localization operators and their eigenfunctions. 

The localization operator $\A$ with symbol $a$ and windows $\psi_1,\psi_2$  can be formally defined by
\begin{equation*}
\A f(x)=\int_{\cG\times\hcG}a(u,\o)V_{\p_1}f(u,\o)M_{\o} T_{u}\p_{2}(x)\,dud\o.
\end{equation*}
In particular, if $a\in L^\infty(\cG\times\hcG)$ and the windows $\psi_1,\psi_2$ are in $L^2(\cG)$ then $\A$ is bounded on $L^2(\cG)$, cf. \cite{WongLocalization}.

For a linear bounded operator $T$ on $L^2(\cG)$ we denote by $\sigma(T)$ the {\slshape spectrum} of $T$, that is the set $\{\lambda\in\bC\,|\,T-\lambda I_{L^2}\, \mbox{is\,not\,invertible}\}$; in particular, the set $\sigma_P(T)$ denotes the {\slshape point spectrum} of $T$, that is 
\begin{equation*}
\sigma_P(T)=\{\lambda\in\bC\,|\, \exists\,f\in L^2(\cG)\smallsetminus\{0\} \,\mbox{such\,that} \,Tf=\lambda f\},
\end{equation*}
such an $f$ is called {\slshape eigenfunction of $T$ associated to the eigenvalue $\lambda$}.\\

Our main result in this framework can be formulated as follows:\par
\vspace{0.5truecm}
\emph{If the symbol $a$ belongs to the modulation space $M^{p,\infty}(\cG\times\hcG)$ for some $0<p<\infty$, then any eigenfunction $f\in L^2(\cG)$ of the localization operator $\A$ satisfies
	\begin{equation*}
	f\in\bigcap_{\ga>0}M^\ga(\cG).
	\end{equation*}}
In particular, when $\cG=\zd$, this means that any eigenfunction  $f\in \ell^2(\zd)$ satisfies $f\in \bigcap_{\ga>0} \ell^\ga(\zd)$, so the sequence  $f$ displays a  fast  decay at infinity.
 
The study of eigenfunctions of $\A$ is pursued using the connection between localization and Kohn-Niremberg operators $\Opz(\sigma)$. 

Let us first introduce the  Rihaczek distribution.  Given $f,g\in L^2(\cG)$, we define the {\slshape (cross-)Rihaczek distribution} of $f$ and $g$ by
\begin{equation}\label{Rdef}
R(f,g)(x,\xi)=f(x)\overline{\hg(\xi)}\overline{\<\xi,x\>},\qquad(x,\xi)\in\cG\times\hcG,
\end{equation}
$\hg$ being the Fourier transform of $g$ \eqref{Eq-Def-Fourier-Transform}. When $f=g$, $R(f,f)$ is called the {\slshape Rihaczek distribution of $f$}.

Then the {\slshape pseudo-differential operator} $\Opz(\si)$ with Kohn-Nirenberg symbol $\si$ is formally defined by
	\begin{equation}\label{EqDefKN}
	(\Opz(\si) f)(x)=\int_{\hcG} \si(x,\xi)\hf(\xi)\<\xi,x\>\,d\xi.
	\end{equation}
	Equivalently, we can define it weakly by
	\begin{equation}\label{Eq-KonNirenberg-weak}
	\< \Opz(\sigma) f,g\>=\<\sigma,R(g,f)\>.
	\end{equation}
If  $\si\in M^\infty(\cG\times\hcG)$ then
\begin{equation}
\Opz(\si)\colon M^1(\cG)\to M^\infty(\cG)
\end{equation}
is well defined, linear and continuous, see e.g. \cite[Corollary 4.2, Theorem 5.3]{Jakobsen2018}.
 A localization operator  $\A$ can be written in the Kohn-Nirenberg form:
\begin{equation*}
\A=\Opz(a\ast R(\p_2,\p_1)),
\end{equation*}
that is, $\A$ is a  Kohn-Nirenberg operator with symbol
$$\sigma=a\ast R(\p_2,\p_1)$$
the convolution between the localization symbol $a$ and the cross-Rihaczek distribution $R(\p_2,\p_1)$ of its windows $\p_2,\p_1$.
It becomes then natural to study the properties of Kohn-Nirenberg pseudo-differential operators  and convolution relations for modulation spaces on LCA groups.

We obtain new boundedness results for such operators in modulation spaces and describe the decay of their eigenfunctions in $L^2(\cG)$, see Theorem \ref{Th-cont-KN-3-indici} and Proposition \ref{Pro-KN-eigenfunction} below.  The convolution properties are contained in Proposition \ref{Pro-convolution-Mpq}.

We point out that Theorem \ref{Th-cont-KN-3-indici} is not an easy generalization of the Euclidean case. It requires frame theory on quasi-lattices  and  proofs with  high level of technicalities, cf. Subsection 3.2 below. Quasi-lattices were used by Gr\"{o}chenig and Strohmer in \cite{GroStr2007} since not every $\cG$ admits a lattice, e.g. the $p$-adic groups $\mathbb{Q}_p$.\par 
They are the key issue in showing the boundedness properties for Kohn-Nirenberg operators in the subsequent Theorem \ref{Th-cont-KN-3-indici}, and we believe that these new techniques for Gabor frames on quasi-lattices  can be  valuable in and of  themselves and applied in other contexts.  Loosely  speaking, the main insight (suggested in \cite{GroStr2007}) is  ``to consider the quotient group", cf. Subsections 3.2 for details.
\par 	
The paper is organized as follows. In  Section 1 we establish technical assumptions and notations. Section 2 is devoted to the new general theory for modulation spaces $M^{p,q}_m(\cG)$ with $0<p,q\leq\infty$. In Section 3 we study continuity properties on modulations spaces for the Rihaczek distribution and pseudo-differential operators with Kohn-Nirenberg symbols. Gabor frame over quasi-lattices, analysis and synthesis operators, convolution relations are investigated as well. Section 4 deals with localization operators and their eigenfunctions. In the Appendix A we resume the coorbit theory presented in the thesis of Voigtlaender \cite{Voig2015} and  compare it with the one of Feichtinger and Gr\"{o}chenig. 
We strongly recommend  the reader who is not familiar with coorbit theory  to read the Appendix A,  for it is heavily used in Section 2 and subsequent sections. 
\section{Preliminaries}
We mainly follow the notations and assumptions of Gr\"{o}chenig and Strohmer \cite{GroStr2007}.
\subsection{Notations}
$\cG$ denotes a LCA group with the Hausdorff property. $\hcG$ is the dual group of $\cG$. The group operation on $\cG$, and on any abelian group such as $\cG\times\hcG$, is written additively. The unit in $\cG$ and $\hcG$ are denoted by $e$ and $\hat{e}$, respectively.\\
$\cG$ is assumed second countable, which is equivalent to $L^2(\cG)$ separable (see \cite[Theorem 2]{deVries1978}) and implies the metrizability of the group (\cite[Pag. 34]{MontZipp1955}).  In order to avoid uncountable sets of indexes and sums we  require $\cG$ to be $\si$-compact; this last property is equivalent to $\si$-finiteness  \cite[Proposition 2.22]{Folland_AHA1995}, as observed in \cite[Remark 2.3.2]{Voig2015}. Note that, due to \cite[Theorem 4.2.7]{ReiSte2000} and Pontrjagin's duality, $\cG$ is second countable and $\sigma$-compact if and only if $\hcG$ is second countable and $\sigma$-compact.

 In the sequel, $A\lesssim B$ means that there exists a constant $c>0$ independent of $A$ and $B$ such that $A\leq cB$; we write $A\asymp B$ if both $A\lesssim B$ and $B\lesssim A$. If $f\colon X\to \bC,x\mapsto f(x)$ and $g\colon Z\to\bC,z\mapsto g(z)$, then we define the {\slshape tensor product of $f$ and $g$} as $f\otimes g\colon X\times Z\to\bC,(x,z)\mapsto f(x)g(z)$. We denote by $X\hookrightarrow Z$ the continuous injection of $X$ into $Z$.

\subsection{Fundamental operators, special test functions, Rihaczek distribution}
We adopt the space of special test functions $\cS_\cC(\cG)$ introduced in \cite{GroStr2007} and defined below. The definition is based on
the structure theorem $\cG\cong\rd\times\cG_0$ \cite[Theorem 24.30]{HewRos63}, where $d\in\bN_0$ and $\cG_0$ is a LCA group containing a compact open subgroup $\cK$. Consequently, we can identify $\hcG$ with $\rd\times\hcG_0$, where the dual group $\hcG_0$ contains the compact open subgroup $\cK^\perp$, see e.g. \cite[Lemma 6.2.3]{Gro-AspectsGabor-1998}. We endow $\cG$ and $\hcG$ with the Haar measures $dx$ and $d\xi$, respectively, where $d\xi$ is the dual Haar measure.  The {\slshape Fourier transform} is
\begin{equation}\label{Eq-Def-Fourier-Transform}
	\cF f(\xi)=\hat{f}(\xi)=\int_\cG f(x)\overline{\la\xi,x\ra}\,dx,\qquad\xi\in\hcG.
\end{equation}
$\cF $ is an isometry from $L^2(\cG)$ onto $L^2(\hcG)$
.\\
On account of the  structure theorem above, we define the {\slshape generalized Gaussian on $\cG$} as 
\begin{equation}\label{Gauss}
\f(x_1, x_2)\coloneqq e^{-\pi x_1^2} \chi_\cK(x_2)\eqqcolon\f_1(x_1)\f_2(x_2),\quad (x_1, x_2) \in \rd\times \mathcal{G}_0,
\end{equation}  
and the {\slshape set of special test functions} 
\begin{equation}\label{Sc}
\mathcal{S_\mathcal{C}}(\mathcal{G})\coloneqq\mbox{span}\left\{\pi(\textbf{x})\f,\quad \textbf{x}=(x,\xi)\in \mathcal{G}\times\mathcal{\widehat{G}}\right\}\subseteq L^2(\cG),
\end{equation}
that is, the set of all time-frequency shifts of the Gaussian $\f=\f_1\otimes\f_2$ in \eqref{Gauss}. For the main properties of this space we refer to \cite[Section 2]{GroStr2007}.

For $x=(x_1,x_2)\in \rd\times \mathcal{G}_0$ and $\xi=(\xi_1,\xi_2)\in  \rd\times \widehat{\cG}_0$, the Rihaczek distribution of   $\f=\f_1\otimes\f_2$ in \eqref{Gauss} is given by 
\begin{align}
R(\f,\f)(x,\xi)&=R(\f_1,\f_1)(x_1,\xi_1)R(\f_2,\f_2)(x_2,\xi_2)\notag\\
&=e^{-2\pi i \xi_1x_1}e^{-\pi(x_1^2+\xi_1^2)}\chi_{\cK}(x_2)c(\cK)\chi_{{\cK}^\perp}(\xi_2)\overline{\la \xi_2,x_2\ra}\notag\\
&=c(\cK)e^{-2\pi i \xi_1x_1}e^{-\pi(x_1^2+\xi_1^2)}\overline{\la \xi_2,x_2\ra}\chi_{\cK\times\cK^\perp}(x_2,\xi_2)\notag\\
&=c(\cK)\overline{\la \xi,x\ra}e^{-\pi(x_1^2+\xi_1^2)}\otimes\chi_{\cK\times\cK^\perp}(x_2,\xi_2),\label{Eq-Rff}
\end{align}
where $c(\cK)>0$ is  a constant depending on the compact subgroup $\cK$. Hence $R(\f,\f)(x,\xi)$ is up to a positive constant and a `` chirp" a Gaussian on $\rdd \times (\cG_0\times \hcG_0)$, where we fixed $\cK\times\cK^\perp$ as compact open subgroup of the non Euclidean component. We recall the following covariance property \cite[Lemma 4.2 (i)]{GroStr2007}: for $\textbf{x}=(x,\xi),$ $\textbf{y}=(y,\eta)\in\cG\times\hcG$, $f,g\in \mathcal{S_\mathcal{C}}(\mathcal{G})$,
\begin{equation}\label{Rtfs}
R(\pi(\textbf{x})f,\pi(\textbf{y})g)=\la \eta, x-y\ra M_{\mathcal{J}(\textbf{y}-\textbf{x})} T_{(x,\eta)}R(f,g),
\end{equation}
where $\cJ$ is the topological isomorphism
\begin{equation}\label{DefJ}
\cJ\colon\cG\times\hcG\to\hcG\times\cG,\,(x,\xi)\mapsto (-\xi,x),
\end{equation}
and $\cJ^{-1}(\xi,x)=(x,-\xi)$. In what follows we shall need also the following identity: 
	\begin{equation}\label{Eq-SFTFgauss}
		V_\f \f(x,\xi)=c(\cK)e^{-\frac\pi2(x^2_1+\xi^2_1)}\otimes\chi_{\cK\times\cK^\perp}(x_2,\xi_2),
	\end{equation}
	see \cite{GroStr2007} for calculations.
Using a similar argument as in the estimate \cite[formula (12)]{GroStr2007}, one can show that  $R(f,g)$ and $V_g f$ are in $L^p_m(\cG\times\hcG)$, $0<p\leq\infty$, for arbitrary moderate weight functions, which will be defined in the Appendix A, and any $f,g\in\cS_\cC(\cG)$. Similarly, every function 
 in $\mathcal{S_\mathcal{C}}(\cG)$ belongs to $L^p_m(\cG)$, $0<p\leq\infty$.
Recall that for any $f,g\in L^2(\cG)$ \cite[formula (8)]{GroStr2007} 
 \begin{equation}\label{STFTtfsfts}
V_{M_\eta T_y g} M_\o T_u f(x,\xi)=\overline{\la \xi-\o, u\ra}\la \eta,x-u\ra T_{(u-y,\o-\eta)}V_gf(x,\xi).
 \end{equation}
The previous formula, jointly with \eqref{STFTtfsfts} and \eqref{Rtfs}, allows us to write explicitly every STFT  and cross-Rihaczek distribution of elements in $\cS_\cC(\cG)$.
\begin{lemma}\label{Lem-STFTinS_C}
	Consider $f,g\in\cS_\cC(\cG)$, hence
	\begin{equation*}
	f=\sum_{k=1}^{n}a_k\pi(\bfu_k)\f,\qquad g=\sum_{j=1}^{m}b_j\pi(\bfy_j)\f,
	\end{equation*}
	for some $n,m\in\bN$, $a_k,b_j\in\bC$ and $\bfu_k=(u_k,\o_k),\bfy_j=(y_j,\eta_j)\in\cG\times\hcG$. Then for every $(x,\xi)\in\cG\times\hcG$:
	\begin{align}
	V_gf(x,\xi)&=\sum_{k=1}^{n}\sum_{j=1}^{m}a_k\overline{b_j\<\xi-\o_k,u_k\>}\<\eta_j,x-u_k\>T_{\bfu_k-\bfy_j}V_\f\f(x,\xi),\\
	R(f,g)(x,\xi)&=\sum_{k=1}^{n}\sum_{j=1}^{m}a_k b_j \la \eta_j,u_k-y_j\ra M_{\cJ(\bfy_j-\bfu_k)}T_{(u_k,\eta_j)}R(\f,\f)(x,\xi).
	\end{align}
\end{lemma}
\begin{proof}
	We write $\bfx=(x,\xi)$. The first claim follows from \eqref{STFTtfsfts} after the following rephra{\Huge {\tiny }}sing:
	\begin{align*}
	V_gf(\bfx)&=\<f,\pi(\bfx)g\>=\<\sum_{k=1}^{n}a_k\pi(\bfu_k)\f,\pi(\bfx)\sum_{j=1}^{m}b_j\pi(\bfy_j)\f\>\\
	&=\sum_{k=1}^{n}\sum_{j=1}^{m}a_k\overline{b_j}\<\pi(\bfu_k)\f,\pi(\bfx)\pi(\bfy_j)\f\>=\sum_{k=1}^{n}\sum_{j=1}^{m}a_k\overline{b_j}\left(V_{\pi(\bfy_j)\f}\pi(\bfu_k)\f\right)(\bfx).
	\end{align*}
	For the second issue we  write
	\begin{align*}
	R(f,g)(\bfx)
	&=\sum_{k=1}^{n}\sum_{j=1}^{m}a_k b_j\pi(\bfu_k)\f(x)\widehat{\pi(\bfy_j)\f}(\xi)\overline{\la\xi,x\ra}\\
	&=\sum_{k=1}^{n}\sum_{j=1}^{m}a_k b_jR(\pi(\bfu_k)\f,{\tiny {\normalsize }}\pi(\bfy_j)\f)(\bfx)
	\end{align*}
	and use \eqref{Rtfs}.
\end{proof}

\section{Modulation spaces over a LCA group} A short survey of coorbit spaces on a locally compact Hausdorff (LCH) group with respect to a solid quasi-Banach function (QBF) space $Y$ (developed by Voigtlaender in his Ph.D. thesis \cite{Voig2015}) is contained in the Appendix A. In particular, see the Appendix A for the following concepts: left  $L_x$ and right  $R_x$ translations, relatively separated families, a discrete space $Y_d$ associated to $Y$, BUPUs, maximal functions $\sfM_Q f$, Wiener amalgam spaces $W_Q(Y)=W_Q(L^\infty,Y)$ and their right-sided version $W^R_Q(Y)=W^R_Q(L^\infty,Y)$.  Definition \ref{Def-weights} contains the hypothesis on weights and the class $\cM_v$ used in what follows. Note that the  coorbit space  construction is listed in items  \textbf{A} -- \textbf{J}  (unitary representation $\rho$, wavelet transform $W^\rho_g f$, assumptions on weights, sets $\bG_v$, $\bA^r_v$, $\cT_v$, $\cR_v$). Each of these items will be revisited in this section under specific choices, see  list \textbf{A$'$} --\textbf{J$'$} below. The Appendix A reports also some fundamental results of Voigtlaender \cite{Voig2015} and  a comparison with the earlier coorbit theory by Feichtinger and Gr\"{o}chenig.
\par
Relying on the theory  in  Appendix A, we are able to give a definition of modulation spaces on LCA groups which covers Feichtinger's orginal one \cite{feichtinger-modulation} and deals with the quasi-Banach case. The subsequent construction of $M^{p,q}_m(\cG)$ was suggested for the Banach case in \cite[p. 67]{FeiGro1988}, although the coorbit theory applied here is different. \par
Since the group $\bH_\cG$ defined below is noncommutative, we adopt the multiplicative notation for its operation.
\begin{definition}\label{defiHG}
	Let $\bT$ be the torus with the complex multiplication. 
	We define the {\slshape Heisenberg-type group associated to $\cG$}, Heisenberg group for short, as 
	\begin{equation}
		\bH_\cG\coloneqq\cG\times\hcG\times\bT,
	\end{equation}
	endowed with the product topology and the following operation:
	\begin{equation}\label{Eq-operationH}
		(x,\xi,\tau)(x',\xi',\tau')=\left(x+x',\xi+\xi',\tau\tau'\la\xi',x\ra\right),
	\end{equation}
	for $(x,\xi,\tau),(x',\xi',\tau')\in\bH_\cG$.
\end{definition}
The group $\bH_\cG$ is also called {\slshape Mackey obstruction group} of $\cG\times\hcG$, see \cite[Section 4]{Chr1996}, in particular Example 4.6.

\begin{lemma}\label{Lem-HeinsenbergGroup}
	The topological product space $\bH_\cG$ with the operation in \eqref{Eq-operationH} is a topological LCH, $\si$-compact, noncommutative, unimodular group with Haar measure the product measure $dxd\xi d\tau$, $dx$ and $d\xi$ being dual Haar measures on $\cG$ and $\hcG$ and $d\tau(\bT)=1$.
\end{lemma}
\begin{proof}
	Hausdorff property, local compactness, $\si$-compactness and noncommutativity are trivial. For the proof that $\bH_\cG$ is a topological unimodular group we refer to Theorem 3 in \cite{Igusa1972},  for the bi-invariance of $dxd\xi d\tau$  see \cite[p. 12]{Igusa1972} or, alternatively, \cite[Lemma 4.3]{Chr1996}.
\end{proof}
The identity in $\bH_\cG$ is $(e,\hat{e},1)$ and the inverse of an element $(x,\xi,\tau)$ is
\begin{equation*}
	(x,\xi,\tau)^{-1}=(-x,-\xi,\overline{\tau}\la\xi,x\ra).
\end{equation*}

\begin{lemma}\label{Lem-SchrodingerRep}
	The mapping
	\begin{equation}
		\vr\colon \bH_\cG\to\cU(L^2(\cG)),(x,\xi,\tau)\mapsto\tau M_\xi T_x
	\end{equation}
	is a unitary, strongly continuous, irreducible, integrable representation of $\bH_\cG$ on $L^2(\cG)$. We call $\vr$ {\slshape Schr\"odinger representation}.
\end{lemma}
\begin{proof}
	Well-posedness of $\vr$ is trivial, from the commutations relations \eqref{Eq-commutation-relations} it is straightforward to see that $\vr$ is a group homomorphism.
	Observe that 
	\begin{equation*}
	\pi\colon \cG\times\hcG\to\cU(L^2(\cG)),(x,\xi)\mapsto M_\xi T_x
	\end{equation*}
	is a {\slshape projective representation} in the terminology of \cite[Definition 4.1]{Chr1996}. In fact,  (i) $\pi(e,\hat{e})=I_{L^2}$; (ii) from the commutation relations \eqref{Eq-commutation-relations} we obtain
	\begin{equation*}
		\pi\left((x,\xi)+(x',\xi')\right)=\overline{\la\xi',x\ra}\pi(x,\xi)\pi(x',\xi'),
	\end{equation*}
	where $\la\cdot{,}\cdot\ra$ is continuous on $\cG\times\hcG$; (iii) the continuity of the STFT guarantees the required measurability. 
	To verify that $\vr$ is strongly continuous, one can proceed as in the Euclidean case, see e.g. \cite{CordRod2020}. The result then follows from \cite[Lemma 4.4 (ii)]{Chr1996}.\\
	The fact that $\vr$ is irreducible was proved in \cite{Igusa1972}, see page 14 before §5. For the integrability, consider the Gaussian $\vp\in L^2(\cG)$ in \eqref{Gauss} and observe that the torus is compact and $\abs{W^\vr_\vp \vp}=\abs{V_\vp \vp}$ (see \eqref{Eq-DefWaveletGenericRho} for the definition of $W^\vr_\f \f$). Then from \eqref{Eq-SFTFgauss} we have $V_\f \f\in L^1(\cG\times\hcG)$ and $W^\vr_\f \f\in L^1(\bH_\cG)$. This concludes the proof.
\end{proof}

\begin{definition}\label{Def-LpqHeisenberg}
	 We define the {\slshape extension of $m\in\cM_v(\cG\times\hcG)$} as 
	\begin{equation}\label{Eq-DefExtendedWeight}
		\tilde{m}\colon\bH_\cG\to(0,+\infty),(x,\xi,\tau)\mapsto m(x,\xi).
	\end{equation}
	For $0<p,q\leq\infty$, the space $L^{p,q}_{\tilde{m}}(\bH_\cG)$ consists of those equivalence classes of measurable complex-valued functions on $\bH_\cG$, where two functions are identified if they coincide a.e., for which the following application is finite
	\begin{equation}\label{Eq-Def-Norm-Lpq-Heis}
		\norm{F}_{L^{p,q}_{\tilde{m}}(\bH_\cG)}\coloneqq\norm{F}_{L^{p,q}_{\tilde{m}}}\coloneqq\left(\int_{\hcG\times\bT}\left(\int_\cG \abs{F(x,\xi,\tau)}^p m(x,\xi)^p\,dx\right)^{\frac{q}{p}}\,d\xi d\tau\right)^{\frac{1}{q}},
	\end{equation}
	obvious modifications for $p=\infty$ or $q=\infty$.
\end{definition}

$(L^{p,q}_{\tilde{m}}(\bH_\cG),\norm{\cdot}_{L^{p,q}_{\tilde{m}}(\bH_\cG)})$ is a solid QBF space on $\bH_\cG$. 
If $m$ is moderate with respect to a submultiplicative weight $v$ on $\cG\times\hcG$, then $\tilde{m}$ is left- and right-moderate w.r.t. $\tilde{v}$ on $\bH_\cG$, 
$\tilde{v}$ 
as in \eqref{Eq-DefExtendedWeight}. Therefore $L^{p,q}_{\tilde{m}}(\bH_\cG)$ is left and right invariant, see Definition \ref{Def-QBF-space-Y}.

\begin{lemma}\label{Lem-r-norm-Lpq}
	Consider $0<p,q\leq\infty$. Then $\norm{\cdot}_{L^{p,q}_{\tilde{m}}(\bH_\cG)}$ is an $r$-norm on $L^{p,q}_{\tilde{m}}(\bH_\cG)$ with $r\coloneqq\min\{1,p,q\}$.
\end{lemma}
\begin{proof}
	We present the proof for generic product measure space $X\times Y$, with product measure $d\mu(x)d\nu(y)$, instead of $\bH_\cG\cong\cG\times(\hcG\times\bT)$. We tackle the unweighted case, the weighted one follows immediately. We recall that for $0<p\leq\infty$ the application
	\begin{equation*}
		\norm{f}_{L^p(X)}\coloneqq \left(\int_X\abs{f(x)}^pd\mu(x)\right)^{\frac1p},
	\end{equation*}
	with obvious modification for $p=\infty$, is an $\min\{1,p\}$-norm, see e.g. \cite[Example 2.1.3]{Voig2015}. Therefore it is a $\min\{1,p,q\}$-norm also. Let us consider $f,g\in L^{p,q}(X\times Y)$ and $r\coloneqq\min\{1,p,q\}$, using the fact that $\norm{\cdot}_{L^p(X)}$ is an $r$-norm and $q/r\geq1$:
	\begin{align*}
		\norm{f+g}^r_{L^{p,q}(X\times Y)}&=\left(\int_{Y}\left(\int_{X}\abs{f(x,y)+g(x,y)}^pd\mu(x)\right)^\frac{q}{p}d\nu(y)\right)^{\frac{r}{q}}\\
		&=\left(\int_{Y}\left(\int_{X}\abs{f(x,y)+g(x,y)}^pd\mu(x)\right)^{\frac{r}{p}\frac{q}{r}}d\nu(y)\right)^{\frac{r}{q}}\\
		&\leq \left(\int_{Y}\left(\left(\int_{X}\abs{f(x,y)}^p d\mu(x)\right)^\frac{r}{p}+\left(\int_{X}\abs{g(x,y)}^p d\mu(x)\right)^\frac{r}{p}\right)^{\frac{q}{r}}d\nu(y)\right)^{\frac{r}{q}}\\
		&\leq \left(\int_{Y}\left(\left(\int_{X}\abs{f(x,y)}^p d\mu(x)\right)^\frac{r}{p}\right)^{\frac{q}{r}}d\nu(y)\right)^{\frac{r}{q}}\\
		&+\left(\int_{Y}\left(\left(\int_{X}\abs{g(x,y)}^p d\mu(x)\right)^\frac{r}{p}\right)^{\frac{q}{r}}d\nu(y)\right)^{\frac{r}{q}}\\
		&=\norm{f}^r_{L^{p,q}(X\times Y)}+\norm{g}^r_{L^{p,q}(X\times Y)}.
	\end{align*}
	The proof is concluded.
\end{proof}

\begin{lemma}\label{Lem-Bi-InvarianceLpq}
	Consider $m\in\cM_v(\cG\times\hcG)$ and $0<p,q\leq\infty$. Then there exists $C=C(m,v)>0$ such that for any $F\in L^{p,q}_{\tilde{m}}(\bH_\cG)$ and $(x,\xi,\tau)\in\bH_\cG$ 
	\begin{equation}\label{Eq-Bi-InvarianceLpq}
		\norm{R_{(x,\xi,\tau)}F}_{L^{p,q}_{\tilde{m}}}\leq Cv(-x,-\xi)\norm{F}_{L^{p,q}_{\tilde{m}}},\quad\norm{L_{(x,\xi,\tau)}F}_{L^{p,q}_{\tilde{m}}}\leq Cv(x,\xi)\norm{F}_{L^{p,q}_{\tilde{m}}}.
	\end{equation}
\end{lemma}
\begin{proof}
	The claim is a straightforward calculation which follows by the bi-invariance of the Haar measure on $\bH_\cG$. For $p,q\neq\infty$,
	\begin{align*}
		\norm{R_{(x,\xi,\tau)}F}^q_{L^{p,q}_{\tilde{m}}}&=\int_{\hcG\times\bT}\left(\int_\cG\abs{F((u,\o,t)(x,\xi,\tau))}^p\tilde{m}(u,\o,t)^p\,du\right)^{\frac{q}{p}}\,d\o dt\\
		&\underset{m,v}{\lesssim}\int_{\hcG\times\bT}\left(\int_\cG\abs{F(u',\o',t')}^p\tilde{m}(u',\o',t')^p\tilde{v}((x,\xi,\tau)^{-1})^p\,du'\right)^{\frac{q}{p}}\,d\o' dt'\\
		&=v(-x,-\xi)^q\norm{F}^q_{L^{p,q}_{\tilde{m}}}.
	\end{align*}
	Left translations are treated  similarly, as well as the cases 
	$p=\infty$ or $q=\infty$.
\end{proof}
Due to the symmetry of $v$ (Definition \ref{Def-weights}), the first inequality in \eqref{Eq-Bi-InvarianceLpq} reads as
\begin{equation*}
		\norm{R_{(x,\xi,\tau)}F}_{L^{p,q}_{\tilde{m}}}\leq Cv(x,\xi)\norm{F}_{L^{p,q}_{\tilde{m}}}.
\end{equation*}

\begin{lemma}\label{Lem-Bounded-R-Wiener}
	Let $0<p,q\leq\infty$. Fix $V_\cG\subseteq\cG$ and $V_{\hcG}\subseteq\hcG$ open, relatively compact, neighbourhoods of $e\in\cG$ and $\hat{e}\in\hcG$, respectively. Define
	\begin{equation}\label{Eq-V-open-neigh}
		V\coloneqq V_\cG\times V_{\hcG}\times\bT.
	\end{equation}
	 Consider $m\in\cM_v(\cG\times\hcG)$. Then there exists $C=C(m,v)>0$ such that for every $(x,\xi,\tau)\in\bH_\cG$
	\begin{equation}\label{Eq-Boundedness-R-W(Lpq)}
		\opnorm{R_{(x,\xi,\tau)}}_{W_V(L^{p,q}_{\tilde{m}})\to W_V(L^{p,q}_{\tilde{m}})}\leq C v(-x,-\xi).
	\end{equation}
\end{lemma}
\begin{proof}
	$V$ is a open, relatively compact, unit neighbourhood and the set
	\begin{equation}\label{Eq-Decomposition-V}
		V_{1,2}\coloneqq V_\cG\times V_{\hcG}
	\end{equation}
	is also open, relatively compact, unit neighbourhood in $\cG\times\hcG$. For $F\in L^\infty_{loc}(\bH_\cG)$ 
	\begin{equation}
		\sfM_V[R_{(x,\xi,\tau)}F]=\sfM_{V(x,\xi,\tau)}F,
	\end{equation}
	see \cite[Lemma 2.3.18, 1.]{Voig2015}. For any $(x,\xi,\tau)\in\bH_\cG$ 
	\begin{align*}
		V(x,\xi,\tau)
		=\left(V_\cG+x\right)\times\left(V_{\hcG}+\xi\right)\times\underset{u\in V_{\cG}}{\bigcup}\bT \tau \la\xi,u\ra
		=(x,\xi,\tau)V.
	\end{align*}
	If $F\in W_V(L^{p,q}_{\tilde{m}})$ and $(x,\xi,\tau)\in\bH_\cG$, from what just observed we obtain:
	\begin{align*}
		\sfM_V[R_{(x,\xi,\tau)}F](u,\o,t)
		&=\underset{(y,\eta,s)\in(u,\o,t)(x,\xi,\tau)V}{\essupp}\abs{F(y,\eta,s)}
		=R_{(x,\xi,\tau)}[\sfM_V F(u,\o,t)].
	\end{align*}
	Eventually by using \eqref{Eq-Bi-InvarianceLpq}
	\begin{align*}
		\norm{R_{(x,\xi,\tau)}F}_{W_V(L^{p,q}_{\tilde{m}})}&=\norm{\sfM_V[R_{(x,\xi,\tau)}F]}_{L^{p,q}_{\tilde{m}}}=\norm{R_{(x,\xi,\tau)}[\sfM_V F]}_{L^{p,q}_{\tilde{m}}}\\
		&\leq Cv(-x,-\xi)\norm{\sfM_{V}F}_{L^{p,q}_{\tilde{m}}}=Cv(-x,-\xi)\norm{F}_{W_V(L^{p,q}_{\tilde{m}})},
	\end{align*}
	for some $C=C(m,v)>0$. This concludes the proof.
\end{proof}
As already highlighted, inequality \eqref{Eq-Boundedness-R-W(Lpq)} can be equivalently written with $v(x,\xi)$ in place of $v(-x,-\xi)$. Observe that the constant $C$ involved in \eqref{Eq-Bi-InvarianceLpq} and \eqref{Eq-Boundedness-R-W(Lpq)} is the one coming from the $v$-moderateness condition: $m((x,\xi)+(u,\o))\leq C v(x,\xi)m(u,\o)$.
\begin{corollary}\label{Cor-Bounded-R-Wiener}
	Let $0<p,q\leq\infty$. Consider $Q\subseteq\bH_\cG$ measurable, relatively compact, unit neighbourhood and $m\in\cM_v(\cG\times\hcG)$. Then there exists $C_Q=C(Q,m,v)>0$ such that for every $(x,\xi,\tau)\in\bH_\cG$
	\begin{equation}\label{Eq-Boundedness-R-W(Lpq)-Cor}
	\opnorm{R_{(x,\xi,\tau)}}_{W_Q(L^{p,q}_{\tilde{m}})\to W_Q(L^{p,q}_{\tilde{m}})}\leq C_Q v(-x,-\xi).
	\end{equation}
\end{corollary}
\begin{proof}
	The claim follows from the independence of the Wiener Amalgam space $W(L^{p,q}_{\tilde{m}})$ from the window subset (Lemma \ref{Lem-WienerSpaceIndependence}) together with Lemma \ref{Lem-Bounded-R-Wiener}.
\end{proof}

\begin{remark}\label{Rem-WaveletSchrodinger}
	Consider the {\slshape (generalized) wavelet transform induced by  the Schr\"odinger representation $\vr$} in \eqref{Eq-DefWaveletGenericRho} taking $G=\bH_\cG$ and $f,g\in\cH=L^2(\cG)$:
	\begin{equation}\label{Eq-DefWaveletSchrodinger}
			W^\vr_g f\colon \bH_\cG\to\bC, (x,\xi,\tau)\mapsto\la f,\tau M_\xi T_x g\ra_{L^2(\cG)}.
	\end{equation}
	This is a continuous and bounded function. It is straightforward to see that
	\begin{equation}
		W^\vr_g f(x,\xi,\tau)=\la f,\tau M_\xi T_x g\ra=\overline{\tau}V_g f(x,\xi),\quad\forall\,(x,\xi,\tau)\in\bH_\cG,
	\end{equation}
	which implies
	\begin{equation}
		\abs{W^\vr_g f(x,\xi,\tau)}=\abs{V_g f(x,\xi)},\quad\forall\, (x,\xi,\tau)\in\bH_\cG.
	\end{equation}
	Therefore for $f,g\in L^2(\cG)$, being $\bT$ compact,
	\begin{equation}
		W^\vr_g f\in L^{p,q}_{\tilde{m}}(\bH_\cG)\quad\Leftrightarrow\quad V_g f\in L^{p,q}_m(\cG\times\hcG)
	\end{equation}
	and
	\begin{equation}
		W^\vr_g f\in W(L^\infty(\bH_\cG),L^{p,q}_{\tilde{m}}(\bH_\cG))\quad\Leftrightarrow\quad V_g f\in W(L^\infty(\cG\times\hcG),L^{p,q}_m(\cG\times\hcG)).
	\end{equation}
\end{remark}

We are now able to revisit steps \textbf{A} -- \textbf{J} in the Appendix A, as follows.
\begin{itemize}
	\item[\textbf{A$'$}.] For $G=\bH_\cG$ the Heisenberg group associated to $\cG$, $\cH=L^2(\cG)$ and $\rho=\vr\colon\bH_\cG\to L^2(\cG)$ the Schr\"odinger representation, the requirements of \textbf{A} are fulfilled due to Lemma \ref{Lem-HeinsenbergGroup} and \ref{Lem-SchrodingerRep}.
	
	\item[\textbf{B$'$}.] $W^\vr_g f$ was described in \eqref{Eq-DefWaveletSchrodinger} and the integrability of $\vr$ was proved in Lemma \ref{Lem-SchrodingerRep} as well as that every element of $\cS_\cC(\cG)$ is admissible for $\vr$.
	
	\item[\textbf{C$'$}.] Take $Y=L^{p,q}_{\tilde{m}}(\bH_\cG)$ (Definition \ref{Def-LpqHeisenberg}) and $r=\min\{1,p,q\}$ (Lemma \ref{Lem-r-norm-Lpq}).
\end{itemize}

\begin{itemize}	
	\item[\textbf{D$'$}.] The right invariance for each measurable, relatively compact, unit neighbourhood $Q\subseteq \bH_\cG$ of $W_Q(L^\infty,L^{p,q}_{\tilde{m}})$ is guaranteed by the right invariance of $L^{p,q}_{\tilde{m}}(\bH_\cG)$, Lemma \ref{Lem-Bi-InvarianceLpq} and Lemma \ref{Lem-WienerSpaceIndependence}. Since $\bH_\cG$ is unimodular, \eqref{Eq-Condition-D-1} and \eqref{Eq-Condition-D_2} can be summarized as
	\begin{equation}
		w(x,\xi,\tau)\underset{Q}{\gtrsim}\opnorm{R_{(x,\xi,\tau)^{\pm 1}}}_{W_Q(L^{p,q}_{\tilde{m}})\to W_Q(L^{p,q}_{\tilde{m}})},
	\end{equation}
	for some (hence every) measurable, relatively compact, unit neighbourhood $Q\subseteq \bH_\cG$. Therefore, on account of \eqref{Eq-Boundedness-R-W(Lpq)-Cor} and the definition of $\cM_v(\cG\times\hcG)$, we can take $w=\tilde{v}$ the extension of $v$ defined as in \eqref{Eq-DefExtendedWeight}.
	\item[\textbf{E$'$}.] We take $\tilde{v}$ as control weight for $L^{p,q}_{\tilde{m}}(\bH_\cG)$, see \textbf{E}.
	
	\item[\textbf{F$'$}.] The class of good vectors we are considering is
	\begin{equation}
		\bG_{\tilde{v}}\coloneqq\left\{g\in L^2(\cG)\,|\,W^\vr_g g\in L^1_{\tilde{v}}\right\}.
	\end{equation}
	We shall prove that it is nontrivial.
	
	\item[\textbf{G$'$}.] Our class of analysing vectors is
	\begin{equation}
		\bA^r_{\tilde{v}}\coloneqq\left\{g\in L^2(\cG)\,|\, W^\vr_g g\in W^R(L^\infty,W(L^\infty,L^r_{\tilde{v}}))\right\}.
	\end{equation}		
\end{itemize} 

It is due to \cite[Lemma 2.4.9]{Voig2015} that $\bA^r_{\tilde{v}}$ is a vector space, as observed in the proof of \cite[Theorem 2.4.9]{Voig2015}, and that 
\begin{equation}\label{Eq-remark-bA^r_v}
	W^\vr_h g\in W^R(L^\infty,W(L^\infty,L^r_{\tilde{v}}))
\end{equation}
for every $g,h\in \bA^r_{\tilde{v}}$.

\begin{lemma}\label{Lem-SC-window-space}
	Let us define
	\begin{equation}\label{Eq-maximal-window-space-ALL-modulation-spaces}
	\sA_{\tilde{v}}=\sA_{\tilde{v}}(\cG)=\bigcap_{0<r\leq 1}\bA^r_{\tilde{v}}.
	\end{equation}
	The following inclusions hold true:
	\begin{equation}
	\cS_\cC(\cG)\subseteq \sA_{\tilde{v}}\subseteq\bG_{\tilde{v}}.
	\end{equation}
\end{lemma}
\begin{proof}
	The only inclusion to be shown is the first one, the second one was already mentioned in Remark \ref{Rem-A-subset-G} (ii). Fix $0<r\leq 1$. First, we show that the Gaussian $\f\in L^2(\cG)$  in \eqref{Gauss} belongs to $\bA^r_{\tilde{v}}$. From \eqref{Eq-SFTFgauss}:
	\begin{equation*}
		W^\vr_\f \f(x,\xi,\tau)=\overline{\tau}c(\cK)e^{-\frac\pi2(x^2_1+\xi^2_1)}\otimes\chi_{\cK\times\cK^\perp}(x_2,\xi_2)=\overline{\tau}V_\f \f(x,\xi),
	\end{equation*}
	for some $c(\cK)>0$. Take $V\subseteq \bH_\cG$ as in \eqref{Eq-V-open-neigh} and observe that if $F\in L^\infty_{loc}(\bH_\cG)$
	\begin{align*}
		\sfM_V[\sfM^R_V F](x,\xi,\tau)
		=\underset{(u,\o,t)\in(x,\xi,\tau)V}{\essupp}\abs{\underset{(y,\eta,s)\in V(u,\o,t)}{\essupp}\abs{F(y,\eta,s)}}
		\leq
		\underset{(y,\eta,s)\in V(x,\xi,\tau)V}{\essupp}\abs{F(y,\eta,s)}.
	\end{align*}
	If $F=W^\vr_\f \f$, adopting notation of \eqref{Eq-Decomposition-V}, we get
	\begin{align*}
		\sfM_V[\sfM^R_V W^\vr_\f \f](x,\xi,\tau)&\leq\underset{(y,\eta,s)\in V(x,\xi,\tau)V}{\essupp}\abs{\overline{s}V_\f \f(y,\eta)}=\underset{(y,\eta)\in V_{1,2}+(x,\xi)+V_{1,2}}{\essupp}\abs{V_\f \f(y,\eta)}\\
		&=\underset{(y,\eta)\in (x,\xi)+2V_{1,2}}{\essupp}\abs{V_\f \f(y,\eta)}=\sfM_{2V_{1,2}}V_\f \f(x,\xi),
	\end{align*}
	where $2V_{1,2}\coloneqq V_{1,2}+V_{1,2}$ is a open, relatively compact, unit neighbourhood in $\cG\times\hcG$.\\
	From the solidity of $L^r_{\tilde{v}}$,
	\begin{equation}\label{Eq-WienerNormWavelet}
		\norm{W^\vr_\f \f}_{W^R(W(L^r_{\tilde{v}}))}\asymp\norm{W^\vr_\f \f}_{W^R_V(W_V(L^r_{\tilde{v}}))}\leq\norm{\sfM_{2V_{1,2}}V_\f \f}_{L^r_{v}(\cG\times\hcG)}
	\end{equation}
	and we shall prove the right-hand side to be finite. Due to the arbitrariness of $V_\cG$ and $V_{\hcG}$, we can assume that
	\begin{equation}
		V_{1,2}=V_\cG\times V_{\hcG}\cong \left(E_1\times D_1 \right)\times\left(E_2\times D_2\right)\cong\left(E_1\times E_2\right)\times\left(D_1\times D_2\right),
	\end{equation}
	 where $E_1,E_2\subseteq \rd$, $D_1\subseteq \cG_0$ and $D_2\subseteq \hcG$ are open, relatively compact, unit neighbourhoods. As done previously,
	\begin{align*}
		E_{1,2}\coloneqq E_1\times E_2\subseteq\rdd,&\quad D_{1,2}\coloneqq D_1\times D_2\subseteq\cG_0\times\hcG_0,\\
		2E_{1,2}\coloneqq E_{1,2}+E_{1,2},&\quad 2D_{1,2}\coloneqq D_{1,2}+D_{1,2}.
	\end{align*}
	Hence
	\begin{align*}
		\sfM_{2V_{1,2}}V_\f \f(x,\xi)&=c(\cK)\underset{\substack{((y_1,\eta_1),(y_2,\eta_2))\in\\((x_1,\xi_1),(x_2,\xi_2))+2E_{1,2}\times 2D_{1,2}}}{\essupp}\abs{e^{-\frac\pi2(y^2_1+\eta^2_1)}\chi_{\cK\times\cK^\perp}(y_2,\eta_2)}\\
		&=c(\cK)\underset{(y_1,\eta_1)\in(x_1,\xi_1)+2E_{1,2}}{\essupp}\abs{e^{-\frac\pi2(y^2_1+\eta^2_1)}}\underset{(y_2,\eta_2)\in(x_2,\xi_2)+2D_{1,2}}{\essupp}\abs{\chi_{\cK\times\cK^\perp}(y_2,\eta_2)}.
	\end{align*}
	Since $v(x,\xi)$ is submultiplicative, using the structure theorem we can majorize as follows:
	\begin{equation*}
		v(x,\xi)=v((x_1,\xi_1),(x_2,\xi_2))\leq v((x_1,\xi_1),(e_0,\hat{e}_0))v((0,0),(x_2,\xi_2)),
	\end{equation*} 
	where $x=(x_1,x_2)\in\rd\times\cG_0,\,\xi=(\xi_1,\xi_2)\in\rd\times\hcG_0$. Let us define
	\begin{equation*}
		v_1(x_1,\xi_1)\coloneqq v((x_1,\xi_1),(e_0,\hat{e}_0)),\quad v_2(x_2,\xi_2)\coloneqq v((0,0),(x_2,\xi_2)),
	\end{equation*}
	$(x_1,\xi_1)\in\rdd$ and $(x_2,\xi_2)\in\cG_0\times\hcG_0$, which are still submultiplicative. Hence
	\begin{align*}
		\norm{\sfM_{2V_{1,2}}V_\f \f}_{L^r_{v}(\cG\times\hcG)}^r\leq&c(\cK)^r\underbrace{\int_\rdd\underset{(y_1,\eta_1)\in(x_1,\xi_1)+2E_{1,2}}{\essupp}\abs{e^{-\frac\pi2(y^2_1+\eta^2_1)}}^r v_1(x_1,\xi_1)^r\,dx_1d\xi_1}_{\eqqcolon I_1}\\
		&\times\underbrace{\int_{\cG_0\times\hcG_0}\underset{(y_2,\eta_2)\in(x_2,\xi_2)+2D_{1,2}}{\essupp}\abs{\chi_{\cK\times\cK^\perp}(y_2,\eta_2)}v_2(x_2,\xi_2)^r\,dx_2d\xi_2}_{\eqqcolon I_2}.
	\end{align*}
		For $N>2d$ and considering the weight $\la \cdot\ra\coloneqq(1+|\cdot|^2)^{1/2}$, we can write 
	\begin{align*}
	I_1&=\int_\rdd\frac{\la (x_1,\xi_1)\ra^N }{\la (x_1,\xi_1)\ra^N}\underset{(y_1,\eta_1)\in(x_1,\xi_1)+2E_{1,2}}{\essupp}\abs{e^{-\frac\pi2(y^2_1+\eta^2_1)}}^r v_1(x_1,\xi_1)^r\,dx_1d\xi_1\\
	&\leq \int_\rdd\frac{1 }{\la (x_1,\xi_1)\ra^N} \underset{(y_1,\eta_1)\in(x_1,\xi_1)+2E_{1,2}} {\essupp} \left[ e^{-\frac {r\pi}{2}(y^2_1+\eta^2_1)} v_1(y_1,\eta_1)^r\la (y_1,\eta_1)\ra^N\right]\,dx_1d\xi_1\\
	&\leq \int_\rdd\frac{1 }{\la (x_1,\xi_1)\ra^N} \underset{(y_1,\eta_1)\in\rdd}{\essupp} \left[ e^{-\frac {r\pi}{2}(y^2_1+\eta^2_1)} v_1(y_1,\eta_1)^r\la (y_1,\eta_1)\ra^N\right]\,dx_1d\xi_1.
	\end{align*}
	In fact,	\begin{align*} \underset{(y_1,\eta_1)\in(x_1,\xi_1)+2E_{1,2}}{\essupp} & e^{-\frac {r\pi}{2}(y^2_1+\eta^2_1)}\la (x_1,\xi_1)\ra^N v_1(x_1,\xi_1)^r\\ &\leq \underset{(y_1,\eta_1)\in(x_1,\xi_1)+2E_{1,2}}{\essupp} \left[ e^{-\frac {r\pi}{2}(y^2_1+\eta^2_1)} v_1(y_1,\eta_1)^r\la (y_1,\eta_1)\ra^N\right]\\
	&\leq \|e^{\frac{-r \pi}{2}|\cdot|^2 }v_1^r(\cdot)\la \cdot \ra^N\|_{L^\infty(\rdd)}<+\infty
	\end{align*}
	because  $v_1$ is submultiplicative so it can grow at most exponentially \cite[Lemma 2.1.4]{CordRod2020}. Hence  $I_1<+\infty$.\\
	We now study the integral $I_2$. Observe that the integrand is not equal to zero if and only if $(\cK\times\cK^\perp)\cap((x_2,\xi_2)+2D_{1,2})\neq\varnothing$, which means that there exist 
	$k\in\cK\times\cK^\perp$ and $h\in 2D_{1,2}$, all depending on $(x_2,\xi_2)$, such that
$	k=(x_2,\xi_2)+h$ if and only if $(x_2,\xi_2)= k-h$, which implies $(x_2,\xi_2)\in\cK\times\cK^\perp-2D_{1,2}.$
Equivalently, $(x_2,\xi_2)\notin\cK\times\cK^\perp-2D_{1,2}$ if and oly if $\underset{(y_2,\eta_2)\in(x_2,\xi_2)+2D_{1,2}}{\essupp}\abs{\chi_{\cK\times\cK^\perp}(y_2,\eta_2)}=0$, 	that implies
	\begin{equation*}
		\underset{(y_2,\eta_2)\in(x_2,\xi_2)+2D_{1,2}}{\essupp}\abs{\chi_{\cK\times\cK^\perp}(y_2,\eta_2)}\leq \chi_{\cK\times\cK^\perp-2D_{1,2}}(x_2,\xi_2).
	\end{equation*}
	Note that $\cK\times\cK^\perp-2D_{1,2}$ is  relatively compact, hence of finite measure. The local boundedness of the submultiplicative weight $v_2$, shown in \cite[Theorem 2.2.22]{Voig2015}, ensures that the integral on $\cG_0\times\hcG_0$ is finite.\\
	So far we have shown $\f\in\bA^r_{\tilde{v}}$. We now consider $f=\sum^n_{k=1}a_k\pi(u_k,\o_k)\f\in\cS_\cC(\cG)$ and apply \eqref{Eq-WienerNormWavelet}, Lemma \ref{Lem-STFTinS_C} and left/right invariance of $W_{2V_{1,2}}(L^r_{v}(\cG\times\hcG))$:
	\begin{align*}
		\norm{W^\vr_f f}_{W^R_V(W_V(L^r_{\tilde{v}}))}&\leq\norm{\sfM_{2V_{1,2}}V_f f}_{L^r_{v}(\cG\times\hcG)}=\norm{V_f f}_{W_{2V_{1,2}}(L^r_{v}(\cG\times\hcG))}\\
		&=\norm{\sum_{k,j=1}^{n}a_k\overline{a_j\<\xi-\o_k,u_k\>}\<\o_j,x-u_k\>T_{(u_k,\o_k)-(u_j,\o_j)}V_\f\f(x,\xi)}_{W_{2V_{1,2}}(L^r_{v}(\cG\times\hcG))}\\
		&\underset{n,r}{\lesssim}\sum_{k,j=1}^{n}\abs{a_k a_j}\norm{T_{(u_k,\o_k)-(u_j,\o_j)}V_\f\f(x,\xi)}_{W_{2V_{1,2}}(L^r_{v}(\cG\times\hcG))}<+\infty.
	\end{align*}
	This concludes the proof.
\end{proof}

Of course, $\sA_{\tilde{v}}$ is a vector space. We shall use the extended notation $\sA_{\tilde{v}}(\cG)$ only when confusion may occur. It is also clear that writing $\sA_{\tilde{v}}(\cG\times\hcG)$ we mean the weight $v$ to be defined on $(\cG\times\hcG)\times(\hcG\times\cG)$, as done in the subsequent Corollary \ref{Cor-Rfg-windows}.

\begin{corollary}\label{Cor-Rfg-windows}
	Let $f,g\in\cS_\cC(\cG)$, then $R(f,g)\in\sA_{\tilde{v}}(\cG\times\hcG)$.
\end{corollary}
\begin{proof}
	The proof follows the same arguments  in Lemma \ref{Lem-SC-window-space},  together with \eqref{Eq-Rff} and Lemma \ref{Lem-STFTinS_C}.
\end{proof}

\begin{itemize}
	\item[\textbf{H$'$}.] For a fixed $g\in\bG_{\tilde{v}}\smallsetminus\{0\}$, the space of test vectors is 
	\begin{equation}
	\cT_{\tilde{v}}\coloneqq \left\{f\in L^2(\cG)\,|\,W^\vr_g f\in L^1_{\tilde{v}}(\bH_\cG)\right\}
	\end{equation}
	endowed with the norm
	\begin{equation}
	\norm{f}_{\cT_{\tilde{v}}}\coloneqq\norm{W^\vr_g f}_{L^1_{\tilde{v}}}.
	\end{equation}
\end{itemize}
$(\cT_{\tilde{v}},\norm{\cdot}_{\cT_{\tilde{v}}})$ is a $\vr$-invariant Banach space which embeds continuously into $L^2(\cG)$ and it is independent from the choice of the window vector $g\in\bG_{\tilde{v}}\smallsetminus\{0\}$, see \cite[Lemma 2.4.7]{Voig2015}.

\begin{lemma} For any $g\in \cS_\cC(\cG)\smallsetminus\{0\}$, the following equality holds true
	\begin{equation}
	\bG_{\tilde{v}}=\cT_{\tilde{v}}=\left\{f\in L^2(\cG)\,|\,V_g f\in L^1_{v}(\cG\times\hcG)\right\}.
	\end{equation}
\end{lemma}
\begin{proof}
	The second equality is just Remark \ref{Rem-WaveletSchrodinger}, for the first one the proof follows the pattern of \cite[Proposition 3.6]{Berge2021}.
	From \cite[Lemma 2.4.7]{Voig2015}: $\bG_{\tilde{v}}\subseteq \cT_{\tilde{v}}$. Being the Duflo-Moore operator (\cite[Theorem 3]{DufloMoore1976}) the identity, the orthogonality relations for $f,h\in L^2(\cG)$ and $g,\ga\in\bG_{\tilde{v}}$ are
	\begin{equation*}
		\la W^\vr_g f,W^\vr_\ga h\ra_{L^2(\bH_\cG)}=\la \ga,g\ra_{L^2(\cG)}\la f,h\ra_{L^2(\cG)},
	\end{equation*}
	see \cite[Theorem 2.4.3]{Voig2015}. 
	 Fix $f\in\cT_{\tilde{v}}$, take $\ga=g\neq0$, $h=\vr(x,\xi,\tau)f$ and using Fubini's Theorem, symmetry and submultiplicativity of $\tilde{v}$ we compute
	\begin{align*}
		&\norm{W^\vr_f f}_{L^1_{\tilde{v}}}
		=\int_{\bH_\cG}\abs{\la f, \vr(x,\xi,\tau)f\ra}\tilde{v}(x,\xi,\tau)\,dxd\xi d\tau\\
		&=\frac{1}{\norm{g}^2_{L^2}}\int_{\bH_\cG}\abs{\la W^\vr_g f,W^\vr_ g [\vr(x,\xi,\tau)f]\ra}\tilde{v}(x,\xi,\tau)\,dxd\xi d\tau\\
		&\leq\frac{1}{\norm{g}^2_{L^2}}\int_{\bH_\cG}\int_{\bH_\cG}\left| W^\vr_g f(y,\eta,s)W^\vr_g[\vr(x,\xi,\tau)f](y,\eta,s)\right|\, dyd\eta ds\,\tilde{v}(x,\xi,\tau)\,dxd\xi d\tau\\
		&=\frac{1}{\norm{g}^2_{L^2}}\int_{\bH_\cG}\left|W^\vr_g f(y,\eta,s)\right|\left(\int_{\bH_\cG}\left| W^\vr_g[\vr(x,\xi,\tau)f](y,\eta,s)\right|\tilde{v}(x,\xi,\tau)\,dxd\xi d\tau\right) dyd\eta ds.
	\end{align*}
	Observe
	\begin{equation*}
	W^\vr_g[\vr(x,\xi,\tau)f](y,\eta,s)=\la\vr(x,\xi,\tau)f,\vr(y,\eta,s)g\ra=W^\vr_g f\left((x,\xi,\tau)^{-1}(y,\eta,s)\right),
	\end{equation*}
	so that
	\begin{align*}	
		\norm{W^\vr_f f}_{L^1_{\tilde{v}}}&\leq\frac{1}{\norm{g}^2_{L^2}}\int_{\bH_\cG}\left|W^\vr_g f(y,\eta,s)\right|\\
		&\qquad\times\left(\int_{\bH_\cG}\left| W^\vr_g f\left((x,\xi,\tau)^{-1}(y,\eta,s)\right)\right|\tilde{v}(x,\xi,\tau)\,dxd\xi d\tau\right) dyd\eta ds\\
		&\leq\frac{1}{\norm{g}^2_{L^2}}\int_{\bH_\cG}\left|W^\vr_g f(y,\eta,s)\right|\\
		&\qquad\times\left(\int_{\bH_\cG}\left| W^\vr_g f(x',\xi',\tau')\right|\tilde{v}(x',\xi',\tau')\tilde{v}(y,\eta,s)\,dx'd\xi'd\tau'\right) dyd\eta ds\\
		&=\frac{1}{\norm{g}^2_{L^2}}\left(\int_{\bH_\cG}\left| W^\vr_g f(x',\xi',\tau')\right|\tilde{v}(x',\xi',\tau')\,dx'd\xi'd\tau'\right)^2\\
		&=\frac{1}{\norm{g}^2_{L^2}}\norm{ W^\vr_g f}^2_{L^1_{\tilde{v}}}<+\infty.
	\end{align*}
	Hence $f\in\bG_{\tilde{v}}$ and the proof in concluded.
\end{proof}

\begin{lemma}\label{Lem-Sc-dense-Tv}
		$\cS_\cC(\cG)$ is dense in $(\cT_{\tilde{v}},\norm{\cdot}_{\cT_{\tilde{v}}})$.
\end{lemma}
\begin{proof}
	In Lemma \ref{Lem-SC-window-space} we have shown that the Gaussian $\f$ in \eqref{Gauss} belongs to $\bG_{\tilde{v}}$. Then from \cite[Lemma 2.4.7, 5.]{Voig2015} we have that
	\begin{equation*}
		\cS_\cC^\vr(\cG)\coloneqq\mbox{span}\left\{\vr(x,\xi,\tau)\f\,|\,(x,\xi,\tau)\in\bH_\cG\right\}
	\end{equation*}
	is dense in $(\cT_{\tilde{v}},\norm{\cdot}_{\cT_{\tilde{v}}})$. The claim follows from the trivial fact that $\cS^\vr_\cC(\cG)=\cS_\cC(\cG)$.
\end{proof}

\begin{itemize}
	\item[\textbf{I$'$}.] The reservoir is the Banach space 
	\begin{equation}
	\cR_{\tilde{v}}\coloneqq\cT_{\tilde{v}}^\neg\coloneqq\left\{f\colon\cT_{\tilde{v}}\to\bC\,|\, \text{antilinear and continuous}\right\}.
	\end{equation}
\end{itemize}

\begin{remark}\label{Rem-S_0}
	The Feichtinger algebra $S_0(\cG)$ \cite{Feichtinger-S_0-1979,Feichtinger-1980-kernel, Feichtinger-On-a-new-segal-alg-1981} has numerous equivalent descriptions, see \cite{Jakobsen2018}. It can be seen as the vector space
	\begin{equation}
		S_0(\cG)=\left\{f\in L^2(\cG)\,|\,V_g f\in L^1(\cG\times\hcG)\right\},
	\end{equation}
	for some fixed non-zero window function $g\in L^2(\cG)$. Equipped with the norm
	\begin{equation}
		\norm{f}_{S_0}=\norm{V_g f}_{L^1},
	\end{equation} 
	$S_0(\cG)$ is a Banach space. If $v\equiv 1$, then 
	\begin{equation}
	\cT_1=\bG_1=S_0(\cG),\quad\cR_1=S'_0(\cG).
	\end{equation}
	If $v$ is not constant, then
	\begin{equation}
		\cT_{\tilde{v}}=\bG_{\tilde{v}}\hookrightarrow S_0(\cG),\qquad \cR_{\tilde{v}}\hookleftarrow S'_0(\cG).
	\end{equation}
\end{remark}

\begin{corollary}
	The following inclusion holds true:
	\begin{equation}
		\sA_{\tilde{v}}\subseteq C_0(\cG),
	\end{equation}
	the latter being the space of continuous complex-valued functions on $\cG$ which vanish at infinity.
\end{corollary}
\begin{proof}
	Combining Lemma \ref{Lem-SC-window-space} and Remark \ref{Rem-S_0} we have $\sA_{\tilde{v}}\subseteq\bG_{\tilde{v}}=\cT_{\tilde{v}}\subseteq S_0(\cG)$. We conclude using the fact that $S_0(\cG)\subseteq C_0(\cG)$, see e.g. \cite[Theorem 4.1]{Jakobsen2018}.
\end{proof}

\begin{itemize}	
	\item[\textbf{J$'$}.] We extend the wavelet transform to $f\in\cR_{\tilde{v}}$ and $g\in\cT_{\tilde{v}}$:
	\begin{equation}
	W^\vr_g f\colon \bH_\cG\to\bC, (x,\xi,\tau)\mapsto{}_{\cR_{\tilde{v}}}\la f,\tau M_\xi T_xg\ra_{\cT_{\tilde{v}}}.
	\end{equation}
	From now on we shall simply write $\la\cdot{,}\cdot\ra$. Observe $W^\vr_g f\in C(\bH_\cG)\cap L^\infty_{1/\tilde{v}}(\bH_\cG)$.
\end{itemize}

\begin{remark}\label{Rem-scelta-K}
	The class $\cS_\cC(\cG)$ defined in \eqref{Sc} actually depends on the compact open subgroup $\cK$ in $\cG_0$, where $\cG\cong \rd\times\cG_0$. Then we might write $\cS^\cK_\cC$ in place of $\cS_\cC$. Observe that if $\cK'$ is a compact open subgroup different from $\cK$ Lemma \ref{Lem-SC-window-space} is still valid. More generally, if $\bK$ is the class of all compact open subgroups in $\cG_0$
	:
	\begin{equation}
	\bS_\cC(\cG)\coloneqq\bigcup_{\cK\in\bK}\cS^\cK_\cC(\cG)\subseteq \sA_{\tilde{v}}\subseteq\bG_{\tilde{v}}.
	\end{equation}
	Therefore, coorbit spaces (defined in the subsequent \eqref{Eq-Def-coorbit-Lpq-Heis}) are independent of the window $g\in\bS_\cC(\cG)$. Concretely, this gives us the freedom to chose the subgroup $\cK$ which fits better to our purposes, as done in the proof of Lemma \ref{Lem-Analogo-GalSam-Lem3.2}. Arguing similarly, we could replace $e^{-\pi x^2_1}$ in \eqref{Gauss} with any $e^{-a x^2_1}$, $a>0$. This fact will be used in Proposition \ref{Pro-convolution-Mpq}.\\ 
	From now on, for sake of simplicity, we shall only use the notation $\cS_\cC(\cG)$ with the convention that $\cK$ and the coefficient of the Gaussian on $\rd$ can be chosen freely, so that we shall ever explicitly use the symbol $\bS_\cC(\cG)$. 
\end{remark}
\begin{itemize}	
	\item[\textbf{K$'$}.] The coorbit space on $\bH_\cG$ with respect to $L^{p,q}_{\tilde{m}}(\bH_\cG)$, $0<p,q\leq\infty$,  is, for some fixed non-zero window $g\in\cS_\cC(\cG)$,
	\begin{equation}\label{Eq-Def-coorbit-Lpq-Heis}
	\Co(L^{p,q}_{\tilde{m}}(\bH_\cG))\coloneqq\Co(L^{p,q}_{\tilde{m}})\coloneqq\left\{f\in\cR_{\tilde{v}}\,|\,W^\vr_g f\in W(L^\infty,L^{p,q}_{\tilde{m}}(\bH_\cG))\right\}
	\end{equation}
	endowed with the (quasi-)norm
	\begin{equation}
	\norm{f}_{\Co(L^{p,q}_{\tilde{m}})}\coloneqq\norm{W^\vr_g f}_{W(L^\infty,L^{p,q}_{\tilde{m}})}.
	\end{equation}
\end{itemize}
We stress that $\Co(L^{p,q}_{\tilde{m}})$ is independent of the window  $g$ and  $(\Co(L^{p,q}_{\tilde{m}}),\norm{\cdot}_{\Co(L^{p,q}_{\tilde{m}})})$ is a (quasi-)Banach space continuously embedded into $\cR_{\tilde{v}}$. Moreover, $\norm{\cdot}_{\Co(L^{p,q}_{\tilde{m}})}$ is a $r$-norm, with $r=\min\{1,p,q\}/2$. Notice that 
\begin{equation*}
	\Co(L^{p,q}_{\tilde{m}}(\bH_\cG))=\left\{f\in\cR_{\tilde{v}}\,|\,V_g f\in W(L^\infty,L^{p,q}_{m}(\cG\times\hcG))\right\}.
\end{equation*}

\begin{remark}
	It is clear from the general coorbit theory (see the Appendix A), that the set $\sA_{\tilde{v}}$ defined in \eqref{Eq-maximal-window-space-ALL-modulation-spaces} is the maximal window space  for all the coorbit spaces $\Co(L^{p,q}_{\tilde{m}})$, $0<p,q\leq\infty$. For sake of simplicity we shall mainly work with window functions in the smaller class $\cS_\cC(\cG)$ and adopt the whole space $\sA_{\tilde{v}}$ only when necessary, as done in Section 3.
\end{remark}
The coorbit spaces are independent of the reservoir, in the sense shown below.
\begin{proposition}\label{Pro-Independence-Coo-Reservoir}
	Fix a non-zero window $g\in\cS_\cC(\cG)$, then
	\begin{equation}
		\Co(L^{p,q}_{\tilde{m}}(\bH_\cG))=\left\{f\in S'_0(\cG)\,|\,W^\vr_g f\in W(L^\infty,L^{p,q}_{\tilde{m}}(\bH_\cG))\right\},
	\end{equation}
	in the sense that the restriction map
	\begin{equation*}
		\left\{f\in S'_0(\cG)\,|\,W^\vr_g f\in W(L^{p,q}_{\tilde{m}})\right\}\to\Co(L^{p,q}_{\tilde{m}}(\bH_\cG)),f\mapsto f|_{\cT_{\tilde{v}}}
	\end{equation*}
	is a bijection.
\end{proposition}
\begin{proof}
	 If $v\equiv 1$ the claim is trivial since $\cT_1=S_0$ and $\cR_1=S'_0$, with equal norms, see Remark \ref{Rem-S_0}. If $v$ is not constant, then $v\gtrsim 1$ (since $v$ is bounded from below), and the thesis follows from what observed in Remark \ref{Rem-S_0} and \cite[Theorem 2.4.9, 3.]{Voig2015}.
\end{proof}

\begin{definition}\label{definizione-modulazione}
	Consider $m\in\cM_v(\cG\times\hcG)$ and $0<p,q\leq\infty$. The {\slshape modulation space $M^{p,q}_m(\cG)$} is defined as
	\begin{equation}
		M^{p,q}_m(\cG)=\Co(L^{p,q}_{\tilde{m}}(\bH_\cG)),
	\end{equation}
	endowed with the (quasi-)norm
	\begin{equation}
		\norm{\cdot}_{M^{p,q}_m}=\norm{\cdot}_{\Co(L^{p,q}_{\tilde{m}})}.
	\end{equation}
\end{definition}
We adopt the notations $M^p_m=M^{p,p}_m$ and $M^{p,q}=M^{p,q}_{1}$.

\begin{theorem}
	For $0<p,q\leq\infty$, the modulation spaces $(M^{p,q}_m(\cG),\norm{\cdot}_{M^{p,q}_m})$ are (quasi-)Banach spaces continuously embedded into $\cR_{\tilde{v}}$ which do not depend on the window function $g\in\cS_\cC(\cG)\smallsetminus\{0\}$, in the sense that different windows yield equivalent (quasi-)norms. 
\end{theorem}
\begin{proof}
	Since $(M^{p,q}_m(\cG),\norm{\cdot}_{M^{p,q}_m})=(\Co(L^{p,q}_{\tilde{m}}(\bH_\cG)),\norm{\cdot}_{\Co(L^{p,q}_{\tilde{m}})})$, the claim follows from the coorbit spaces theory, Lemma \ref{Lem-SC-window-space} and \cite[Theorem 2.4.9]{Voig2015}.
\end{proof}

\begin{remark}
	If $g,h\in\cS_\cC(\cG)\smallsetminus\{0\}$ (or $\sA_{\tilde{v}}\smallsetminus\{0\}$) and $f\in M^{p,q}_m(\cG)$, then from the proof in \cite[Theorem 2.4.9]{Voig2015} we see that
	\begin{equation}\label{Eq-Foundations-(11.33)-general}
		\norm{W^\vr_h f}_{W_Q(L^{p,q}_{\tilde{m}})}\underset{Q,v,r}{\lesssim}\frac{\norm{W^\vr_gh}_{W_Q(L^r_{\tilde{v}})}}{\norm{g}^2_{L^2}}\norm{W^\vr_g f}_{W_Q(L^{p,q}_{\tilde{m}})}=\frac{\norm{h}_{M^r_v(\cG)}}{\norm{g}^2_{L^2}}\norm{W^\vr_g f}_{W_Q(L^{p,q}_{\tilde{m}})},
	\end{equation} 
	where $r=\min\{1,p,q\}$ as in \emph{ \textbf{C$'$}}; actually we could replace $r$ with any $r'$ such that $0<r'\leq r$.\\
	In the Banach case we have  $r=1$ and recapture \cite[(11.33)]{Grochenig_2001_Foundations}, after taking into account Theorem \ref{Th-S_C-dense-Banach} and Remark \ref{Rem-Modulation-coincide}.
\end{remark}
In order to prove the expected inclusion relations between modulation spaces, we need particular types of relatively separated families, BUPUs and discrete spaces. The proofs of some subsequent lemmas are omitted because well known or trivial. 

\begin{lemma}\label{Lem-discrete-Lpq-independece-set}
	Let $Q,Q'\subseteq \bH_\cG$ be relatively compact, unit neighbourhoods and $\mathfrak{F}=\{(x_l,\xi_l,\tau_l)\}_{l\in L}\subseteq\bH_\cG$ relatively separated family, consider $0<p,q\leq\infty$ and $m\in\cM_v(\cG\times\hcG)$. Then 
	\begin{equation*}
		(L^{p,q}_{\tilde{m}}(\bH_\cG))_d(\mathfrak{F},Q)=(L^{p,q}_{\tilde{m}}(\bH_\cG))_d(\mathfrak{F},Q')
	\end{equation*}
	with equivalent quasi-norms. Moreover, the equivalence constants depend only on $Q$, $Q'$, $m$ and $v$:
	\begin{equation*}
		\norm{\left(\lambda_l\right)_{l\in L}}_{(L^{p,q}_{\tilde{m}}(\bH_\cG))_d(\mathfrak{F},Q)}\underset{Q,Q',m,v}{\asymp}\norm{\left(\lambda_l\right)_{l\in L}}_{(L^{p,q}_{\tilde{m}}(\bH_\cG))_d(\mathfrak{F},Q')}.
	\end{equation*}
	In particular, they do not depend on $\mathfrak{F}$ or $p$ and $q$.
\end{lemma}
\begin{proof}
	 From \eqref{Eq-Bi-InvarianceLpq} we have that for every $0<p,q\leq\infty$ and $(x,\xi,\tau)\in\bH_\cG$
	\begin{equation*}
		\opnorm{R_{(x,\xi,\tau)}}_{L^{p,q}_{\tilde{m}}\to L^{p,q}_{\tilde{m}}}\leq C v(x,\xi),
	\end{equation*}
	where $C=C(m,v)>0$ is the constant of $v$-moderateness for $m$. 
Since $L^{p,q}_{\tilde{m}}(\bH_\cG)$ is right invariant, the proof goes like the one of \cite[Lemma 2.3.16]{Voig2015} applying the additional majorization above.
\end{proof}

\begin{lemma}\label{Lem-Lpq-Wiener-discrete-equivalence}
	Let $Q,U\subseteq \bH_\cG$ be relatively compact, unit neighbourhoods, $\Delta=\{\delta_l\}_{l\in L}$ $U$-BUPU on $\bH_\cG$ with $U$-localizing family $\mathfrak{F}=\{(x_l,\xi_l,\tau_l)\}_{l\in L}\subseteq\bH_\cG$, consider $0<p,q\leq\infty$ and $m\in\cM_v(\cG\times\hcG)$.Then 
	\begin{equation*}
		\norm{f}_{W_Q(L^{p,q}_{\tilde{m}}(\bH_\cG))}\underset{Q,U,\mathfrak{F,}m,v}{\asymp}\norm{\left(\norm{\delta_l\cdot f}_{L^\infty}\right)_{l\in L}}_{(L^{p,q}_{\tilde{m}}(\bH_\cG))_d(\mathfrak{F},Q)}.
	\end{equation*}
   	In particular, the equivalence constants do not depend on $p$ and $q$.
\end{lemma}
\begin{proof}
	The result come from the proof \cite[Theorem 2.3.17]{Voig2015} (see \eqref{Eq-discrete-equiv-norm-Wiener} in the Appendix A) together with Lemma \ref{Lem-discrete-Lpq-independece-set}.
\end{proof}

\begin{lemma}\label{Lem-product-rel-sep-fam}
	Consider $X=\{x_i\}_{i\in I}\subseteq \cG$, $\Xi=\{\xi_j\}_{j\in J}\subseteq\hcG$ and $T=\{\tau_z\}_{z\in Z}\subseteq \bT$ relatively separated families. Then $\mathfrak{X}\coloneqq X\times\Xi\times T$ is a relatively separated family in $\bH_\cG$.
\end{lemma}
We remark that if the group is $\si$-compact, then any relatively separated family is (at most) countable, \cite[Lemma 2.3.10]{Voig2015}.
\begin{lemma}\label{Lem-BUPU-tensor}
	Let $U\subseteq \cG$ and $D\subseteq\hcG$ be relatively compact, unit neighbourhoods. Consider $\Psi=\{\psi_i\}_{i\in I}$ $U$-BUPU with localizing family $X=\{x_i\}_{i\in I}$ and $\Gamma=\{\gamma_j\}_{j\in J}$ $D$-BUPU with localizing family $\Xi=\{\xi_j\}_{j\in J}$. Then 
	\begin{equation}\label{Eq-BUPU-tensor}
		\Psi\otimes \Gamma\otimes \mathbb{I}\coloneqq\left\{\psi_i\otimes\gamma_j\otimes\chi_\bT,\,(i,j)\in I\times J\right\}
	\end{equation}
	is a $U\times D\times\bT$-BUPU in $\bH_\cG$ with localizing family $\mathfrak{X}\coloneqq X\times\Xi\times\{1\}$.
\end{lemma}
The following is a generalization of \cite[Lemma 2.3.21]{Voig2015} and we follow the pattern of its proof. Although we present it for the Heisenberg group $\bH_\cG\cong\cG\times\left(\hcG\times\bT\right)$, it can be easily adapted to any  product group $G_1\times G_2$, $G_1$ and $G_2$ even not abelian. A similar result for $1\leq p=q\leq\infty$ had been stated in \cite[Remark 4, p. 518]{feichtinger-wiener-type} without proof.
\begin{lemma}\label{Lem-Lpq-disc-ell-p-q}
	Consider $X=\{x_i\}_{i\in I}\subseteq \cG$ and $\Xi=\{\xi_j\}_{j\in J}\subseteq\hcG$ relatively separated families,  $\mathfrak{X}$ as in Lemma \ref{Lem-BUPU-tensor}, and  $V=V_\cG\times V_{\hcG}\times\bT$  as in \eqref{Eq-V-open-neigh}. For $m\in\cM_v(\cG\times\hcG)$ and $0<p,q\leq\infty$,
	\begin{equation}
	\left(L^{p,q}_{\tilde{m}}(\bH_\cG)\right)_d(\mathfrak{X},V)=\ell^{p,q}_{m_{\mathbf{X}}}(I\times J),
	\end{equation}
	where 
	\begin{equation}
	m_{\mathbf{X}}\colon I\times J\to(0,+\infty),(i,j)\mapsto m(x_i,\xi_j),
	\end{equation}  
	with equivalence of the relative (quasi-)norms depending on $X$, $\Xi$, $V_\cG$, $V_{\hcG}$, $v$, $p$ and $q$.
 \end{lemma}
\begin{proof}The proof is divided into four cases.\\
	\emph{Case $p,q<\infty$.} Consider a sequence $\left(\lambda_i\right)_{i\in I}\in \bC^I$. For every $x\in\cG$, we define $I_x$ the subset of indexes 
	\begin{equation}\label{Eq-Def-Ix}
	I_x=\{i\in I\,|\,\chi_{x_i+V_\cG}(x)\neq\varnothing\}\subseteq\{i\in I\,|\,\left(x_i+\overline{V}_\cG\right)\cap\left(x+\{e\}\right)\neq\varnothing\}.
	\end{equation} 
	From \cite[Lemma 2.3.10]{Voig2015}, we have
	\begin{equation}\label{Eq-Ix-unif-bound}
	\#\{i\in I\,|\,\left(x_i+V_\cG\right)\cap\left(x+\{e\}\right)\neq\varnothing\}\leq C_{X,\overline{V}_\cG}<+\infty,\qquad\forall x\in\cG,
	\end{equation}
	$C_{X,\overline{V}_\cG}\in\bN$ as in \eqref{Eq-Constant-Rel-Sep-Fam}. Whence $\# I_x\leq C_{X,\overline{V}_\cG}$ and 
	\begin{align*}
	\left(\sum_{i\in I}\abs{\lambda_i}\chi_{x_i+V_\cG}(x)\right)^{p}&\leq \left(\# I_x\cdot \max\{\abs{\lambda_i}\,|\,i\in I_x\}\right)^{p}\leq C_{X,\overline{V}_\cG}^{p}\max\{\abs{\lambda_i}^{p}\,|\,i\in I_x\}\\
	&\leq C_{X,\overline{V}_\cG}^{p} \sum_{i\in I_x}\abs{\lambda_i}^{p}=C_{X,\overline{V}_\cG}^{p} \sum_{i\in I}\abs{\lambda_i}^{p}\chi_{x_i+V_\cG}(x).
	\end{align*}
	Vice versa
	\begin{align*}
	\left(\sum_{i\in I}\abs{\lambda_i}\chi_{x_i+V_\cG}(x)\right)^{p}&\geq\left(\max\{\abs{\lambda_i}\,|\,i\in I_x\}\right)^{p}=\max\{\abs{\lambda_i}^{p}\,|\,i\in I_x\}\\
	&\geq C_{X,\overline{V}_\cG}^{-p}\sum_{i\in I_x}\abs{\lambda_i}^{p}=C_{X,\overline{V}_\cG}^{-p}\sum_{i\in I}\abs{\lambda_i}^{p}\chi_{x_i+V_\cG}(x).
	\end{align*}
	Hence we have shown the equivalence
	\begin{equation}\label{Eq-equiv-p-sum-in-out}
		\left(\sum_{i\in I}\abs{\lambda_i}\chi_{x_i+V_\cG}(x)\right)^{p}\asymp\sum_{i\in I}\abs{\lambda_i}^{p}\chi_{x_i+V_\cG}(x).
	\end{equation}
	Analogous equivalences hold for every relatively separated family and sequence on the corresponding set of indexes, which under our hypothesis are always countable.
	Due to the chosen $V$,
	\begin{equation*}
	\chi_{(x_i,\xi_j,1)V}(x,\xi,\tau)=\chi_{x_i+V_\cG}(x)\chi_{\xi_j+V_{\hcG}}(\xi)\qquad\forall\,(x,\xi,\tau)\in\bH_\cG.
	\end{equation*}
	Taking a sequence $\left(\lambda_{ij}\right)_{i\in I,j\in J}\in \bC^{I\times J}$ and using twice the  equivalence \eqref{Eq-equiv-p-sum-in-out}, we compute   
	\begin{align*}
		\norm{\left(\lambda_{ij}\right)_{i,j}}_{(L^{p,q}_{\tilde{m}}(\bH_\cG))_d(\mathfrak{X},V)}\!\!\!\!&=\left(\!\int_{\hcG\times\bT}\!\!\left(\!\int_\cG\!\!\left(\sum_{i\in I,j\in J}\abs{\lambda_{ij}}\chi_{x_i+V_\cG}(x)\chi_{\xi_j+V_{\hcG}}(\xi)\!\right)^p \!\!\!m(x,\xi)^p\,dx\!\!\right)^{\frac{q}{p}}\!\!\!d\xi d\tau\!\!\right)^{\frac{1}{q}}\\
		&
		\asymp\left(\int_{\hcG}\left(\int_\cG\sum_{i\in I,j\in J}\abs{\lambda_{ij}}^p\chi_{x_i+V_\cG}(x)\chi_{\xi_j+V_{\hcG}}(\xi) m(x,\xi)^pdx\right)^{\frac{q}{p}}\!\!d\xi\right)^{\frac{1}{q}}\\
		&=\left(\int_{\hcG}\left(\sum_{j\in J} \sum_{i\in I}\abs{\lambda_{ij}}^p\int_\cG m(x,\xi)^p\chi_{x_i+V_\cG}(x)dx\,\chi_{\xi_j+V_{\hcG}}(\xi)\right)^{\frac{q}{p}}\!\!d\xi\right)^{\frac{1}{q}}\\
		&
		\asymp\left(\int_{\hcG}\sum_{j\in J}\left( \sum_{i\in I}\abs{\lambda_{ij}}^p\int_\cG m(x,\xi)^p\chi_{x_i+V_\cG}(x)\,dx\right)^{\frac{q}{p}}\!\!\chi_{\xi_j+V_{\hcG}}(\xi)\,d\xi\!\right)^{\frac{1}{q}}\\
		&=\left(\sum_{j\in J}\int_{V_{\hcG}}\left( \sum_{i\in I}\abs{\lambda_{ij}}^p\int_{V_\cG} m(x+x_i,\xi+\xi_j)^pdx\right)^{\frac{q}{p}}\!\!d\xi\right)^{\frac{1}{q}}.
	\end{align*}
	The monotone convergence theorem justifies the interchanges of integration with summation performed. From \cite[Corollary 2.2.23]{Voig2015} we have
	\begin{equation}\label{Eq-Voig-Cor-2.2.23}
	\left(\sup_{\overline{V}_{1,2}\cup-\overline{V}_{1,2}}v\right)^{-1}m((x,\xi)+(u,\o))\leq m(u,\o)\leq\left(\sup_{\overline{V}_{1,2}\cup-\overline{V}_{1,2}}v\right) m((x,\xi)+(u,\o)),
	\end{equation}
	for every $(u,\o)\in\cG\times\hcG$ and  $(x,\xi)\in\overline{V}_{1,2}$, with $V_{1,2}$ defined in \eqref{Eq-Decomposition-V}. Therefore, if $\xi\in V_{\hcG}$, we have
	\begin{equation}\label{Eq-Equiv-integral-weight}
	\int_{V_\cG} m(x+x_i,\xi+\xi_j)^p\,dx\underset{v,V_{1,2}}{\asymp}\int_{V_\cG}m(x_i,\xi_j)^{p}\,dx=m(x_i,\xi_j)^{p}dx(V_\cG).
	\end{equation}
	Using the equivalences above,
	\begin{align*}
		\norm{\left(\lambda_{ij}\right)_{i,j}}_{(L^{p,q}_{\tilde{m}}(\bH_\cG))_d(\mathfrak{X},V)}&\asymp \left(\sum_{j\in J}\int_{V_{\hcG}}\left( \sum_{i\in I}\abs{\lambda_{ij}}^p\int_{V_\cG} m(x+x_i,\xi+\xi_j)^pdx\right)^{\frac{q}{p}}d\xi\right)^{\frac{1}{q}}\\
		&\asymp \left(\sum_{j\in J}\left( \sum_{i\in I}\abs{\lambda_{ij}}^pm(x_i,\xi_j)^p\right)^{\frac{q}{p}}\right)^{\frac{1}{q}}=\norm{\left(\lambda_{ij}\right)_{i,j}}_{\ell^{p,q}_{m_{\mathbf{X}}}(I\times J)}.
	\end{align*}
	\emph{Case $p=q=\infty$.}
	For $(x,\xi)\in\cG\times\hcG$ we define
	\begin{equation*}
		\mathbf{I}_{(x,\xi)}=\{(i,j)\in I\times J\,|\,\chi_{(x_i,\xi_j)+V_{1,2}}(x,\xi)\neq\varnothing\}.
	\end{equation*}
	 Arguing as for \eqref{Eq-Def-Ix} and \eqref{Eq-Ix-unif-bound}, we have that there exists $N=N(\mathbf{X},V_{1,2})=C_{\mathbf{X},\overline{V}_{1,2}}\in\bN$ (see \eqref{Eq-Constant-Rel-Sep-Fam}) such that $\# \mathbf{I}_{(x,\xi)}\leq N$, where $\mathbf{X}\coloneqq X\times\Xi$ and $V_{1,2}$ as in \eqref{Eq-Decomposition-V}.  Using \eqref{Eq-Voig-Cor-2.2.23}, for $\left(\lambda_{ij}\right)_{i\in I,j\in J}\in \bC^{I\times J}$,
	\begin{align}
		\sum_{i\in I,j\in J}\abs{\lambda_{ij}}\chi_{(x_i,\xi_j)+V_{1,2}}(x,\xi)m(x,\xi)&=\sum_{(i,j)\in \mathbf{I}_{(x,\xi)}}\abs{\lambda_{ij}}m(x_i+u_i(x),\xi_j+\o_j(\xi))\label{Eq-equiv-sum-con&senza-ij}\\
		&\asymp\sum_{(i,j)\in \mathbf{I}_{(x,\xi)}}\abs{\lambda_{ij}}m(x_i,\xi_j)\notag\\
		&=\sum_{i\in I,j\in J}\abs{\lambda_{ij}}m(x_i,\xi_j)\chi_{(x_i,\xi_j)+V_{1,2}}(x,\xi),\notag
	\end{align}
	where $(u_i(x),\o_j(\xi))\in V_{1,2}$ for every $(i,j)\in\mathbf{I}_{(x,\xi)}$. Consider now $\left(\lambda_{ij}\right)_{i\in I,j\in J}\in \ell^{\infty}_{m_{\mathbf{X}}}(I\times J)$. Then
	\begin{align*}
		\norm{\left(\lambda_{ij}\right)_{i,j}}_{(L^\infty_{\tilde{m}}(\bH_\cG))_d(\mathfrak{X},V)}&=\norm{\left(\lambda_{ij}\right)_{i,j}}_{(L^\infty_{m}(\cG\times\hcG))_d(\mathbf{X},V_{1,2})}\\
		&=\norm{\sum_{i\in I,j\in J}\abs{\lambda_{ij}}\chi_{(x_i,\xi_j)+V_{1,2}}(x,\xi)m(x,\xi)}_{L^\infty(\cG\times\hcG)}\\
		&\asymp \norm{\sum_{i\in I,j\in J}\abs{\lambda_{ij}}m(x_i,\xi_j)\chi_{(x_i,\xi_j)+V_{1,2}}(x,\xi)}_{L^\infty(\cG\times\hcG)}\\
		&\leq \norm{\sum_{i\in I,j\in J}\sup_{l,s}\abs{\lambda_{ls}}m(x_l,\xi_s)\chi_{(x_i,\xi_j)+V_{1,2}}(x,\xi)}_{L^\infty(\cG\times\hcG)}\\
		&\leq \norm{\left(\lambda_{ij}\right)_{i,j}}_{\ell^{\infty}_{m_{\mathbf{X}}}(I\times J)}\norm{N\chi_{\cG\times\hcG}}_{L^\infty(\cG\times\hcG)}=N\norm{\left(\lambda_{ij}\right)_{i,j}}_{\ell^{\infty}_{m_{\mathbf{X}}}(I\times J)}.
	\end{align*}
	Vice versa, if $\left(\lambda_{ij}\right)_{i\in I,j\in J}\in (L^\infty_{\tilde{m}}(\bH_\cG))_d(\mathfrak{X},V)$,
	\begin{align*}
		\norm{\left(\lambda_{ij}\right)_{i,j}}_{\ell^{\infty}_{m_{\mathbf{X}}}(I\times J)}&=\sup_{i\in I,j\in J}\abs{\lambda_{ij}}m_{\mathbf{X}}(i,j)=\sup_{i\in I,j\in J}\abs{\lambda_{i,j}}\chi_{(x_i,\xi_j)+V_{1,2}}(x_i,\xi_j)m(x_i,\xi_j)\\
		&\leq\sup_{i\in I,j\in J}\norm{\abs{\lambda_{ij}}\chi_{(x_i,\xi_j)+V_{1,2}}(x,\xi)m(x,\xi)}_{L^\infty(\cG\times\hcG)}\\
		&\leq\sup_{i\in I,j\in J}\norm{\sum_{l\in I,s\in J}\abs{\lambda_{ls}}\chi_{(x_l,\xi_s)+V_{1,2}}(x,\xi)m(x,\xi)}_{L^\infty(\cG\times\hcG)}\\
		&=\norm{\sum_{l\in I,s\in J}\abs{\lambda_{ls}}\chi_{(x_l,\xi_s)+V_{1,2}}(x,\xi)m(x,\xi)}_{L^\infty(\cG\times\hcG)}\\
		&=\norm{\left(\lambda_{ij}\right)_{i,j}}_{(L^\infty_{\tilde{m}}(\bH_\cG))_d(\mathfrak{X},V)}.
	\end{align*}
	\emph{Case $p=\infty$ and $q<\infty$.} 
	We  show the equivalence
	\begin{align}\label{Eq-esssup-in-out}
	\essupp_{x\in\cG}\sum_{i\in I,j\in J}&\abs{\lambda_{ij}}m_{\mathbf{X}}(i,j)\chi_{x_i+V_{\cG}}(x)\chi_{\xi_j+V_{\hcG}}(\xi)\\&\underset{\Xi,V_{\hcG}}
	\asymp\sum_{j\in J}\essupp_{x\in\cG}\sum_{i\in I }\abs{\lambda_{ij}}m_{\mathbf{X}}(i,j)\chi_{x_i+V_{\cG}}(x)\chi_{\xi_j+V_{\hcG}}(\xi).\notag
	\end{align}
	In fact,  arguing as in  \eqref{Eq-Def-Ix} and \eqref{Eq-Ix-unif-bound}, for $\xi\in\hcG$ fixed and
	$J_\xi\coloneqq\{j\in J\,|\,\chi_{\xi_j+V_{\hcG}}(\xi)\neq\varnothing\}$,
	 there exists $M=M(\Xi,V_{\hcG})\in\bN$ such that $\# J_\xi\leq M$. Therefore,
	\begin{align*}
		\essupp_{x\in\cG}\sum_{i\in I,j\in J}\abs{\lambda_{ij}}m_{\mathbf{X}}(i,j)\chi_{x_i+V_{\cG}}(x)&\chi_{\xi_j+V_{\hcG}}(\xi)=\essupp_{x\in\cG}\sum_{j\in J_\xi}\sum_{i\in I }\abs{\lambda_{ij}}\chi_{x_i+V_{\cG}}(x)m_{\mathbf{X}}(i,j)\\
		&\leq\sum_{j\in J_\xi}\essupp_{x\in\cG}\sum_{i\in I }\abs{\lambda_{ij}}\chi_{x_i+V_{\cG}}(x)m_{\mathbf{X}}(i,j)\\
		&=\sum_{j\in J}\essupp_{x\in\cG}\sum_{i\in I}\abs{\lambda_{ij}}\chi_{x_i+V_{\cG}}(x)m_{\mathbf{X}}(i,j)\chi_{\xi_j+V_{\hcG}}(\xi).
	\end{align*}
	On the other hand,
	\begin{align*}
		\sum_{j\in J}\essupp_{x\in\cG}\sum_{i\in I}\abs{\lambda_{ij}}\chi_{x_i+V_{\cG}}(x)m_{\mathbf{X}}(i,j)&\chi_{\xi_j+V_{\hcG}}(\xi)=\sum_{j\in J_\xi}\essupp_{x\in\cG}\sum_{i\in I }\abs{\lambda_{ij}}\chi_{x_i+V_{\cG}}(x)m_{\mathbf{X}}(i,j)\\
		&\leq M \max\{\essupp_{x\in\cG}\sum_{i\in I }\abs{\lambda_{ij}}\chi_{x_i+V_{\cG}}(x)m_{\mathbf{X}}(i,j)\,|\,j\in J_{\xi}\}\\
		&\leq M \essupp_{x\in\cG}\sum_{j\in J_\xi}\sum_{i\in I }\abs{\lambda_{ij}}\chi_{x_i+V_{\cG}}(x)m_{\mathbf{X}}(i,j)\\
		&=M \essupp_{x\in\cG}\sum_{j\in J}\sum_{i\in I }\abs{\lambda_{ij}}\chi_{x_i+V_{\cG}}(x)m_{\mathbf{X}}(i,j)\chi_{\xi_j+V_{\hcG}}(\xi).
	\end{align*}
	Finally, using the previous cases, the equivalences in  \eqref{Eq-equiv-sum-con&senza-ij} and \eqref{Eq-esssup-in-out}, we can write
	\begin{align*}
	\norm{\left(\lambda_{ij}\right)_{i,j}}_{(L^{\infty,q}_{\tilde{m}}(\bH_\cG))_d(\mathfrak{X},V)}&=\left(\int_{\hcG}\left(\essupp_{x\in\cG}\sum_{i\in I,j\in J}\abs{\lambda_{ij}}\chi_{x_i+V_{\cG}}(x)\chi_{\xi_j+V_{\hcG}}(\xi)m(x,\xi)\right)^q\, d\xi\right)^{\frac{1}{q}}\\
		&\asymp\left(\int_{\hcG}\left(\sum_{j\in J}\essupp_{x\in\cG}\sum_{i\in I }\abs{\lambda_{ij}}m_{\mathbf{X}}(i,j)\chi_{x_i+V_{\cG}}(x)\chi_{\xi_j+V_{\hcG}}(\xi)\right)^q\, d\xi\right)^{\frac{1}{q}}\\
		&=\left(\int_{\hcG}\left(\sum_{j\in J}\norm{\sum_{i\in I }\abs{\lambda_{ij}}m_{\mathbf{X}}(i,j)\chi_{x_i+V_{\cG}}(\cdot)}_{L^\infty(\cG)}\chi_{\xi_j+V_{\hcG}}(\xi)\right)^q\, d\xi\right)^{\frac{1}{q}}\\
		&=\left(\int_{\hcG}\left(\sum_{j\in J}\norm{\left(\lambda_{ij}m_{\mathbf{X}}(i,j)\right)_{i\in I}}_{(L^\infty(\cG))_d(X,V_\cG)}\chi_{\xi_j+V_{\hcG}}(\xi)\right)^q\, d\xi\right)^{\frac{1}{q}}\\
		&\asymp\left(\int_{\hcG}\left(\sum_{j\in J}\norm{\left(\lambda_{ij}m_{\mathbf{X}}(i,j)\right)_{i\in I}}_{\ell^\infty(I)}\chi_{\xi_j+V_{\hcG}}(\xi)\right)^q\, d\xi\right)^{\frac{1}{q}}\\
		&=\norm{\left(\norm{\left(\lambda_{ij}m_{\mathbf{X}}(i,j)\right)_{i\in I}}_{\ell^\infty(I)}\right)_{j\in J}}_{(L^q(\hcG))_d(\Xi,V_{\hcG})}\\
		&\asymp\norm{\left(\norm{\left(\lambda_{ij}m_{\mathbf{X}}(i,j)\right)_{i\in I}}_{\ell^\infty(I)}\right)_{j\in J}}_{\ell^q(J)}	=\norm{\left(\lambda_{ij}\right)_{i,j}}_{\ell^{\infty,q}_{m_{\mathbf{X}}}(I\times J)}.
	\end{align*}
	\emph{Case $p<\infty$ and $q=\infty$.} Similarly to what has been done before,
	\begin{align*}
		\norm{\left(\lambda_{ij}\right)_{i,j}}_{(L^{p,\infty}_{\tilde{m}}(\bH_\cG))_d(\mathfrak{X},V)}&
		=\essupp_{\xi\in\hcG}\!\!\left(\!\int_\cG\!\left(\sum_{i\in I,j\in J}\abs{\lambda_{ij}}\chi_{x_i+V_\cG}(x)\chi_{\xi_j+V_{\hcG}}(\xi)\!\!\right)^p\!\!\! m(x,\xi)^pdx\right)^\frac1p\\
		&\asymp\essupp_{\xi\in\hcG}\!\!\left(\int_\cG\sum_{i\in I,j\in J}\abs{\lambda_{ij}}^p\chi_{x_i+V_\cG}(x)\chi_{\xi_j+V_{\hcG}}(\xi)m(x,\xi)^pdx\right)^\frac1p\\
		&\asymp\essupp_{\xi\in\hcG}\left(\sum_{i\in I,j\in J}\abs{\lambda_{ij}}^pm(x_i,\xi_j)^p\chi_{\xi_j+V_{\hcG}}(\xi)\right)^\frac1p\\
		&\asymp\essupp_{\xi\in\hcG}\sum_{j\in J}\left(\sum_{i\in I}\abs{\lambda_{ij}}^pm(x_i,\xi_j)^p\right)^\frac1p\chi_{\xi_j+V_{\hcG}}(\xi)\\
		&=\essupp_{\xi\in\hcG}\sum_{j\in J}\norm{\left(\lambda_{ij}m_{\mathbf{X}}(i,j)\right)_{i\in I}}_{\ell^p(I)}\chi_{\xi_j+V_{\hcG}}(\xi)\\
		&=\norm{\left(\norm{\left(\lambda_{ij}m_{\mathbf{X}}(i,j)\right)_{i\in I}}_{\ell^p(I)}\right)_{j\in J}}_{(L^\infty(\hcG))_d(\Xi,V_{\hcG})}\\&\asymp
		\norm{\left(\lambda_{ij}\right)_{i,j}}_{\ell^{p,\infty}_{m_{\mathbf{X}}}(I\times J)}.
	\end{align*}
	The proof is concluded.
\end{proof}

\begin{remark}\label{Rem-equivalence-constants}
	We want to state explicitly the equivalence constants involved in the previous lemma. We distinguish four cases, as done in the proof.\\
	\emph{Case $p,q<\infty$.} We have
	\begin{equation*}
		A^{-1}B \norm{\left(\lambda_{ij}\right)_{i,j}}_{\ell^{p,q}_{m_{\mathbf{X}}}(I\times J)}\leq\norm{\left(\lambda_{ij}\right)_{i,j}}_{(L^{p,q}_{\tilde{m}}(\bH_\cG))_d(\mathfrak{X},V)}\leq A B\norm{\left(\lambda_{ij}\right)_{i,j}}_{\ell^{p,q}_{m_{\mathbf{X}}}(I\times J)},
	\end{equation*}
 	where
 	\begin{align*}
	A\coloneqq A(X,\Xi,V_\cG,V_{\hcG},v,p)&\coloneqq C_{X,\overline{V}_{\cG}}C^{\frac1p+1}_{\Xi,\overline{V}_{\hcG}}\left(\sup_{\overline{V}_{1,2}\cup-\overline{V}_{1,2}}v\right),\\
	B\coloneqq B(V_\cG,V_{\hcG},p,q)&\coloneqq dx(V_{\cG})^{\frac1p}\,d\xi(V_{\hcG})^{\frac1q}.
 	\end{align*}
 	\emph{Case $p=q=\infty$.} The equivalence is 
	\begin{equation*}
		\norm{\left(\lambda_{ij}\right)_{i,j}}_{\ell^{\infty}_{m_{\mathbf{X}}}(I\times J)}\leq\norm{\left(\lambda_{ij}\right)_{i,j}}_{(L^{\infty}_{\tilde{m}}(\bH_\cG))_d(\mathfrak{X},V)}\leq \left(\sup_{\overline{V}_{1,2}\cup-\overline{V}_{1,2}}v\right) C_{\mathbf{X},\overline{V}_{1,2}}\norm{\left(\lambda_{ij}\right)_{i,j}}_{\ell^{\infty}_{m_{\mathbf{X}}}(I\times J)}.
	\end{equation*}
	\emph{Case $p=\infty$ and $q<\infty$.}  We got
	\begin{equation*}
		D\norm{\left(\lambda_{ij}\right)_{i,j}}_{\ell^{\infty,q}_{m_{\mathbf{X}}}(I\times J)}\leq\norm{\left(\lambda_{ij}\right)_{i,j}}_{(L^{\infty,q}_{\tilde{m}}(\bH_\cG))_d(\mathfrak{X},V)}\leq E\norm{\left(\lambda_{ij}\right)_{i,j}}_{\ell^{\infty,q}_{m_{\mathbf{X}}}(I\times J)},
	\end{equation*}
	where
	\begin{align*}
		D\coloneqq D(\Xi,V_\cG,V_{\hcG},v,q)&\coloneqq C^{-2}_{\Xi,\overline{V}_{\hcG}}\left(\sup_{\overline{V}_{1,2}\cup-\overline{V}_{1,2}}v\right)^{-1}\,d\xi(V_{\hcG})^{\frac1q},\\
		E\coloneqq B(X,\Xi,V_\cG,V_{\hcG},v,q)&\coloneqq C_{X,\overline{V}_{\cG}}
		C_{\Xi,\overline{V}_{\hcG}}\left(\sup_{\overline{V}_{1,2}\cup-\overline{V}_{1,2}}v\right)\,d\xi(V_{\hcG})^{\frac1q}.
	\end{align*}
	\emph{Case $p<\infty$ and $q=\infty$.} The last equivalence is given by
		\begin{equation*}
		L\norm{\left(\lambda_{ij}\right)_{i,j}}_{\ell^{p,\infty}_{m_{\mathbf{X}}}(I\times J)}\leq\norm{\left(\lambda_{ij}\right)_{i,j}}_{(L^{p,\infty}_{\tilde{m}}(\bH_\cG))_d(\mathfrak{X},V)}\leq M\norm{\left(\lambda_{ij}\right)_{i,j}}_{\ell^{p,\infty}_{m_{\mathbf{X}}}(I\times J)},
	\end{equation*}
	where
	\begin{align*}
		L\coloneqq L(X,\Xi,V_\cG,V_{\hcG},v,p)&\coloneqq C^{-1}_{X,\overline{V}_{\cG}}
		C^{-1}_{\Xi,\overline{V}_{\hcG}}\left(\sup_{\overline{V}_{1,2}\cup-\overline{V}_{1,2}}v\right)^{-1}\,dx(V_{\cG})^{\frac1p},\\
		M\coloneqq M(X,\Xi,V_\cG,V_{\hcG},v,p)&\coloneqq C_{X,\overline{V}_{\cG}}
		C^2_{\Xi,\overline{V}_{\hcG}}\left(\sup_{\overline{V}_{1,2}\cup-\overline{V}_{1,2}}v\right)\,dx(V_{\cG})^{\frac1p}.
	\end{align*}
	We recall that the definition of the constants $C_{X,\overline{V}_{\cG}}, C_{\Xi,\overline{V}_{\hcG}}, C_{\mathbf{X},\overline{V}_{1,2}}$ is given in  \eqref{Eq-Constant-Rel-Sep-Fam}.
\end{remark}

On account of the constants shown in the previous remark, we have the following corollary.

\begin{corollary}\label{Cor-equivalence-constants-pq>delta}
	Fix $0<\delta\leq\infty$ and take $p,q$ such that $0<\delta\leq p,q\leq\infty$. Under the same assumptions of Lemma \ref{Lem-Lpq-disc-ell-p-q}, there are two constants
	\begin{equation*}
		C_1\coloneqq C_1(X,\Xi,V_\cG,V_{\hcG},v,\delta)>0\quad\text{and}\quad C_1\coloneqq C_1(X,\Xi,V_\cG,V_{\hcG},v,\delta)>0
	\end{equation*}
	such that 
	\begin{equation*}
		C_1\norm{\left(\lambda_{ij}\right)_{i,j}}_{\ell^{p,q}_{m_{\mathbf{X}}}(I\times J)}\leq\norm{\left(\lambda_{ij}\right)_{i,j}}_{(L^{p,q}_{\tilde{m}}(\bH_\cG))_d(\mathfrak{X},V)}\leq C_2\norm{\left(\lambda_{ij}\right)_{i,j}}_{\ell^{p,q}_{m_{\mathbf{X}}}(I\times J)}
	\end{equation*}
	for every sequence $\left(\lambda_{ij}\right)_{i,j}$ in $(L^{p,q}_{\tilde{m}}(\bH_\cG))_d(\mathfrak{X},V)=\ell^{p,q}_{m_{\mathbf{X}}}(I\times J)$.
\end{corollary}	
\begin{proof}
	We notice that if $b\geq1$, then $b^{\frac1p}$ is a strictly decreasing function of $p\in(0,\infty)$ and $b^{\frac1p}\geq1$. Likewise $b^{-\frac1p}$ is strictly increasing and $0<b^{-\frac1p}\leq1$. The claim follows now from Remark \ref{Rem-equivalence-constants}.
\end{proof}

\begin{remark}	
	Although in Lemma \ref{Lem-Lpq-disc-ell-p-q} we considered $V=V_{\cG}\times V_{\hcG}\times\bT$ with $V_{\cG}$ and $V_{\hcG}$ open, this last assumption can be relaxed into measurability. Even in this case the above lemma and the subsequent Corollary \ref{Cor-Lpq-disc-ell-p-q} hold true.
\end{remark}

\begin{corollary}\label{Cor-Lpq-disc-ell-p-q}
	Consider $0<p_1\leq p_2\leq\infty$, $0<q_1\leq q_2\leq\infty$ and $m_1,m_2\in\cM_v(\cG\times\hcG)$ such that $m_2\lesssim m_1$. Let $V$, $X$, $\Xi$ and $\mathfrak{X}$ be as in Lemma \ref{Lem-Lpq-disc-ell-p-q}. Then
	\begin{equation}
	\left(L^{p_1,q_1}_{\tilde{m}_1}(\bH_\cG)\right)_d(\mathfrak{X},V)\hookrightarrow\left(L^{p_2,q_2}_{\tilde{m}_2}(\bH_\cG)\right)_d(\mathfrak{X},V).
	\end{equation}
\end{corollary}
\begin{proof}
It is a straightforward consequence of Lemma \ref{Lem-Lpq-disc-ell-p-q} and the continuous inclusions
	\begin{equation}
		\ell^{p_1,q_1}_{m_{1,\mathbf{X}}}(I\times J)\hookrightarrow\ell^{p_2,q_2}_{m_{2,\mathbf{X}}}(I\times J),
	\end{equation}
	since $m_{2,\mathbf{X}}\lesssim m_{1,\mathbf{X}}$.
\end{proof}

\begin{proposition}\label{Pro-Incl-Mpq}
	Consider $0<p_1\leq p_2\leq\infty$, $0<q_1\leq q_2\leq\infty$ and $m_1,m_2\in\cM_v(\cG\times\hcG)$ such that $m_2\lesssim m_1$. Then we have the following continuous inclusions:
	\begin{equation}\label{Eq-Incl-Mpq}
	M^{p_1,q_1}_{m_1}(\cG)\hookrightarrow M^{p_1,q_2}_{m_2}(\cG).
	\end{equation}
\end{proposition}
\begin{proof}
	Under  the hypothesis of Lemma \ref{Lem-BUPU-tensor}, it is always possible to find a BUPU on $\bH_\cG$ of the type \eqref{Eq-BUPU-tensor}, see \cite[Lemma 2.3.12]{Voig2015}. For such a BUPU $$\Psi\otimes\Gamma\otimes\mathbb{I}=\{\psi_i\otimes\gamma_j\otimes\chi_\bT,\,(i,j)\in I\times J\},$$ the corresponding localizing family $\mathfrak{X}=X\times\Xi\times\{1\}$ fulfils the requirements of Corollary \ref{Cor-Lpq-disc-ell-p-q}. To get the desired result we use the equivalence of (quasi-)norms shown in \eqref{Eq-discrete-equiv-norm-Wiener}:
	\begin{align*}
	\norm{f}_{M^{p_2,q_2}_{m_2}}&\asymp\norm{W^\vr_g f}_{W(L^{p_2,q_2}_{\tilde{m}_2})}\asymp\norm{\left(\norm{\left(\psi_i\otimes\gamma_j\otimes\chi_\bT\right)\cdot W^\vr_g f}_{L^\infty}\right)_{i,j}}_{(L^{p_2,q_2}_{\tilde{m_2}}(\bH_\cG))_d(\mathfrak{X},V)}\\
	&\lesssim \norm{\left(\norm{\left(\psi_i\otimes\gamma_j\otimes\chi_\bT\right)\cdot W^\vr_g f}_{L^\infty}\right)_{i,j}}_{(L^{p_1,q_1}_{\tilde{m_1}}(\bH_\cG))_d(\mathfrak{X},V)}\\
	&\asymp\norm{W^\vr_g f}_{W(L^{p_1,q_1}_{\tilde{m}_1})}\asymp\norm{f}_{M^{p_1,q_1}_{m_1}}.
	\end{align*}
	This concludes the proof.
\end{proof}

If $m\in\cM_v(\cG\times\hcG)$, from the submultiplicativity and symmetry of $v$ we have $1/m\in\cM_v(\cG\times\hcG)$. This remark is implicitly used in the following issue.
\begin{proposition}\label{Pro-Duality-Mpq-Banach}
	If $1\leq p,q<\infty$, then $\left(M^{p,q}_m (\cG)\right)'=M^{p',q'}_{1/m}(\cG)$ under the duality
	\begin{equation}
		\la f,h\ra= \la V_g f,V_g h\ra_{L^2(\cG\times\hcG)},
	\end{equation}
	for all $f\in M^{p,q}_m(\cG)$, $h\in M^{p',q'}_{1/m}(\cG)$ and some $g\in\cS_\cC(\cG)\smallsetminus\{0\}$.
\end{proposition}
\begin{proof}
	For $1\leq p,q\leq\infty$, $L^{p,q}_{\tilde{m}}(\bH_\cG)$ is a solid bi-invariant Banach function space continuously embedded into $L^1_{loc}(\bH_\cG)$. Therefore, from Theorem \ref{Th-Rauhut-Coo-Th6.1} combined with Remark \ref{Rem-WaveletSchrodinger}, we have
	\begin{equation}\label{Eq-Proof-Duality-1}
		M^{p,q}_m(\cG)=\Co(L^{p,q}_{\tilde{m}}(\bH_\cG))=\CoFG(L^{p,q}_{\tilde{m}}(\bH_\cG))=\{f\in\cR_{\tilde{v}}\,|\,V_g f\in L^{p,q}_{m}(\cG\times\hcG)\}
	\end{equation} 
	with
	\begin{equation}\label{Eq-Proof-Duality-2}
		\norm{V_gf}_{W(L^{p,q}_m)}\asymp\norm{V_gf}_{L^{p,q}_m}.
	\end{equation}
	The proof then goes as in \cite[Theorem 11.3.6]{Grochenig_2001_Foundations}, after noticing that we can identify $(L^1_m)'$ with $L^\infty_{1/m}$ since under our assumptions $\cG\times\hcG$ is $\si$-finite, similarly for mixed-norm cases. 
\end{proof}

\begin{theorem}\label{Th-S_C-dense-Banach}
		$(i)$ If $0< p,q<\infty$, then $\cS_\cC(\cG)$ is quasi-norm-dense in $M^{p,q}_m(\cG)$.
		$(ii)$ If $1\leq p,q\leq\infty$ and at least one between $p$ and $q$ is equal to $\infty$, then $\cS_\cC(\cG)$ is w-$*$-dense in $M^{p,q}_m(\cG)$.
\end{theorem}
\begin{proof}
	For any $0\leq p,q\leq\infty$,  $\cS_\cC(\cG)$ is a subspace of $M^{p,q}_m (\cG)$, cf.  the computations in the proof of Lemma \ref{Lem-SC-window-space} and the inclusions in \eqref{Eq-Incl-Mpq}.\par
	$(i)$ Let $\f$ be as in \eqref{Gauss} and consider the relatively compact unit neighbourhood $U_0$ coming from Theorem \ref{Th-Voigt-Th-2.4.19}. Without loss of generality we can assume $U_0=V_{\cG}\times V_{\hcG}\times \bT=V$ as in \eqref{Eq-V-open-neigh}, see the proofs of \cite[Theorem 2.4.19]{Voig2015} and \cite[Lemma 2.4.17]{Voig2015}. Then there exists a $U_0$-BUPU with localizing family $\mathfrak{X}=\{(x_i,\xi_j,1)\}_{(i,j)\in I\times J}$ such that any $f\in M^{p,q}_m(\cG)$ can be written as 
	\begin{equation}\label{Eq-decomp-Rv}
		f=\sum_{i\in I,j\in J}\lambda_{ij}(f)\vr(x_i,\xi_j,1)\f=\sum_{i\in I,j\in J}\lambda_{ij}(f)\pi(x_i,\xi_j)\f,
	\end{equation}
	with unconditional convergence in $M^{p,q}_m(\cG)$ since the finite sequences are dense in $\ell^{p,q}_{m_{\mathbf{X}}}(I\times J)=(L^{p,q}_{\tilde{m}}(\bH_\cG))_d(\mathfrak{X},V)$,  $p,q<\infty$.\\
	$(ii)$ We show the case $p=q=\infty$, the remaining ones are  analogous. From Proposition  \ref{Pro-Duality-Mpq-Banach}, $M^\infty_m(\cG)$ can be seen as the dual of $M^1_{1/m}(\cG)$. Therefore, with 	$\f$ the Gaussian  in \eqref{Gauss},
	\begin{equation}
	{}^\perp\cS_\cC(\cG)=\{f\in M^{1}_{1/m}\,|\,\la V_\f f,V_\f h\ra=0,\quad\forall\,h\in\cS_\cC(\cG)\}.
	\end{equation} 
For fixed $(u,\o)\in\cG\times\hcG$ consider $h=\pi(u,\o)\f\in\cS_\cC(\cG)$. From Lemma \ref{Lem-STFTinS_C}
	\begin{equation}
	V_\f h(x,\xi)=\overline{\la\xi-\o,u\ra}T_{(u,\o)}V_\f \f(x,\xi).
	\end{equation}
	In particular, from \eqref{Eq-SFTFgauss}, $V_\f h(u,\o)\neq 0$ and it is continuous. Therefore if $f\in {}^\perp\cS_\cC(\cG)$
	\begin{equation*}
	\la V_\f f,V_\f h\ra=\int_{\cG\times\hcG}V_\f f(x,\xi)\overline{V_\f h(x,\xi)}\,dxd\xi=0\quad\Rightarrow\quad V_\f f\overline{V_\f h}=0\,\text{a.e.},
	\end{equation*}
	but since $V_\f f\overline{V_\f h}$ is continuous this implies $V_\f f(x,\xi)\overline{V_\f h(x,\xi)}=0$ for every $(x,\xi)\in\cG\times\hcG$. Necessarily
	 \,$V_\f f$ vanishes on a neighbourhood of $(u,\o)$. On account of the arbitrariness of the point $(u,\o)\in\cG\times\hcG$, we have $V_\f f\equiv 0$ which also means $W^\vr_\f f\equiv 0$. Since the application
	\begin{equation*}
	W^\vr_\f\colon\cR_{\tilde{v}}\to C(\bH_\cG)\cap L^\infty_{1/\tilde{v}}(\bH_\cG)
	\end{equation*}
	is injective, see \cite[Lemma 2.4.8]{Voig2015}, we infer $f=0$. Therefore ${}^\perp\cS_\cC(\cG)=\{0\}$ and
	\begin{equation*}
	\overline{\cS_\cC(\cG)}^{w-\ast}=\left({}^\perp\cS_\cC(\cG)\right)^\perp=\left(\{0\}\right)^\perp=M^{\infty}_m(\cG).
	\end{equation*}
	This concludes the proof.
\end{proof}

\begin{lemma}
	For every , $0<p,q\leq\infty$ and $m\in\cM_v(\cG\times\hcG)$
	\begin{equation*}
		\sA_{\tilde{v}}(\cG)\subseteq M^{p,q}_m(\cG).
	\end{equation*}
\end{lemma}
\begin{proof}
	We just need to show that for every $0<r\leq 1$ the inclusion
	\begin{equation}
		\sA_{\tilde{v}}(\cG)\subseteq M^r_v(\cG)
	\end{equation}
	holds true, then the claim follows from the inclusion relations for modulation spaces.
	From \eqref{Eq-remark-bA^r_v} and the inclusion relations in \cite[p. 113]{Voig2015}, if $g\in \sA_{\tilde{v}}\subseteq \bA^r_{\tilde{v}}$ and $\f$ is the Gaussian as in \eqref{Gauss}, we get that
	\begin{equation*}
		W^\vr_\f g\in  W^R(L^\infty,W(L^\infty,L^r_{\tilde{v}}))\hookrightarrow W(L^\infty, L^r_{\tilde{v}}).
	\end{equation*}
	Hence $g\in M^r_v(\cG)$.
\end{proof}

\begin{corollary}
	If $0<p,q<\infty$, then $\sA_{\tilde{v}}$ is quasi-norm-dense in $M^{p,q}_m(\cG)$.
\end{corollary}
\begin{proof}
	The claim follows from the above theorem, the previous lemma and the inclusion $\cS_\cC\subseteq\sA_{\tilde{v}}$.
\end{proof}

\begin{corollary}\label{Cor-Sc-w*-dense-S'0}
	For every $f\in S'_0(\cG)$ there exists a net $\left(f_\alpha\right)_{\alpha\in A}\subseteq \cS_\cC(\cG)$ such that
	\begin{equation}
		\lim_{\al\in A}\la f_\al, h\ra_{L^2(\cG)}={}_{S'_0}\la f,h\ra_{S_0},\qquad\forall\, h\in S_0(\cG).
	\end{equation}
\end{corollary}
\begin{proof}
	From Lemma \ref{Lem-Sc-dense-Tv} we have that $\cS_\cC(\cG)$ is norm-dense in $\cT_{1}=S_0(\cG)$. From \cite[Proposition 6.15]{Jakobsen2018} there exists a bounded net $(f_\be)_{\be\in B}\subseteq S_0(\cG)$ such that
	\begin{equation}
	\lim_{\be\in B}\la f_\be,h\ra_{L^2(\cG)}={}_{S'_0}\la f,h\ra_{S_0},\qquad\forall\,h\in S_0(\cG).
	\end{equation} 
	This concludes the proof.
\end{proof}

\begin{remark}\label{Rem-Modulation-coincide}
	From  Theorem \ref{Th-S_C-dense-Banach} and relations  \eqref{Eq-Proof-Duality-1} and \eqref{Eq-Proof-Duality-2} it follows that the modulation spaces introduced in Definition \ref{definizione-modulazione} coincide with the classical ones in \cite{feichtinger-modulation,GroStr2007}. This implies that
	\begin{equation}
		M^1_m(\cG)\cong\left(\text{clos}_{M^\infty_{1/m}}\left(\cS_\cC(\cG)\right)\right)',
	\end{equation}
	the dual of the closure of $\cS_\cC(\cG)$ with respect to the norm on $M^\infty_{1/m}(\cG)$. If $f\in M^{\infty,1}_m(\cG)$ and $g\in M^{1,\infty}_{1/m}(\cG)$, then for $\f$ as in \eqref{Gauss} 
	\begin{equation}\label{Eq-duality-Minfty1-M1infty}
		\abs{\la V_\f f,V_\f g\ra_{L^2(\cG\times\hcG)}}\lesssim\norm{f}_{M^{\infty,1}_m}\norm{g}_{M^{1,\infty}_{1/m}}.
	\end{equation}
	See \cite[Proposition 2.2]{GroStr2007}.\\
	(ii) The theory for $\cG=\rd$ developed in \cite{Galperin2004} is recovered for every $0<p,q\leq\infty$. In fact, it was observed in \cite[Section 8]{Rauhut2007Coorbit} that from \cite[Lemma 3.2]{Galperin2004} follows the equality
	\begin{equation*}
		\Co(L^{p,q}_{\tilde{m}}(\bH_{\rd}))=\{f\in\cS'\,|\,V_gf\in L^{p,q}_m(\rdd)\}\qquad 0<p,q\leq\infty,
	\end{equation*}
	with equivalent (quasi-)norms. 
\end{remark}
	 For a general LCA group $\cG$ it is an open problem whether a construction of the type
	 \begin{equation*}
	 \{f\in\cR_{\tilde{v}}\,|\,V_g f\in L^{p,q}_{m}(\cG\times\hcG)\},
	 \end{equation*}
	 with obvious (quasi-)norm, could make sense or not when at least one between $p$ and $q$ is smaller than $1$. However, we are able to answer affirmatively if $\cG$ is discrete or compact, see the lemma and corollary below.
\begin{lemma}\label{Lem-Analogo-GalSam-Lem3.2}
	Let 
	$0<p,q\leq\infty$. Suppose $\cG$ is discrete or compact. Then there exists $C>0$ such that for every $f\in M^{p,q}_m(\cG)$
	\begin{equation}\label{Eq-Analogo-GalSam-Lem3.2}
		\norm{W^\vr_g f}_{W(L^{p,q}_{\tilde{m}})}\leq C \norm{W^\vr_g f}_{L^{p,q}_{\tilde{m}}},
	\end{equation}
	for some $g\in\cS_\cC(\cG)\smallsetminus\{0\}$.
\end{lemma}
\begin{proof}
	If we prove for some suitable unit neighbourhood $Q\subseteq\cG\times\hcG$ that there exists $C>0$ such that
	\begin{equation}
		\norm{V_g f}_{W_Q(L^{p,q}_{m})}\leq C \norm{V_g f}_{L^{p,q}_m},
	\end{equation}
	then \eqref{Eq-Analogo-GalSam-Lem3.2} holds true, see Remark \ref{Rem-WaveletSchrodinger}. Moreover, as shown in Proposition \ref{Pro-Independence-Coo-Reservoir}, we can consider the modulation spaces as subsets of $S'_0(\cG)$.\\
	\emph{Case $\cG$ discrete.} 
	$\hcG$ is compact and the structure theorem reads as $\cG=\cG_0$ and $\hcG=\hcG_0$. In the definition of the Gaussian function \eqref{Gauss} we take, Remark \ref{Rem-scelta-K}, the open and compact subgroup $\cK=\{e\}$, therefore
	\begin{equation*}
		\f(x)\coloneqq\chi_{\{e\}}(x)\eqqcolon\delta_e(x).
	\end{equation*}
	We also choose $Q\coloneqq\{e\}\times\hcG$, which is a measurable, relatively compact, unit neighbourhood. Fix $f\in M^{p,q}_m(\cG)$, from \cite[Proposition 6.15]{Jakobsen2018}, we have that there exists a bounded net $(f_\a)_{\a\in A}\subseteq S_0(\cG)$ such that
	\begin{equation}
		\lim_{\al\in A}\la f_\al,h\ra_{L^2(\cG)}={}_{S'_0}\la f,h\ra_{S_0},\qquad\forall\,h\in S_0(\cG).
	\end{equation}
	Recall that $\cS_\cC(\cG)\subseteq
	S_0(\cG)$, then adopting the widow function $\f$, we compute
	\begin{align*}
		V_\f f(x,\xi)&=\la f,\pi(x,\xi)\delta_e\ra=\lim_{\al\in A}\la f_{\al},\pi(x,\xi)\delta_e\ra=\lim_{\al\in A}\sum_{u\in\cG}f_\al(u)\overline{\la\xi,u\ra\delta_x(u)}\\
		&=\lim_{\al\in A}f_\al(x)\overline{\la\xi,x\ra}=\overline{\la\xi,x\ra}\lim_{\al\in A}f_\al(x),\\
		\sfM_QV_\f f(x,\xi)&=\underset{(y,\eta)\in(x,\xi)+\{e\}\times\hcG}{\essupp}\abs{\overline{\la\eta,y\ra}\lim_{\al\in A}f_\al(y)}=\underset{(y,\eta)\in\{x\}\times\hcG}{\essupp}\abs{\lim_{\al\in A}f_\al(y)}\\
		&=\abs{\lim_{\al\in A}f_\al(x)}=\abs{V_\f f(x,\xi)}.
	\end{align*}
	Therefore
	\begin{equation*}
		\norm{V_\f f}_{W_Q(L^{p,q}_{m})}=\norm{\sfM_QV_\f f}_{L^{p,q}_{m}}=\norm{V_\f f}_{L^{p,q}_m}.
	\end{equation*}
	\emph{Case $\cG$ compact.} 
	The argument is identical to the previous one, take $\cK=\cG$ and $Q\coloneqq \cG\times\{\hat{e}\}$.	
\end{proof}

\begin{corollary}
	Suppose $\cG$ is discrete or compact. Consider $m\in\cM_v(\cG\times\hcG)$ and $0<p,q\leq\infty$. Then
	\begin{equation*}
		M^{p,q}_m(\cG)
		=\{f\in S'_0(\cG)\,|\,V_gf\in L^{p,q}_{m}(\cG\times\hcG)\}
	\end{equation*}
	and
	\begin{equation}
		\norm{f}_{M^{p,q}_m}
		\asymp\norm{V_g f}_{L^{p,q}_{m}},
	\end{equation}
	for some $g\in\cS_\cC(\cG)\smallsetminus\{0\}$.
\end{corollary}
\begin{proof}
	We consider $M^{p,q}_m(\cG)$ as a subspace of $S'_0(\cG)$ instead of $\cR_{\tilde{v}}$ (Proposition \ref{Pro-Independence-Coo-Reservoir}). The claim then follows from the continuous embedding $W(L^{p,q}_{\tilde{m}})\hookrightarrow L^{p,q}_{\tilde{m}}$, Lemma \ref{Lem-Analogo-GalSam-Lem3.2} and Remark \ref{Rem-WaveletSchrodinger}.
\end{proof}
\section{Continuity of the Rihaczek and Kohn-Nirenberg operators}
In this section we investigate the continuity of the Rihaczek distribution \eqref{Rdef} on modulation spaces and infer boundedness results for the Kohn-Nirenberg operators, defined in \eqref{EqDefKN}.
\subsection{Boundedness results} We first study the boundedness of the Rihaczek distribution on modulation spaces. The techniques are mainly borrowed from \cite[Theorem 3.1]{Wignersharp2018} and \cite[Theorem 4]{Cor2020} for the Wigner distribution on $\rd$.\\
From now on we shall mainly work with $S_0(\cG)$ and $S'_0(\cG)$ instead of $\cT_{\tilde{v}}$ and $\cR_{\tilde{v}}$ (Proposition \ref{Pro-Independence-Coo-Reservoir}).
Preliminarly, we exhibit a proof for Young's inequality in $L^{p,q}(\cG\times\hcG)$ and some generalizations. This result is folklore, but no explicit proof is available according to authors' knowledge. 
	 
\begin{proposition}\label{P1} Consider  $1\leq p_i,q_i, r_i\leq\infty$, $i=1,2$, such that
	\begin{equation}\label{indici}
\frac{1}{p_i}+\frac{1}{q_i}=1+\frac{1}{r_i},\quad i=1,2.
	\end{equation}
	If $F\in L^{p_1,p_2}(\cG\times\hcG)$ and $H\in L^{q_1,q_2}(\cG\times\hcG)$, then $F\ast H\in L^{r_1,r_2}(\cG\times\hcG)$ with
	\begin{equation}\label{Young-in}
	\|F\ast H\|_{L^{r_1,r_2}}\leq \|F\|_{L^{p_1,p_2}}\|H\|_{L^{q_1,q_2}}.
	\end{equation}
\end{proposition}
\begin{proof}
	We follow the pattern of \cite[Part II, Theorem 1, b)]{BenPan61}. It suffices to prove the claim for $F,H\geq 0$. Given a measurable function $W\colon\cG\times\hcG\to\bC$ and $1\leq s\leq\i$, we define the (measurable) function on $\hcG$
	\begin{equation}
		\norm{W}_{(s)}(\xi)\coloneqq\begin{cases}
		 \left(\int_{\cG}\abs{W(x,\xi)}^s\,dx\right)^{\frac1s}\quad\text{if}\quad s<\i,\\
		\essupp_{x\in\cG}\abs{W(x,\xi)}\quad\text{if}\quad s=\i.
		\end{cases}
	\end{equation}
	We show the case $r_1<\i$, the case $r_1=\i$ is done similarly. In the following we shall use  Minkowski's integral inequality  (see \cite[Appendix A.1]{Stein70}):  
	\begin{align*}
		\norm{F\ast H}_{(r_1)}(\xi)
		&=\left(\int_\cG\left[\int_{\cG\times\hcG}F((x,\xi)-(u,\o))H(u,\o)\,dud\o\right]^{r_1}\,dx\right)^{\frac1{r_1}}\nonumber\\
		&=\left(\int_\cG\left[\int_{\hcG}\left(\int_\cG F((x,\xi)-(u,\o))H(u,\o)\,du\right)\,d\o\right]^{r_1}\,dx\right)^{\frac1{r_1}}\nonumber\\
		&\leq\int_{\hcG}\left(\int_{\cG}\left[\int_\cG F((x,\xi)-(u,\o))H(u,\o)\,du\right]^{r_1}\,dx\right)^{\frac1{r_1}}\,d\o\\
		&=\int_{\hcG}\left(\int_{\cG}\left[ [F(\cdot,\xi-\o)\ast H(\cdot,\o)](x)\right]^{r_1}\,dx\right)^{\frac1{r_1}}\,d\o\nonumber\\
		&=\int_{\hcG}\norm{F(\cdot,\xi-\o)\ast H(\cdot,\o)}_{L^{r_1}(\cG)}\,d\o:=I
		\end{align*}
		 Using Young's inequality (see \cite[Theorem 20.18]{HewRos63}) with indexes $p_1,q_1,r_1$ as  in \eqref{indici} we majorize as
		\begin{align*}
		I&\leq\int_{\hcG}\norm{F(\cdot,\xi-\o)}_{L^{p_1}(\cG)}\norm{ H(\cdot,\o)}_{L^{q_1}(\cG)}\,d\o\\
		&=\int_{\hcG}\left(\int_\cG F(x,\xi-\om)^{p_1}\,dx\right)^{\frac1{p_1}}\left(\int_\cG H(x,\om)^{q_1}\,dx\right)^{\frac1{q_1}}\,d\o\\
		&=\int_{\hcG}\norm{F}_{(p_1)}(\xi-\o)\norm{H}_{(q_1)}(\o)\,d\o\\
		&=\left(\norm{F}_{(p_1)}\ast\norm{H}_{(q_1)}\right)(\xi)
	\end{align*}
	Using Young's inequality with indices $p_2,q_2,r_2$ in  \eqref{indici} we obtain the desire result. Namely, 
	\begin{align*}
		\norm{F\ast H}_{L^{r_1,r_2}(\cG\times\hcG)}
		&=\left(\int_{\hcG}\left[\norm{F\ast H}_{(r_1)}(\xi)\right]^{r_2}\,d\xi\right)^{\frac1{r_2}}\nonumber\\
		&\leq\left(\int_{\hcG}\left[\left(\norm{F}_{(p_1)}\ast\norm{H}_{(q_1)}\right)(\xi)\right]^{r_2}\,d\xi\right)^{\frac1{r_2}}\nonumber\\
		&=\norm{\norm{F}_{(p_1)}\ast\norm{H}_{(q_1)}}_{L^{r_2}(\hcG)}\nonumber\\
		&\leq\norm{\norm{F}_{(p_1)}}_{L^{p_2}(\hcG)}\norm{\norm{H}_{(q_1)}}_{L^{q_2}(\hcG)}\label{UsoYoung2}\\
		&=\norm{F}_{L^{p_1,p_2}(\cG\times\hcG)}\norm{\norm{H}_{(q_1)}}_{L^{q_1,q_2}(\cG\times\hcG)}\nonumber.
	\end{align*}
	This concludes the proof.
\end{proof}
A straightforward consequence is the weighted Young's inequality below.
\begin{corollary}
	Consider  $1\leq p_i,q_i, r_i\leq\infty$, $i=1,2$, such that
	\begin{equation}
	\frac{1}{p_i}+\frac{1}{q_i}=1+\frac{1}{r_i},\quad i=1,2.
	\end{equation}
	Consider $m\in\cM_v(\cG\times\hcG)$. If $F\in L^{p_1,p_2}_m(\cG\times\hcG)$ and $H\in L^{q_1,q_2}_v(\cG\times\hcG)$, then $F\ast H\in L^{r_1,r_2}_m(\cG\times\hcG)$ with
	\begin{equation}
	\|F\ast H\|_{L^{r_1,r_2}_m}\leq \|F\|_{L^{p_1,p_2}_m}\|H\|_{L^{q_1,q_2}_v}.
	\end{equation}
\end{corollary}
Note that  Proposition \ref{P1} can be easily generalized to  $N$ indices, $N\geq 2$, as in \cite[Part II, Theorem 1, b)]{BenPan61}:
\begin{proposition}
	Consider $N\in\bN$ and let $\cG_i$ be a LCA, $\si$-finite group with Haar measure $dx_i$, $i=1,\dots,N$. Consider  $1\leq p_i,q_i, r_i\leq\infty$, $i=1,\dots,N$, such that
	\begin{equation*}
	\frac{1}{p_i}+\frac{1}{q_i}=1+\frac{1}{r_i},\quad i=1,\dots,N.
	\end{equation*}
	If $F\in L^{p_1,\dots,p_N}(\cG_1\times\dots\times\cG_N)$ and $H\in L^{q_1,\dots,q_N}(\cG_1\times\dots\times\cG_N)$, then $F\ast H\in L^{r_1,\dots,r_N}(\cG_1\times\dots\times\cG_N)$ with
	\begin{equation*}
	\|F\ast H\|_{L^{r_1,\dots,r_N}(\cG_1\times\dots\times\cG_N)}\leq \|F\|_{L^{p_1,\dots,p_N}(\cG_1\times\dots\times\cG_N)}\|H\|_{L^{q_1,\dots,q_N}(\cG_1\times\dots\times\cG_N)},
	\end{equation*}
	where the product LCA $\si$-finite group $\cG_1\times\dots\times\cG_N$ is endowed with the product Haar measure $dx_1\dots dx_N$.
\end{proposition}

We need to  extend  \cite[formula (51)]{GroStr2007} to wider classes of functions. Namely,
\begin{lemma}\label{lemmae2}
		Consider  $\psi\in\mathcal{S_\mathcal{C}}(\cG)$ and  $f,g\in S'_0(\cG)$. Then
	\begin{equation}\label{Eq-GroStro-formula-51}
	V_{R(\psi,\psi)}R(g,f)((x,\xi),(\o,u))=\overline{\<\xi,u\>}V_{\psi}g(x,\xi+\o)\overline{V_{\psi}f(x+u,\xi)},
	\end{equation}
	with $x,u\in\cG$ and $\xi,\o\in\hcG$. 
\end{lemma}
\begin{proof}
For $f,g,\psi\in\mathcal{S_\mathcal{C}}(\cG)$ formula \eqref{Eq-GroStro-formula-51} is proved in  \cite[formula (51)]{GroStr2007}. Consider now $f,g\in S'_0(\cG)$. From Corollary \ref{Cor-Sc-w*-dense-S'0} there exist nets $(f_\al)_{\al\in A},(g_\al)_{\al\in A}\in S_0(\cG)$ which converge pointwisely to $f$ and $g$ in $S'_0(\cG)$. Therefore for every $x,u\in\cG$ and $\xi\in\hcG$,
\begin{equation*}
\lim_{\al\in A}V_{\p}f_\al(x+u,\xi)=\lim_{\al\in A}\la f_\al,\pi(x+u,\xi)\p\ra=\la f,\pi(x+u,\xi)\p\ra=V_{\p}f(x+u,\xi),
\end{equation*}
and similarly for $V_\psi g$. For the left-hand side of \eqref{Eq-GroStro-formula-51}, observe that
\begin{equation*}
R(f_\al,g_\al)(x,\xi)=\overline{\la\xi,x\ra}\cF_2(f_\al\otimes\overline{g}_\al)(x,\xi).
\end{equation*}
 The partial Fourier transform $\cF_2$ is a topological isomorphism from $S_0(\cG\times\cG)$ onto $S_0(\cG\times\hcG)$ and from $S'_0(\cG\times\cG)$ onto $S'_0(\cG\times\hcG)$. Write $\bfx=(x,\xi)$ and $\boldsymbol{\o}=(\o,u)$,
\begin{align*}
\lim_{\al\in A} V_{R(\p,\p)}R(f_\al,g_\al)(\bfx,\boldsymbol{\o})&=\lim_{\al\in A}\la\overline{\la\cdot{,}\cdot\ra}\cF_2(f_\al\otimes\overline{g}_\al),\pi(\bfx,\boldsymbol{\o})R(\p,\p)\ra\\
&=\la\overline{\la\cdot{,}\cdot\ra}\cF_2(f\otimes\overline{g}),\pi(\bfx,\boldsymbol{\o})R(\p,\p)\ra\\
&=V_{R(\p,\p)}R(f,g)(\bfx,\boldsymbol{\o}),
\end{align*}
being $R(\p,\p)\in S_0(\cG\times\hcG)$. This concludes the proof.
\end{proof}
\begin{proposition}\label{Pro-continuity-R}
	Consider $p,q,p_i,q_i\in(0,\infty]$, $i=1,2$, such that
	\begin{align}
	p_i,q_i&\leq q,\quad i=1,2;\label{Eq-continuity-R-condition-1}\\
	\min\set{\frac1{p_1}+\frac1{p_2},\frac1{q_1}+\frac1{q_2}}&\geq \frac1p+\frac1q.\label{Eq-continuity-R-condition-2}
	\end{align}
	Let $v$ be a even submultiplicative weight bounded from below  on $\cG\times\hcG$, and $\cJ$  the isomorphism in \eqref{DefJ}. For $g\in M^{p_1,q_1}_v(\cG)$ and $f\in M^{p_2,q_2}_v(\cG)$, we have $R(g,f)\in M^{p,q}_{1\otimes v\circ\cJ^{-1}}(\cG\times\hcG)$, with
	\begin{equation}
	\norm{R(g,f)}_{M^{p,q}_{1\otimes v\circ\cJ^{-1}}}\lesssim \norm{g}_{M^{p_1,q_1}_v}\norm{f}_{M^{p_2,q_2}_v}.
	\end{equation}
\end{proposition}
\begin{proof} 
	Consider  $\psi\in\mathcal{S_\mathcal{C}}(\cG)$, $f\in M^{p_2,q_2}_v(\cG)$, $g\in M^{p_1,q_1}_v(\cG)$. By Lemma \ref{lemmae2} the STFT of the Rihaczek distribution is given by 
	\begin{equation}
	V_{R(\psi,\psi)}R(g,f)((x,\xi),(\o,u))=\overline{\<\xi,u\>}V_{\psi}g(x,\xi+\o)\overline{V_{\psi}f(x+u,\xi)},
	\end{equation}
with $x,u\in\cG$ and $\xi,\o\in\hcG$.   Corollary \ref{Cor-Rfg-windows}  shows that $R(\psi,\psi)\in\sA_{\widetilde{1\otimes v\circ\cJ^{-1}}}(\cG\times\hcG)$. Consider $V_\cG\subseteq\cG$ and $V_{\hcG}\subseteq\hcG$ open, relatively compact, unit neighbourhoods. According to the notation  in 
	\eqref{Eq-Decomposition-V} we define
	\begin{align}
		V_{1,2}=V_\cG\times V_{\hcG},\quad V_{2,1}\coloneqq V_{\hcG}\times V_\cG,\quad  O\coloneqq V_{1,2}\times V_{2,1}\times\bT.
	\end{align}
	Set $$H_g((x,\xi),(\o,u),\tau)\coloneqq V_{\psi}g(x,\xi+\o)\quad \mbox{and}\quad H_f((x,\xi),(\o,u),\tau)\coloneqq\overline{V_{\psi}f(x+u,\xi)},$$ which are functions on the Heisenberg group associated to $\cG\times\hcG$. Notice
	\begin{equation*}
	\sfM_O[\overline{\tau}V_{R(\psi,\psi)}R(g,f)]=\sfM_O[H_g\cdot H_f]\leq \sfM_O [H_g]\cdot \sfM_O[H_f].
	\end{equation*}
We compute
	\begin{align*}
	\sfM_O [H_g]((x,\xi),(\o,u),\tau)&=\underset{\substack{((y,\eta),(\nu,z),s)\in\\ ((x,\xi),(\o,u),\tau)O}}{\essupp}\abs{V_{\psi}g(y,\eta+\nu)}\\
	&=\essupp_{\nu\in\o+V_{\hcG}}\essupp_{(y,\eta)\in(x,\xi)+V_{1,2}}\abs{T_{(e,-\nu)}V_{\psi}g(y,\eta)}\\
	&=\essupp_{\nu\in\o+V_{\hcG}} \left(\sfM_{V_{1,2}}[T_{(e,-\nu)}V_\psi g](x,\xi)\right)\\
	&=\essupp_{\nu\in\o+V_{\hcG}} \left(T_{(e,-\nu)}[\sfM_{V_{1,2}}V_\psi g(x,\xi)]\right)\\
	&=\essupp_{\nu\in\o+V_{\hcG}} \left(\sfM_{V_{1,2}}V_\psi g(x,\xi+\nu)\right).
	\end{align*}
	Similarly,
	\begin{align*}
	\sfM_O [H_f]((x,\xi),(\o,u),\tau)&=\essupp_{z\in u+V_{\cG}} \left(T_{(-z,\hat{e})}[\sfM_{V_{1,2}}V_\psi f(x,\xi)]\right)\\
	&=\essupp_{z\in u+V_{\cG}} \left(\sfM_{V_{1,2}}V_\psi f(x+z,\xi)\right).
	\end{align*}
	By the  modulation spaces  independence of the window in 
	$\sA_{\widetilde{1\otimes v\circ\cJ^{-1}}}(\cG\times\hcG)$, we can write
	\begin{align*}
		\norm{R(g,f)}_{M^{p,q}_{1\otimes v\circ\cJ^{-1}}}&\asymp\bigg(\int_{\hcG\times\cG\times\bT}\bigg(\int_{\cG\times\hcG}\abs{\sfM_O[\overline{\tau}V_{R(\psi,\psi)}R(g,f)]((x,\xi),(\o,u),\tau)}^p dxd\xi\bigg)^{\frac{q}{p}}\\
		&\qquad\qquad\qquad\qquad\times v^q\circ\cJ^{-1}(\o,u)\,d\o du d\tau\bigg)^{\frac1q}\\
		&\leq\bigg(\int_{\hcG\times\cG}\bigg(\int_{\cG\times\hcG}\essupp_{\nu\in\o+V_{\hcG}} \left(\sfM_{V_{1,2}}V_\psi g(x,\xi+\nu)\right)^p\\
		&\qquad\qquad\times\essupp_{z\in u+V_{\cG}} \left(\sfM_{V_{1,2}}V_\psi f(x+z,\xi)\right)^p\,dxd\xi\bigg)^{\frac{q}{p}} v^q\circ\cJ^{-1}(\o,u)\,d\o du\bigg)^{\frac1q}\\
		&=\bigg(\!\!\int_{\hcG\times\cG}\!\!\bigg(\!\!\essupp_{(\nu,z)\in(\o,u)+V_{2,1}}\!\int_{\cG\times\hcG}\!\!\!\!\!\sfM_{V_{1,2}}V_\psi g(x,\xi+\nu)^p \sfM_{V_{1,2}}V_\psi f(x+z,\xi)^p\,dxd\xi\bigg)^{\frac{q}{p}}\\
		&\qquad\qquad\times v^q\circ\cJ^{-1}(\o,u)\,d\o du\bigg)^{\frac1q}.
	\end{align*}
	The inner integral can be rephrased  using the left-right invariance of Haar measure  and the involution $h^\ast(\cdot)\coloneqq \overline{h(-\cdot)}$ as follows:
	\begin{align*}
	\int_{\cG\times\hcG}&\sfM_{V_{1,2}}V_{\psi}g(x,\xi+\nu)^p\sfM_{V_{1,2}}V_{\psi}f(x+z,\xi)^p\,dxd\xi\\
	&=\int_{\cG\times\hcG}\sfM_{V_{1,2}}V_{\psi}g(x',\xi')^p\sfM_{V_{1,2}}V_{\psi}f((x',\xi')+(z,-\nu))^p\,dx'd\xi' \\
	&=\int_{\cG\times\hcG}(\sfM_{V_{1,2}}V_{\psi}g)^\ast(x'',\xi'')^p \sfM_{V_{1,2}}V_{\psi}f((z,-\nu)-(x'',\xi''))^p\,dx''d\xi''\\
	&=(\sfM_{V_{1,2}}V_{\psi}g)^\ast\,^p\ast (\sfM_{V_{1,2}}V_{\psi}f)^p(z,-\nu)\\
	&=(\sfM_{V_{1,2}}V_{\psi}g)^\ast\,^p\ast (\sfM_{V_{1,2}}V_{\psi}f)^p\circ\cJ^{-1}(\nu,z).
	\end{align*}
	Whence, using \cite[Lemma 2.3.23]{Voig2015}, we majorize 
	\begin{align*}
		\essupp_{(\nu,z)\in(\o,u)+V_{2,1}}&(\sfM_{V_{1,2}}V_{\psi}g)^\ast\,^p\ast (\sfM_{V_{1,2}}V_{\psi}f)^p\circ\cJ^{-1}(\nu,z)\\
		&=\essupp_{(z',\nu')\in\cJ^{-1}(\o,u)+\cJ^{-1}V_{2,1}}(\sfM_{V_{1,2}}V_{\psi}g)^\ast\,^p\ast (\sfM_{V_{1,2}}V_{\psi}f)^p(z',\nu')\\
		&=\sfM_{\cJ^{-1}V_{2,1}}[(\sfM_{V_{1,2}}V_{\psi}g)^\ast\,^p\ast (\sfM_{V_{1,2}}V_{\psi}f)^p](\cJ^{-1}(\o,u))\\
		&\leq [\sfM_{\cJ^{-1}V_{2,1}}[(\sfM_{V_{1,2}}V_{\psi}g)^\ast\,^p]\ast (\sfM_{V_{1,2}}V_{\psi}f)^p](\cJ^{-1}(\o,u))\\
		&=[\sfM_{\cJ^{-1}V_{2,1}}[(\sfM_{V_{1,2}}V_{\psi}g)^\ast\,^p]\ast (\sfM_{V_{1,2}}V_{\psi}f)^p]\circ\cJ^{-1}(\o,u).
	\end{align*}
	Setting $U\coloneqq -\cJ^{-1}V_{2,1}+V_{1,2}$, which is  an open, relatively compact, unit neighbourhood, we obtain
	\begin{align*}
	\sfM_{\cJ^{-1}V_{2,1}}[(\sfM_{V_{1,2}}V_{\psi}g)^\ast\,^p](u,\o)&
	=\essupp_{(y,\eta)\in(u,\o)+\cJ^{-1}V_{2,1}}\essupp_{(x,\xi)\in-(y,\eta)+V_{1,2}}\abs{V_{\psi}g(x,\xi)}^p\\
	&\leq\essupp_{(y,\eta)\in(u,\o)+\cJ^{-1}V_{2,1}}\underset{\substack{(x,\xi)\in\\-(u,\o)-\cJ^{-1}V_{2,1}+V_{1,2}}}{\essupp}\abs{V_{\psi}g(x,\xi)}^p\\
	&=\underset{\substack{(x,\xi)\in\\-(u,\o)-\cJ^{-1}V_{2,1}+V_{1,2}}}{\essupp}\abs{V_{\psi}g(x,\xi)}^p\\
	&=\left(\sfM_U V_\psi g(-u,-\o)\right)^p
	=\left([\sfM_U V_\psi g(u,\o)]^\ast\right)^p.
	\end{align*}
	Observe that for positive functions $h,l$ on $\cG\times\hcG$ and $v$ a submultiplicative weight we can write
	\begin{equation}
		\left((h\ast l)v\right)(x,\xi)\leq \left(hv\ast lv\right)(x,\xi),\qquad(x,\xi)\in\cG\times\hcG,
	\end{equation}
	moreover $v^p$ is submultiplicative as well.
	Therefore 
	\begin{align*}
		\norm{R(g,f)}_{M^{p,q}_{1\otimes v\circ\cJ^{-1}}}&\lesssim\bigg(\int_{\hcG\times\cG}\bigg([([\sfM_U V_\psi g]^\ast)^p\ast (\sfM_{V_{1,2}}V_{\psi}f)^p]\circ\cJ^{-1}(\o,u)\bigg)^{\frac{q}{p}}\\
		&\qquad\qquad\times v^q\circ\cJ^{-1}(\o,u)\,d\o du\bigg)^{\frac1q}\\
		&\leq\left(\int_{\hcG\times\cG}\!\!\left(([\sfM_U V_\psi g]^\ast\cdot v)^p\ast (\sfM_{V_{1,2}}V_{\psi}f\cdot v)^p\right)^{\frac{q}{p}}\circ\cJ^{-1}(\o,u)\,d\o du\right)^{\frac1q}\\
		&=\norm{([\sfM_U V_\psi g]^\ast\cdot v)^p\ast (\sfM_{V_{1,2}}V_{\psi}f\cdot v)^p}_{L^{q/p}(\cG\times\hcG)}^{1/p}.
	\end{align*}
	By Young's convolution inequality and following the same argumenta as in the proofs in \cite[Theorem 3.1]{Wignersharp2018} and \cite[Theorem 4]{Cor2020} for the Euclidean case (replacing the Wigner distribution with the Rihaczek) we infer the estimate
	\begin{equation}
		\norm{R(g,f)}_{M^{p,q}_{1\otimes v\circ\cJ^{-1}}}\lesssim \norm{g}_{M^{p_1,q_1}_v}\norm{f}_{M^{p_2,q_2}_v},
	\end{equation}
	with indices satisfying the conditions \eqref{Eq-continuity-R-condition-1} and \eqref{Eq-continuity-R-condition-2}. Following the patterns of \cite{Cor2020,Wignersharp2018} the same result is obtained when $p=\infty$ or $q=\infty$.
\end{proof}

The boundedness properties of the Rihaczek distributions enter the study of  Kohn-Nirenberg pseudo-differential operators $\Opz(\si)$, defined in \eqref{EqDefKN} and \eqref{Eq-KonNirenberg-weak}, in the same fashion of \cite{GroStr2007}.

The boundedness  result  for Weyl operators in the Euclidean setting \cite[Theorem 5.1]{Wignersharp2018} can be written for Kohn-Nirenberg operators on groups as follows.
\begin{theorem}
	Consider $p,q,p_i,q_i\in[1,\infty]$, $i=1,2$, such that:
	\begin{align}
	q&\leq\min\{p'_1,q'_1,p_2,q_2\};\\
	\min\set{\frac1{p_1}+\frac1{p'_2},\frac1{q_1}+\frac1{q'_2}}&\geq \frac1{p'}+\frac1{q'}.
	\end{align}
	Consider $v$ submultiplicative weight even and bounded from below  on $\cG\times\hcG$. If $\si\in M^{p,q}_{1\otimes \frac1v\circ\cJ^{-1}}(\cG\times\hcG)$, then $\Opz(\si)$ is a bounded operator from $M^{p_1,q_1}_v(\cG)$ into $M^{p_2,q_2}_{1/v}(\cG)$ with estimate
	\begin{equation}
	\norm{\Opz(\si) f}_{M^{p_2,q_2}_{1/v}}\lesssim\norm{\si}_{M^{p,q}_{1\otimes \frac1v\circ\cJ^{-1}}}\norm{f}_{M^{p_1,q_1}_v}.
	\end{equation}
\end{theorem}
\begin{proof}
	It 
	follows by duality using Proposition \ref{Pro-continuity-R} and the weak definition of $\Opz(\si)$ in \eqref{Eq-KonNirenberg-weak}.
\end{proof}

\subsection{Gabor frames on quasi-lattices}
The key tool in the boundedness properties of Kohn-Nirenberg  operators
on quasi-Banach modulation spaces is the Gabor frame theory. For a detailed treatment of frame theory see, e.g.,  \cite{Olebook2003}.

In what follows we shall recall and prove new properties for Gabor frames on a LCA group. As a byproduct, we generalize the convolution relations for modulation spaces firstly given in \cite[Proposition 3.1]{BasCorNic20}, see Proposition \ref{Pro-convolution-Mpq}.\par
A {\slshape lattice} in $\cG$ is a discrete subgroup $\La$ such that the quotient group $\cG / \La$ is compact. In this case, there is a relatively compact $U\subseteq\cG$, called {\slshape fundamental domain} for $\La$, such that 
\begin{equation*}
\cG=\bigcup_{w\in\La}\left(w+U\right),\qquad\left(w+U\right)\cap\left(u+U\right)=\varnothing\quad\mbox{for}\quad w\neq u.
\end{equation*}
Not every LCA admits a lattice, for example the $p$-adic groups $\mathbb{Q}_p$, therefore we adopt the following construction of Gr\"{o}chenig and Strohmer \cite{GroStr2007}. Recall from the structure theorem $\cG\cong\rd\times\cG_0$.  
Consider  $D\subseteq \cG_0$   a collection of coset representatives of $\cG_0 / \cK$ and $A\in \text{GL}(\rd)$. We define $U\coloneqq A [0,1)^d\times\cK$. The discrete set $\La\coloneqq A\zd\times D$ is called {\slshape quasi-lattice} with {\slshape fundamental domain} $U$.  Observe that we have the following partition
\begin{equation*}
\cG=\bigcup_{w\in\La}\left(w+U\right).
\end{equation*} 
According to the above theory, a quasi-lattice on the phase-space $\cG\times\hcG$ is of the  type:
\begin{equation}\label{Eq-quasi-lattice-phase-space}
\La\coloneqq\La_1\times\La_2\coloneqq\left(A_1\zd\times D_1\right)\times\left(A_2\zd\times D_2\right)\cong A_{1,2}\zdd\times D_{1,2}
\end{equation}
with fundamental domain 
\begin{equation}\label{Eq-fund-domain-phase-space}
U\coloneqq U_1\times U_2\coloneqq\left(A_1[0,1)^d\times\cK\right)\times\left(A_2 [0,1)^d\times\cK^\perp\right)\cong A_{1,2}[0,1)^{2d}\times\left(\cK\times\cK^\perp\right),
\end{equation}
where $D_2\subseteq\hcG_0$ is a set of coset representatives of $\hcG_0/\cK^\perp$ and
\begin{equation}\label{Eq-block-matrix-quasi-lattice}A_{1,2}\coloneqq
\begin{bmatrix}
A_1&0\\
0&A_2
\end{bmatrix},\qquad D_{1,2}\coloneqq D_1\times D_2.
\end{equation}
We shall denote elements of a quasi-lattice $\La$ in $\cG\times\hcG$ as
\begin{equation*}
\bfw=(w,\mu)=((w_1,w_2),(\mu_1,\mu_2))\in\La=\La_1\times\La_2\subseteq\cG\times\hcG.
\end{equation*}


\begin{lemma}
	Let $\La\subseteq\cG\times\hcG$ be a quasi-lattice as in \eqref{Eq-quasi-lattice-phase-space}. Then $\La$ is a relatively separated family.
\end{lemma}
\begin{proof}
	We use Lemma \ref{Lem-product-rel-sep-fam}. The fact that $A_{1,2}\zdd$ is relatively separated in $\rdd$ is trivial. We only have to show that $D_1$ is relatively separated in $\cG_0$; $D_2$  is treated similarly.
	For a fixed compact set $Q_0\subseteq\cG_0$ we have to show that
	\begin{equation*}
	C_{D_1,Q_0}=\sup_{x\in D_1}\#\{y\in D_1\,|\,\left(x+Q_0\right)\cap\left( y+Q_0\right)\neq\varnothing\}<+\infty.
	\end{equation*}
	Since $Q_0$ is compact  and $\cK$ is an open subgroup, there exist $q_1,\dots,q_n\in Q_0$ such that
	\begin{equation*}
	Q_0\subseteq\bigcup_{i=1}^n\left(q_i+\cK\right)\eqqcolon Q'_0.
	\end{equation*}
	For $x,y\in D_1$, if 
$\left(x+Q_0\right)\cap\left( y+Q_0\right)\neq\varnothing$ then   $\left(x+Q'_0\right)\cap\left( y+Q'_0\right)\neq\varnothing$, hence $C_{D_1,Q_0}\leq C_{D_1,Q'_0}$. Assume that $\left(x+Q'_0\right)\cap\left( y+Q'_0\right)\neq\varnothing$, then there are $i_0,j_0\in\{1,\dots,n\}$ and $k_{i_0},k_{j_0}\in\cK$ such that
	\begin{equation*}
	x+q_{i_0}+k_{i_0}=y+q_{j_0}+k_{j_0}\quad\Leftrightarrow\quad y=x+q_{i_0}-q_{j_0}+k_{i_0}-k_{j_0}.
	\end{equation*}
	Fix $x\in D_1$, quotienting by $\cK$,
	\begin{equation}
	[y]^\bullet=[x+q_{i_0}-q_{j_0}]^\bullet\quad\Rightarrow\quad\#\{y\in D_1\,|\,\left(x+Q'_0\right)\cap\left( y+Q'_0\right)\neq\varnothing\}\leq n^2,
	\end{equation}
	where $[y]^\bullet$ denotes the projection of $y\in\cG_0$ onto the quotient $\cG_0/\cK$.
	This proves $C_{D_1,Q_0}\leq C_{D_1,Q'_0}<+\infty$. The desired result follows now from Lemma \ref{Lem-product-rel-sep-fam}.
\end{proof}

\begin{corollary}
	Let $\La\subseteq\cG\times\hcG$ be a quasi-lattice. Then $\La$ is at most countable.
\end{corollary}
\begin{proof}
	We use the fact that $\La$ is a relatively separated family and \cite[Lemma 2.3.10]{Voig2015}.
\end{proof}

In the following issue about the existence of a particular BUPU, we use the quasi-lattice $\La$ both as localizing family and as indexes' set. The argument was presented in \cite[Remark 2.5]{fegr85}.
\begin{lemma}\label{Lem-BUPU-ad-hoc}
	Let $\La\subseteq\cG\times\hcG$ be a quasi-lattice as in \eqref{Eq-quasi-lattice-phase-space} with fundamental domain $U$ as in \eqref{Eq-fund-domain-phase-space}. Then there exist two open, relatively compact, unit neighbourhoods $Q$ and $V_{1,2}$ in $\cG\times\hcG$, where $V_{1,2}$ as in \eqref{Eq-Decomposition-V}, such that $Q\subsetneq V_{1,2}$ and there is a $V_{1,2}$-BUPU 
	\begin{equation*}
	\{\psi_w\otimes\gamma_\mu\}_{(w,\mu)\in\La}
	\end{equation*}  
	with localizing family $\La$ and such that for every $(w,\mu)\in\La$
	\begin{equation*}
	\psi_w\otimes\gamma_\mu\equiv1\qquad\mbox{on}\qquad(w,\mu)+Q.
	\end{equation*}
\end{lemma}

\begin{remark}
	Without loss of generality, the unit neighbourhood $Q$ of the previous lemma can be chosen  such that
	\begin{equation}
	\left(\{0_{\rd}\}\times\cK\right)\times\left(\{0_{\rd}\}\times\cK^\perp\right)\subsetneq Q.
	\end{equation}
	Therefore for every $(w,\mu)\in\La=\La_1\times\La_2$ we have
	\begin{equation*}
	\psi_w\equiv 1\quad\mbox{on}\quad w+\left(\{0_{\rd}\}\times\cK\right),\qquad\ga_\mu\equiv 1\quad\mbox{on}\quad \mu+\left(\{0_{\rd}\}\times\cK^\perp\right).
	\end{equation*}
\end{remark}

\begin{definition}\label{frame} Given a quasi-lattice $\La\subseteq\cG\times\hcG$ and windows $g,h\in L^2(\cG)$,  the {\slshape Gabor system generated by $g$} is
\begin{equation*}
\{\pi(\bfw)g\,|\,\bfw\in\La\}=\{\pi(\bfw)g\}_{\bfw\in\La}.
\end{equation*}
The {\slshape coefficient} or {\slshape analysis operator} is given by
\begin{equation}
\cC_g\colon L^2(\cG)\to\ell^2(\La),f\mapsto \left(\la f,\pi(\bfw)g\ra\right)_{\bfw\in\La}.
\end{equation}
Its adjoint is called {\slshape reconstruction} or {\slshape synthesis operator} and has the form
\begin{equation}
\cC^\ast_g\colon\ell^2(\La)\to L^2(\cG), \left(c_{\bfw}\right)_{\bfw\in\La}\mapsto\sum_{\bfw\in\La}c_{\bfw}\pi(\bfw)g.
\end{equation}
The {\slshape Gabor frame operator}  $S_{h,g}$ is given by
\begin{equation}
S_{h,g}f= \cC^\ast_h \cC_g f=\sum_{\bfw\in\La}\la f,\pi(\bfw)g\ra\pi(\bfw)h.
\end{equation}
We say that $\{\pi(\bfw)g\}_{\bfw\in\La}$ is a {\slshape Gabor frame for $L^2(\cG)$}  if there exist $A,B>0$ such that 
\begin{equation}
A\norm{f}^2_{L^2}\leq\sum_{\bfw\in\La}\abs{\la f,\pi(\bfw)g\ra}^2\leq B\norm{f}^2_{L^2},\quad \forall f\in L^2(\cG);
\end{equation}
this is equivalent to saying that $S_{g,g}$ is invertible on $L^2(\cG)$. If $A=B$ the frame is called {\slshape tight}. Moreover, if $h\in L^2(\cG)$ is such that
\begin{equation}
	S_{h,g}=S_{g,h}=I_{L^2},
\end{equation}
then $h$ is named  {\slshape dual window} for the frame $\{\pi(\bfw)g\}_{\bfw\in\La}$.\\
\end{definition}
We note that  Theorem 2.7 in \cite{GroStr2007} is till valid for the case of the  Gaussian $\f$ 
and considering a Gabor frame not tight. Namely,
\begin{theorem}\label{Th-frame-L2}
	Let $\La\coloneqq \alpha\zdd\times D_{1,2}$, $\alpha\in(0,1)$, be a quasi-lattice in $\cG\times\hcG$
	. Consider  the Gaussian $\f$ in \eqref{Gauss}. Then 
	\begin{equation}\label{Eq-Gabor-frame-L2}
	\{\pi(\bfw)\f\,|\,\bfw\in\La\}
	\end{equation}
	is a Gabor frame for $L^2(\cG)$.
\end{theorem} 

\begin{corollary}\label{Cor-dual-window-in-sA_v}
	There exists $\alpha\in(0,1)$ such that the Gabor frame $\{\pi(\bfw)\f\,|\,\bfw\in\La\}$ defined in \eqref{Eq-Gabor-frame-L2} admits a dual window $h\in\sA_{\tilde{v}}$.
\end{corollary}
\begin{proof}
	We first tackle the problem of finding a dual window. The proof is similar to that in  \cite[Theorem 2.7]{GroStr2007}. We distinguish three cases.\par
	\emph{Case $\cG=\rd$.} In this case the frame we are considering is
	\begin{equation}
	\{\pi(w_1,\mu_1)e^{-\pi x_1^2},\,(w_1,\mu_1)\in\alpha\zdd\}, \quad \alpha\in(0,1).
	\end{equation}
	We fix $\alpha$ such that $\alpha^{2d}<(d+1)^{-1}$. Then the existence of a dual window $\ga_0$ for the Gabor frame generated by the first Hermite function $H_0$ (the Gaussian) was proved by Gr\"{o}chenig and Lyubarskii, see \cite{GroLyu2007,GroLyu2009}. In particular in \cite[Remarks 2]{GroLyu2009} 
	was observed that $\ga_0$ belongs to the Gelfand-Shilov space $\cS^{1/2}_{1/2}(\rd)$, cf. \cite{GelShi68}.\\
	\emph{Case $\cG=\cG_0$.} In this case the frame that we are dealing with is the orthonormal basis for $L^2(\cG_0)$
	\begin{equation}
	\{\pi(w_2,\mu_2)\chi_{\cK}(x_2),\,(w_2,\mu_2)\in D_{1,2}\}.
	\end{equation}
	Therefore  $\chi_{\cK}$ is a dual window itself.\\
	\emph{Case $\cG\cong\rd\times\cG_0$.} The frame in this case is the tensor product of the previous ones:
	\begin{equation}
	\{\pi(\bfw)\f,\,\bfw=((w_1,w_2),(\mu_1,\mu_2))\in \La=\left(\alpha\zd\times D_{1}\right)\times\left(\alpha\zd\times D_2\right)\},
	\end{equation}
	where $\f(x_1,x_2)=e^{-\pi x_1^2}\chi_{\cK}(x_2)=(\f_1\otimes\f_2)(x_1,x_2)$.
	Recall that the functions of the type $f_1\otimes f_2$, with $f_1\in L^2(\rd)$ and $f_2\in L^2(\cG_0)$, are dense in $L^2(\rd\times\cG_0)$. 
	Let us show that
	\begin{equation}
	h(x_1,x_2)\coloneqq (\ga_0\otimes \chi_{\cK})(x_1,x_2)
	\end{equation}
	is a dual window. In fact,
	\begin{align*}
	\sum_{\bfw\in\La}&\la f_1\otimes f_2, \pi(\bfw)\f\ra \pi(\bfw)\ga_0\otimes \chi_{\cK}\\
	&=\sum_{(w_1,\mu_1)}\la f_1, \pi(w_1,\mu_1)\f_1\ra\pi(w_1,\mu_1)\gamma_0\sum_{(w_2,\mu_2)}\la f_2,\pi(w_2,\mu_2)\f_2\ra\pi(w_2,\mu_2)\chi_{\cK}\\
	&=f_1\otimes f_2;
	\end{align*}
	similarly,
	\begin{equation*}
	\sum_{\bfw\in\La}\la f_1\otimes f_2, \pi(\bfw)\ga_0\otimes \chi_{\cK}\ra\pi(\bfw)\f=f_1\otimes f_2.
	\end{equation*}
	The claim follows by density argument.\par
	We now prove that $h\in\sA_{\tilde{v}}$ in the general case $\cG\cong\rd\times\cG_0$. Similarly to the wavelet transform of the generalized Gaussian $\f$ in \eqref{Gauss}, see \eqref{Eq-SFTFgauss}, we obtain 
	\begin{equation}
	W^\vr_h h(x,\xi,\tau)=\overline{\tau}c(\cK) V_{\ga_0}\ga_0\otimes\chi_{\cK\times\cK^\perp}(x,\xi).
	\end{equation}
	Since $V_{\ga_0}\ga_0\in\cS^{1/2}_{1/2}(\rdd)$, see e.g. \cite[Theorem 2.13]{BasTeo21}
	,  calculations similar to the ones performed in Lemma \ref{Lem-SC-window-space} yield the desired result. 
\end{proof}

\begin{lemma}\label{Lem-Gabor-frame-Gaussian-pi/2}
	Let $\La=\alpha\zdd\times D_{1,2}$, $\alpha\in\bR$, be a quasi-lattice in $\cG\times\hcG$
	. Consider the function
	\begin{equation}\label{Eq-Gaussian-pi/2}
	\f^\circ(x)\coloneqq \f^\circ(x_1,x_2)\coloneqq 2^{-\frac{d}{2}}\mbox{meas}(\cK)e^{-\frac{\pi}{2}x^2_1}\otimes\chi_{\cK}(x_2)\in\sA_{\tilde{v}}
	\end{equation}
	for $x=(x_1,x_2)\in\rd\times\cG_0\cong\cG$, where $\mbox{meas}(\cK)$ is the (finite) measure of the compact open closed subgroup $\cK$ in $\cG_0$. Then there exist $\al\in(0,1)$ and a function $h^\circ\in\sA_{\tilde{v}}$ such that
	\begin{equation}\label{Eq-Gabor-frame-L2-pi/2}
	\{\pi(\bfw)\f^\circ\,|\,\bfw\in\La\}
	\end{equation}
	is a Gabor frame for $L^2(\cG)$ with dual window $h^\circ$.
\end{lemma}
\begin{proof}
	The result is obtained using the same arguments as in Theorem \ref{Th-frame-L2} and Corollary \ref{Cor-dual-window-in-sA_v}, combined with \cite[Lemma 3.2.2]{CordRod2020}.
\end{proof}

\begin{theorem}\label{Th-coeff-op-continuity}
	Let $\La\subseteq\cG\times\hcG$ be a quasi-lattice 
	with fundamental domain $U$
	. Consider $0<p,q\leq\infty$, $m\in\cM_v(\cG\times\hcG)$ and $g\in\sA_{\tilde{v}}$. Then the coefficient operator $\cC_g$ admits a unique continuous and linear extension
	\begin{equation}
	\cC_g\colon M^{p,q}_m(\cG)\to \ell^{p,q}_{m_\La}(\La),
	\end{equation}
	where $m_\La$ is the restriction of $m$ to $\La$. Moreover, if $0<\delta\leq\infty$ is such that $0<\delta\leq\min\{p,q\}\leq\infty$, then there is a constant $C=C(\delta)>0$, such that
	\begin{equation*}
		\opnorm{\cC_g}_{M^{p,q}_m\to \ell^{p,q}_{m_\La}}\leq C
	\end{equation*}
	for all $p,q\geq\delta$. The constant $C=C(\delta)$ may depend on other elements, but not on $p$ and $q$.
\end{theorem}
\begin{proof}
	Consider $f\in M^{p,q}_m(\cG)$. Let $\{\psi_w\otimes\gamma_\mu\}_{(w,\mu)\in\La}$ be the BUPU on $\cG\times\hcG$ constructed in Lemma \ref{Lem-BUPU-ad-hoc}. Since tensor product of BUPUs is a BUPU (Lemma \ref{Lem-BUPU-tensor}) it follows that $\{\psi_w\otimes\gamma_\mu\otimes\chi_\bT\}_{(w,\mu)\in\La}$ is a $V$-BUPU on $\bH_\cG$, $V$ as in \eqref{Eq-V-open-neigh}, with localizing family $\mathfrak{X}=\La\times\{1\}$ and such that
	\begin{equation}
	(\psi_w\otimes\gamma_\mu\otimes\chi_\bT)(w,\mu,1)=1\qquad\forall\,(w,\mu)\in\La.
	\end{equation}
	Hence
	\begin{align*}
	\abs{\la f,\pi(w,\mu)g\ra}
	=\abs{(\psi_w\otimes\gamma_\mu\otimes\chi_\bT)(w,\mu,1)\cdot W^\vr_g f(w,\mu,1)}
	\leq\norm{\left(\psi_w\otimes\gamma_\mu\otimes\chi_\bT\right)\cdot W^\vr_g f}_{L^\infty}.
	\end{align*}
	By Lemma \ref{Lem-Lpq-disc-ell-p-q},
	\begin{align*}
	\norm{\cC_g f}_{\ell^{p,q}_{m_\La}(\La)}&=\norm{\left(\la f,\pi(w,\mu)g\ra\right)_{(w,\mu)\in\La}}_{\ell^{p,q}_{m_\La}(\La)}\\
	&\leq\norm{\left(\norm{\left(\psi_w\otimes\gamma_\mu\otimes\chi_\bT\right)\cdot W^\vr_g f}_{L^\infty}\right)_{(w,\mu)\in\La}}_{\ell^{p,q}_{m_\La}(\La)}\\
	&\asymp\norm{\left(\norm{\left(\psi_w\otimes\gamma_\mu\otimes\chi_\bT\right)\cdot W^\vr_g f}_{L^\infty}\right)_{(w,\mu)\in\La}}_{(L^{p,q}_{\tilde{m}})_d(\mathfrak{X},V)}\\
	&\asymp\norm{W^\vr_g f}_{W(L^{p,q}_{\tilde{m}})}=\norm{f}_{M^{p,q}_m}, 
	\end{align*}
	where in the last equivalence we used Lemma \ref{Lem-Lpq-Wiener-discrete-equivalence}, see also \eqref{Eq-discrete-equiv-norm-Wiener} in the Appendix A. The last claim comes from Lemma \ref{Lem-Lpq-Wiener-discrete-equivalence} and Corollary \ref{Cor-equivalence-constants-pq>delta}.
\end{proof}

\begin{theorem}\label{Th-synth-op-continuity}
	Let $\La\subseteq\cG\times\hcG$ be a quasi-lattice 
	with fundamental domain $U$
	. Consider $0<p,q\leq\infty$, $m\in\cM_v(\cG\times\hcG)$ and $g\in\sA_{\tilde{v}}$. Then the synthesis operator $\cC^\ast_g$ admits a unique continuous and linear extension
	\begin{equation}
	\cC^\ast_g\colon \ell^{p,q}_{m_\La}(\La)\to M^{p,q}_m(\cG),
	\end{equation}
	where $m_\La$ is the restriction of $m$ to $\La$. If $p,q\neq\infty$, then the series representing $\cC^\ast_g(c)$ converges unconditionally in $M^{p,q}_m(\cG)$. Otherwise $\cC^\ast_g(c)$ w-$\ast$-converges in $M^\infty_{1/v}(\cG)$. Moreover, if $0<\delta\leq\infty$ is such that $0<\delta\leq\min\{p,q\}\leq\infty$, then there is a constant $C=C(\delta)>0$, such that
	\begin{equation*}
		\opnorm{\cC^\ast_g}_{\ell^{p,q}_{m_\La} \to M^{p,q}_m}\leq C
	\end{equation*}
	for all $p,q\geq\delta$. The constant $C=C(\delta)$ may depend on other elements, but not on $p$ and $q$.
\end{theorem}
\begin{proof}
	The proof follows the pattern displayed in \cite{Galperin2004}. Let $(x,\xi,\tau)\in\bH_\cG$ and $c=\left(c_{\bfw}\right)_{\bfw\in\La}\in\ell^{p,q}_{m_\La}(\La)$, 
	then we write
	\begin{align*}
	\abs{W^\vr_g [\cC^\ast_g(c)](x,\xi,\tau)}&
	=\abs{V_g\left[\sum_{\bfw\in\La}c_\bfw\pi(\bfw)g\right](x,\xi)}=\abs{\sum_{\bfw\in\La}c_\bfw V_g\pi(\bfw)g(x,\xi)}\\
	&\leq \sum_{\bfw\in\La}\abs{c_\bfw} \abs{T_\bfw V_g g(x,\xi)}\eqqcolon F^g_c(x,\xi,\tau).
	\end{align*}
	Let $\{\psi_w\otimes\gamma_\mu\}_{(w,\mu)\in\La}$ be the $V_{1,2}$-BUPU on $\cG\times\hcG$ constructed in Lemma \ref{Lem-BUPU-ad-hoc}. Then $\{\psi_w\otimes\gamma_\mu\otimes\chi_\bT\}_{(w,\mu)\in\La}$ is a $V$-BUPU, $V$ as in \eqref{Eq-V-open-neigh}, on $\bH_\cG$ with localizing family $\mathfrak{X}=\La\times\{1\}$. Using the norm equivalence in \eqref{Eq-discrete-equiv-norm-Wiener} (Appendix A) and Lemma \ref{Lem-Lpq-disc-ell-p-q}
	\begin{align*}
	\norm{\cC^\ast_g(c)}_{M^{p,q}_m}\asymp\norm{\cC^\ast_g(c)}_{W_{V}(L^{p,q}_{\tilde{m}})}&\lesssim\norm{\left(\norm{\left(\psi_w\otimes\ga_\mu\otimes\chi_\bT\right)\cdot F^g_c}_{L^\infty}\right)_{\bfw\in\La}}_{(L^{p,q}_{\tilde{m}})_d(\mathfrak{X},V)}\\
	&\asymp\norm{\left(\norm{\left(\psi_w\otimes\ga_\mu\otimes\chi_\bT\right)\cdot F^g_c}_{L^\infty}\right)_{\bfw\in\La}}_{\ell^{p,q}_{m_{\La}}(\La)}.
	\end{align*}
	We control the latter sequence as follows:
	\begin{align*}
	\norm{\left(\psi_w\otimes\ga_\mu\otimes\chi_\bT\right)\cdot F^g_c}_{L^\infty}&
	\leq\sum_{\bfu\in\La}\abs{c_\bfu}\essupp_{(x,\xi)\in\bfw+V_{1,2}}\abs{T_\bfu V_g g(x,\xi)}\\
	&=
	\sum_{\bfw\in\La}\abs{c_\bfw}\sfM_{V_{1,2}}V_g g(\bfw-\bfu)\\
	&=\Big((\abs{c_\bfu})_{\bfu}\ast(\sfM_{V_{1,2}}V_g g(\bfu))_{\bfu}\Big)(\bfw).
	\end{align*}
	We set $t=\min\{1,p\}$ and $s=\min\{1,p,q\}$. Using the convolution relations for the sequences' spaces in \cite[Lemma 2.7]{Galperin2004}, we obtain
	\begin{align*}
	\norm{\cC^\ast_g(c)}_{M^{p,q}_m}&\lesssim\norm{\Bigg(\Big((\abs{c_\bfu})_{\bfu}\ast(\sfM_{V_{1,2}}V_g g(\bfu))_{\bfu}\Big)(\bfw)\Bigg)_{\bfw\in\La}}_{\ell^{p,q}_{m_{\La}}(\La)}\\
	&\lesssim\norm{c}_{\ell^{p,q}_{m_\La}(\La)}\norm{(\sfM_{V_{1,2}}V_g g(\bfw))_{\bfw\in\La.}}_{\ell^{t,s}_{v_\La}(\La)}.
	\end{align*}
	Arguing as in the proof of Theorem \ref{Th-coeff-op-continuity} and using Lemma \ref{Lem-Lpq-disc-ell-p-q} and \eqref{Eq-discrete-equiv-norm-Wiener}  again
	\begin{align*}
	\norm{(\sfM_{V_{1,2}}V_g g(\bfw))_{\bfw\in\La}}_{\ell^{a,b}_{v_\La}(\La)}&\leq\norm{\left(\norm{\left(\psi_w\otimes\ga_\mu\otimes\chi_\bT\right)\cdot \sfM_{V}W^\vr_g g}_{L^\infty}\right)_{\bfw\in\La}}_{\ell^{t,s}_{v_\La}(\La)}\\
	&\asymp\norm{\left(\norm{\left(\psi_w\otimes\ga_\mu\otimes\chi_\bT\right)\cdot \sfM_{V}W^\vr_g g}_{L^\infty}\right)_{\bfw\in\La}}_{(L^{t,s}_{\tilde{v}})_d(\mathfrak{X},V)}\\
	&\asymp\norm{\sfM_{V}W^\vr_g g}_{W_{V}(L^{t,s}_{\tilde{v}})}
	=\norm{\sfM_{V}\sfM_{V}W^\vr_g g}_{L^{t,s}_{\tilde{v}}}\\
	&\leq\norm{\sfM_{V^2}W^\vr_g g}_{L^{t,s}_{\tilde{v}}}
	=\norm{W^\vr_g g}_{W_{V^2}(L^{t,s}_{\tilde{v}})},
	\end{align*}
	where we set $V^2\coloneqq VV$ (multiplicative notation in $\bH_\cG$). As reported in Remark \ref{Rem-A-subset-G}, for any $0<r\leq 1$ we have  the continuous inclusion
	\begin{equation}
	W^R(L^\infty,W(L^\infty,L^r_{\tilde{v}}))\hookrightarrow W(L^\infty, L^r_{\tilde{v}}).
	\end{equation}
	Arguing as in Proposition \ref{Pro-Incl-Mpq} and taking $r<\min\{t,s\}$ we obtain 
	\begin{equation}
	W(L^\infty, L^r_{\tilde{v}})\hookrightarrow W(L^\infty, L^{t,s}_{\tilde{v}}).
	\end{equation}
	The fact that $g$ is in $\sA_{\tilde{v}}$ (defined in \eqref{Eq-maximal-window-space-ALL-modulation-spaces}) implies then
	\begin{equation*}
	\norm{W^\vr_g g}_{W_{V^2}(L^{t,s}_{\tilde{v}})}<+\infty
	\end{equation*}
	and 
	\begin{equation}
	\norm{\cC^\ast_g(c)}_{M^{p,q}_m}\lesssim\norm{c}_{\ell^{p,q}_{m_\La}(\La)}.
	\end{equation}
	Unconditional convergence for the series defining $\cC^\ast_g(c)$ in $M^{p,q}_m(\cG)$ if $p,q\neq\infty$, and w-$\ast$-convergence in $M^\infty_{1/v}(\cG)$ otherwise, is inferred as in \cite[Theorem 12.2.4]{Grochenig_2001_Foundations}. The last claim comes from Lemma \ref{Lem-Lpq-Wiener-discrete-equivalence} and Corollary \ref{Cor-equivalence-constants-pq>delta}.
\end{proof}

\begin{theorem}\label{Th-gabor-expansions-Mpq}
	Let $0<p,q\leq\infty$, $m\in\cM_v(\cG\times\hcG)$ and $\f$ as in \eqref{Gauss}. Consider $h\in\sA_{\tilde{v}}$ such that
	\begin{equation}
	S_{h,\f}=S_{\f,h}=I_{L^2},
	\end{equation}
	for a suitable quasi-lattice $\La=\La_1\times\La_2\subseteq\cG\times\hcG$. Then
	\begin{equation}\label{Eq-gabor-expansions-Mpq}
	f=\sum_{\bfw\in\La}\la f, \pi(\bfw)\f\ra \pi(\bfw) h=\sum_{\bfw\in\La}\la f, \pi(\bfw)h\ra \pi(\bfw) \f
	\end{equation}
	with unconditional convergence in $M^{p,q}_m(\cG)$ if $p,q\neq\infty$, and w-$\ast$-convergence in $M^\infty_{1/v}(\cG)$ otherwise. Moreover, for every $f\in M^{p,q}_m(\cG)$ we have the following quasi-norm equivalences:
	\begin{align}\label{Eq-disc-equiv-quasinorms-gabor-Mpq}
	\norm{f}_{M^{p,q}_m}&\asymp\left(\sum_{\mu\in\La_2}\left(\sum_{w\in\La_1}\abs{V_\f f(w,\mu)}^{p}m(w,\mu)^p\right)^\frac{q}{p}\right)^{\frac1q}=\norm{\left(V_\f f(\bfw)\right)_{\bfw\in\La}}_{\ell^{p,q}_{m_\La}(\La)},\\
	\norm{f}_{M^{p,q}_m}&\asymp\left(\sum_{\mu\in\La_2}\left(\sum_{w\in\La_1}\abs{V_hf(w,\mu)}^{p}m(w,\mu)^p\right)^\frac{q}{p}\right)^{\frac1q}=\norm{\left(V_h f(\bfw)\right)_{\bfw\in\La}}_{\ell^{p,q}_{m_\La}(\La)},\notag
	\end{align}
	and similarly if $p=\infty$ or $q=\infty$.
\end{theorem}
\begin{proof}
	The proof is based on the continuity of $\cC_\f$, $\cC^\ast_\f$, $\cC_h$ and $\cC^\ast_h$. The pattern is the  same of \cite[Corollary 12.2.6]{Grochenig_2001_Foundations}.
\end{proof}

Expansions and equivalences analogous to \eqref{Eq-gabor-expansions-Mpq} and \eqref{Eq-disc-equiv-quasinorms-gabor-Mpq} hold for $\f^\circ$ and $h^\circ$ defined in Lemma \ref{Lem-Gabor-frame-Gaussian-pi/2}.

\begin{proposition}\label{Pro-convolution-Mpq}
	Consider $m\in\cM_v(\cG\times\hcG)$, define for $x\in\cG$ and $\xi\in\hcG$
	\begin{equation}
	m_1(x)\coloneqq m(x,\hat{e}),
	\quad v_1(x)\coloneqq v(x,\hat{e}),\quad v_2(\xi)\coloneqq v(e,\xi).
	\end{equation}
	Let $\nu (\xi )>0$ be an arbitrary  weight function on $\hcG$ such that
	\begin{equation}
	m_1\otimes\nu, v_1\otimes v_2\nu^{-1}\in\cM_v(\cG\times\hcG).
	\end{equation}
	Let  $0<p,q,r,t,u,\gamma\leq\infty$, with
	\begin{equation}\label{Holderindices}
	\frac 1u+\frac 1t=\frac 1\gamma,
	\end{equation}
	and 
	\begin{equation}
	\frac1p+\frac1q=1+\frac1r,\quad \,\, \text{ for } \, 1\leq r\leq \infty
	\end{equation}
	whereas
	\begin{equation}
	p=q=r,\quad \,\, \text{ for } \, 0<r<1.
	\end{equation}
	Then 
	\begin{equation}\label{mconvm}
	M^{p,u}_{m_1\otimes \nu}(\cG)\ast  M^{q,t}_{v_1\otimes
		v_2\nu^{-1}}(\cG)\hookrightarrow M^{r,\gamma}_m(\cG)
	\end{equation}
	with  quasi-norm inequality 
	\begin{equation}
	\norm{f\ast g}_{M^{r,\ga}_m}\lesssim \norm{f}_{M^{p,u}_{m_1\otimes \nu}}\norm{g}_{ M^{q,t}_{v_1\otimes
			v_2\nu^{-1}}}.
	\end{equation}
\end{proposition}
\begin{proof}
	We follow the patter displayed in \cite[Proposition 3.1]{BasCorNic20}. A direct computation gives $\f\ast\f=\f^\circ$, where $\f$ is defined in \eqref{Gauss} and $\f^\circ$ in \eqref{Eq-Gaussian-pi/2}. Similarly, the following identities can be easily checked:
	\begin{equation*}
	V_{h}f(x,\xi)=\overline{\la\xi,x\ra}\left(f\ast M_\xi[h^\ast]\right)(x),\quad M_\xi[\f^{\circ\ast}](x)=\left(M_\xi[\f^\ast]\ast M_\xi[\f^\ast]\right)(x)
	\end{equation*}
	(recall the involution $h^\ast(x)=\overline{h(-x)}$). Using associativity and commutativity of the convolution product we can write
	\begin{equation}
	V_{\f^\circ}(f\ast g)(x,\xi)=\overline{\la\xi,x\ra}\left(\left(f\ast M_\xi[\f^\ast]\right)\ast \left(g\ast M_\xi[\f^\ast]\right)\right)(x).
	\end{equation}
	In what follows we will use the frame expansions in Theorem \ref{Th-gabor-expansions-Mpq} with  $\f^\circ$ in place of $\f$, see Lemma \ref{Lem-Gabor-frame-Gaussian-pi/2}. We majorize the weight $m$ by
	\begin{equation*}
	m(\bfw)=m(w,\mu)\lesssim m(w,\hat{e})v(e,\mu)=m_1(w)v_2(\mu)\qquad\bfw=(w,\mu)\in\La,
	\end{equation*}
	use  Young's convolution inequality for sequences in the $w$-variable  and
	H\"older's one in the $\mu$-variable. The indices $p,q,r,\gamma,t,u$ fulfil the equalities in the assumptions.  We show in details the case when $r,\gamma,t,u<\infty$. The others are similar. Namely,
	\begin{align*}
	\norm{f\ast g}_{M^{r,\gamma}_m}&\asymp  \norm{\left((V_{\f^\circ}(f\ast h))(\bfw)m(\bfw)\right)_{\bfw\in\La}}_{\ell^{r,\gamma}(\La)}\\
	&=\left(\sum_{\mu\in\La_2}\left(\sum_{w\in\La_1}\abs{V_{\f^\circ}(f\ast g)(w,\mu)}^r m(w,\mu)^r\right)^{\frac{\ga}{r}}\right)^{\frac{1}{\ga}}\\
	&\lesssim \left( \sum_{\mu\in\La_2} \left( \sum_ {w\in\La_1} \abs{(f\ast M_{\mu}[\f^\ast]) \ast (g\ast
		M_{\mu}[\f^\ast])(w)}^r m_1(w)^r \right) ^{\frac{\gamma}{r}} v_2(\mu) ^\gamma\right)^{\frac1\gamma} \\
	&=\left(\sum_{\mu\in\La_2}\norm{\left((f\ast M_{\mu}[\f^\ast]) \ast (g\ast
		M_{\mu}[\f^\ast])(w)\right)_{w\in\La_1}}_{\ell^r_{m_1}(\La_1)}^\gamma v_2(\mu)^\gamma\right)^{\frac1\gamma}\\
	&\lesssim\left(\sum_{\mu\in\La_1}\norm{\left((f\ast M_{\mu}[\f^\ast])(w)\right)_{w\in\La_1}}_{\ell^p_{m_1}(\La_1)}^\gamma\norm{\left((g\ast
		M_{\mu}[\f^\ast])(w)\right)_{w\in\La_1}}_{\ell^q_{v_1}(\La_1)}^\gamma\right.\\
	&\qquad\qquad\times \left.v_2(\mu)^\gamma\frac{\nu(\mu)^\gamma}{\nu(\mu)^\gamma}\right)^{\frac1\gamma}\\
	&\leq \left(\sum_{\mu\in\La_2}\norm{\left((f\ast M_{\mu}[\f^\ast])(w)\right)_{w\in\La_1}}_{\ell^p_{m_1}(\La_1)}^u\nu(\mu)^u\right)^{\frac1u}\\
	&\times \left(\sum_{\mu\in\La_2}\norm{\left((g\ast M_{\mu}[\f^\ast])(w)\right)_{w\in\La_1}}_{\ell^q_{m_1}(\La_1)}^t\frac{v_2(\mu)^t}{\nu(\mu)^t}\right)^{\frac1t}\\
	&=\norm{\left(V_\f f(\bfw)\right)_{\bfw\in\La}}_{\ell^{p,u}_{m_1\otimes\nu}(\La)}\norm{\left(V_\f g(\bfw)\right)_{\bfw\in\La}}_{\ell^{q,t}_{m_1\otimes v_2\nu^{-1}}(\La)}\\
	&\asymp\norm{f}_{M^{p,u}_{m_1\otimes \nu}}\norm{g}_{ M^{q,t}_{v_1\otimes
			v_2\nu^{-1}}},
	\end{align*}
	the last equivalence is  \eqref{Eq-disc-equiv-quasinorms-gabor-Mpq}. This concludes the proof.
\end{proof}

Let us introduce the closed and compact subgroups of $\cG\times\hcG$ and $\hcG\times\cG$, respectively:
\begin{equation}
\bU(\cG)\coloneqq\left(\{0_{\rd}\}\times\cK\right)\times\left(\{0_{\rd}\}\times\cK^\perp\right),\,\,\bU(\hcG)\coloneqq\left(\{0_{\rd}\}\times\cK^\perp\right)\times\left(\{0_{\rd}\}\times\cK\right).
\end{equation}
Given  $\bfx\in\cG\times\hcG$, we will denote its projection on $(\cG\times\hcG)/\bU(\cG)$  by
\begin{equation*}
\overset{\bullet}{\bfx}\quad\mbox{or}\quad[\bfx]^\bullet,
\end{equation*}
and similarly for the projection of $\boldsymbol{\xi}\in\hcG\times\cG$ onto $(\hcG\times\cG)/\bU(\hcG)$.\\
Let  $\La=A_{1,2}\bZ^{2d}\times D_{1,2}\subseteq\cG\times\hcG$ and $\Ga=A_{3,4}\bZ^{2d}\times D_{3,4}\subseteq\hcG\times\cG$ be  quasi-lattices, then their projections
\begin{equation}
\sD(\cG)\coloneqq\sD(\cG,A_{1,2})\coloneqq\overset{\bullet}{\La}\qquad\mbox{and}\qquad\sD(\hcG)\coloneqq\sD(\hcG,A_{3,4})\coloneqq\overset{\bullet}{\Ga}
\end{equation}
are discrete and  at most countable LCA groups. Given a distribution $f$ in $\cR_{\tilde{v}}$, or $S'_0$, and a window $g\in\sA_{\tilde{v}}$, the function
\begin{equation}\label{Eq-quotient-STFT}
\overset{\bullet}{V}_g f(\overset{\bullet}{\bfx})\coloneqq\sup_{\bfz\in\bU(\cG)}\abs{V_g f(\bfx+\bfz)}=\sfM_{\bU(\cG)}V_g f(\bfx)
\end{equation}
is well defined on the quotient group $(\cG\times\hcG)/\bU(\cG)$. In fact, if $\bfu$ is such that $\overset{\bullet}{\bfx}=\overset{\bullet}{\bfu}$, then there exists $\bfn\in\bU(\cG)$ such that $\bfu=\bfx+\bfn$. Setting $\bfy=\bfn+\bfz\in\bU(\cG)$ we have
\begin{equation*}
\sup_{\bfz\in\bU(\cG)}\abs{V_g f(\bfu+\bfz)}=\sup_{\bfz\in\bU(\cG)}\abs{V_g f(\bfx+\bfn+\bfz)}=\sup_{\bfy\in\bU(\cG)}\abs{V_g f(\bfx+\bfy)}.
\end{equation*}
Similarly, given a weight $m\in\cM_v(\cG\times\hcG)$, the function
\begin{equation}
\overset{\bullet}{m}(\overset{\bullet}{\bfx})\coloneqq\sup_{\bfz\in\bU(\cG)}m(\bfx+\bfz)
\end{equation}
is well defined on the quotient. 

\begin{lemma}\label{Lem-quotient-coefficient-operator}
	Consider a quasi-lattice $\La$ in $\cG\times\hcG$. Let $g\in\sA_{\tilde{v}}$, $0<p,q\leq\infty$, $m\in\cM_v(\cG\times\hcG)$ and define the mapping
	\begin{equation}
	\overset{\bullet}{\cC}_g\colon M^{p,q}_m(\cG)\to\ell^{p,q}_{\overset{\bullet}{m}}(\sD(\cG)),\,f\mapsto\left(\overset{\bullet}{V}_g f(\overset{\bullet}{\bfw})\right)_{\overset{\bullet}{\bfw}\in\sD(\cG)},
	\end{equation}
	where the weight $\overset{\bullet}{m}$ is understood to be restricted on $\sD(\cG)$. Then there exists a constant $C>0$ such that for every $f\in M^{p,q}_m(\cG)$ we have
	\begin{equation}
	\norm{\overset{\bullet}{\cC}_g f}_{\ell^{p,q}_{\overset{\bullet}{m}}(\sD(\cG))}\leq C\norm{f}_{M^{p,q}_m}.
	\end{equation}
\end{lemma}
\begin{proof}
	The BUPU $\{\psi_w\otimes\ga_\mu\otimes\chi_\bT,\,\bfw=(w,\mu)\in\La\}$ coming from Lemma \ref{Lem-BUPU-ad-hoc} is such that
	\begin{equation*}
	\psi_w\otimes\ga_\mu\equiv 1\qquad\mbox{on}\qquad\bfw+\bU(\cG).
	\end{equation*}
	Noticing that the projection of $\La$ onto $\sD(\cG)$ is one-to-one we have without ambiguity
	\begin{equation*}
	\overset{\bullet}{V}_g f(\overset{\bullet}{\bfw})\leq\norm{\left(\psi_w\otimes\ga_\mu\otimes\chi_\bT\right)\cdot V_g f}_{L^\infty}=\norm{\left(\psi_w\otimes\ga_\mu\otimes\chi_\bT\right)\cdot W^\vr_g f}_{L^\infty},
	\end{equation*}
	where  $(w,\mu)$ is the only representative of $\overset{\bullet}{\bfw}$ in the quasi-lattice.
	Since  $\bU(\cG)$ is compact   there exists  a constant $C=C(\bU(\cG),v)>0$ such that
	\begin{equation}
	\frac1C m(\bfx+\bfz)\leq m(\bfx)\leq C m(\bfx+\bfz),
	\end{equation}
	for every $\bfx\in\cG\times\hcG$ and $\bfz\in\bU(\cG)$, see   \cite[Corollary 2.2.23]{Voig2015}. For $\bfx=\bfw\in\La$, taking the supremum over $\bfz$ in $\bU(\cG)$ we can unambiguously  write
	\begin{equation}
	\overset{\bullet}{m}(\overset{\bullet}{\bfw})\asymp m(\bfw).
	\end{equation} 
	All together we have 
	\begin{align*}
	\norm{\overset{\bullet}{\cC}_g f}_{\ell^{p,q}_{\overset{\bullet}{m}}(\sD(\cG))}&=\norm{\left(\overset{\bullet}{V}_g f(\overset{\bullet}{\bfw})\cdot\overset{\bullet}{m}(\overset{\bullet}{\bfw})\right)_{\overset{\bullet}{\bfw}\in\sD(\cG)}}_{\ell^{p,q}(\sD(\cG))}\\
	&\lesssim\norm{\left(\norm{\left(\psi_w\otimes\ga_\mu\otimes\chi_\bT\right)\cdot W^\vr_g f}_{L^\infty}\cdot m(\bfw)\right)_{\bfw\in\La}}_{\ell^{p,q}(\La)}.
	\end{align*}
	Then we conclude as in the proof of Theorem \ref{Th-coeff-op-continuity}.
\end{proof}

\subsection{Eigenfunctions of Kohn-Nirenberg operators}
We have now all the instruments to study the eigenfunctions for Kohn-Nirenberg operators.
Let us first  introduce the Gabor matrix of $\Opz(\si)$.
\begin{definition}\label{Def-Gabor-matrix-KN}
	Consider $g\in\cS_\cC(\cG)$ and $\si\in S'_0(\cG\times\hcG)$. The {\slshape Gabor matrix} of the Kohn-Nirenberg operator $\Opz(\si)$ (with respect to $g$) is defined by
	\begin{equation}\label{Eq-Gabor-matrix-KN}
	[M(\si)]_{\bfx,\bfy}\coloneqq\la\Opz(\si)\pi(\bfy)g,\pi(\bfx)g\ra,\qquad\bfx,\bfy\in\cG\times\hcG.
	\end{equation}
\end{definition} 

The machinery developed in the previous subsection let us generalize what stated in \cite[Thereom 3.3 (i)]{BasCorNic20} for Weyl operators on $\rd$ and proved separately in  \cite[Theorem 4.3]{toft1} and \cite[Theorem 3.1]{ToftquasiBanach2017}. We will then obtain properties  for the eigenfunctions in $L^2(\cG)$ of $\Opz(\si)$ similar to the ones for Weyl operators on the Euclidean space, cf. \cite[Proposition 3.5]{BasCorNic20}.\par
We start with the boundedness properties of Weyl operators.
\begin{theorem}\label{Th-cont-KN-3-indici}
	Consider $0<p,q,\ga\leq\infty$ such that
	\begin{equation}\label{Eq-cont-KN-indexes}
	\frac1p+\frac1q=\frac1\ga
	\end{equation}
	and a symbol $\si\in M^{p,\min\{1,\ga\}}(\cG\times\hcG)$. Then Kohn-Nirenberg operator $\Opz(\si)$: $ S_0(\cG)\to  S'_0(\cG)$   admits a unique linear continuous  extension
	\begin{equation*}
	\Opz(\si)\colon M^q(\cG)\to M^\ga(\cG).
	\end{equation*}
\end{theorem}
\begin{proof}
	We distinguish two cases: $\ga\leq1$ and $\ga>1$.\\
	\emph{Case $\ga\leq1$.} Let $\f$ be as in \eqref{Gauss} and consider $h\in\sA_{\tilde{v}}$ and a quasi-lattice $\La$ such that $S_{h,\f}=S_{\f,h}=I_{L^2}$.
	Write
	\begin{equation}\label{Eq-Opz-composition}
	\Opz(\si)=\cC^\ast_h\circ\cC_\f\circ\Opz(\si)\circ\cC^\ast_\f\circ\cC_h\eqqcolon\cC^\ast_h\circ M(\si)\circ\cC_h.
	\end{equation}
	We shall prove that the Gabor matrix  $M(\si)$ is linear and continuous from $\ell^q(\La)$ into $\ell^\ga(\La)$. It is sufficient to prove that the diagram 
	\begin{equation*}
	\begin{diagram}
	\node{{M}^q}\arrow{s,l}{\cC_h} \arrow{c,t}{\Opz(\si)}  \node{ {M}^\gamma}
	\\
	\node{{\ell}^q}  \arrow{c,t}{M(\si)} \node{ {\ell}^\gamma} \arrow{n,r}{\cC^\ast_h}
	\end{diagram}
	\end{equation*}
	is commutative.	We show in detail the cases $p<+\infty$ and $q<+\infty$, the others are similar.
	For $f\in M^q(\cG)$, using the decomposition in  \eqref{Eq-Opz-composition} and the notation for the Gabor matrix \eqref{Eq-Gabor-matrix-KN}, we have
	\begin{align*}
	\Opz(\si)f&=\sum_{\bfw\in\La}\sum_{\bfu\in\La}\la\Opz(\si)\pi(\bfu)\f,\pi(\bfw)\f\ra\la f,\pi(\bfu)h\ra \pi(\bfw)h\\
	&=\sum_{\bfw\in\La}\sum_{\bfu\in\La}\left[M(\si)\right]_{\bfw,\bfu}\la f,\pi(\bfu)h\ra \pi(\bfw)h,
	\end{align*}
	so that
	\begin{equation*}
	M(\si)\colon\ell^q(\La)\to\ell^\ga(\La),\quad \left(c_{\bfw}\right)_{\bfw\in\La}\mapsto\left(\sum_{\bfu\in\La}\left[M(\si)\right]_{\bfw,\bfu}c_{\bfu}\right)_{\bfw\in\La}.
	\end{equation*}
	From the weak definition \eqref{Eq-KonNirenberg-weak} and \eqref{Rtfs} we can write each entry of the (discrete) Gabor matrix of $\Opz(\si)$ as follows:
	\begin{align*}
	\left[M(\si)\right]_{\bfw,\bfu}&=\la\Opz(\si)\pi(\bfu)\f,\pi(\bfw)\f\ra\\
	&=\la\si,R(\pi(\bfw)\f,\pi(\bfu)\f)\ra\\
	&=\la\si,\la\nu,w-u\ra M_{\cJ(\bfu-\bfw)}T_{(w,\nu)}R(\f,\f)\ra\\
	&=\overline{\la\nu,w-u\ra}V_{\Phi}\si\left((w,\nu),\cJ(\bfu-\bfw)\right),
	\end{align*}
	where $\bfw=(w,\mu)$, $\bfu=(u,\nu)$ and $\Phi\coloneqq R(\f,\f)\in\sA_{\tilde{v}}(\cG\times\hcG)$. We introduce the mapping 
	\begin{equation}
	\sfT_0\colon\left(\cG\times\hcG\right)\times\left(\cG\times\hcG\right)\to\cG\times\hcG,\quad \left((w,\mu),(u,\nu)\right)\mapsto(w,\nu)
	\end{equation}
	and write
	\begin{equation}
	\abs{\left[M(\si)\right]_{\bfw,\bfu}}=\abs{V_{\Phi}\si\left(\sfT_0(\bfw,\bfu),\cJ(\bfu-\bfw)\right)}.
	\end{equation}
	Since $\ga\leq1$, we have $\norm{c}_{\ell^1}\leq\norm{c}_{\ell^\ga}$ and we estimate
	\begin{align*}
	\norm{M(\si)c}_{\ell^\ga(\La)}&=\left(\sum_{\bfw\in\La}	\abs{\sum_{\bfu\in\La}\left[M(\si)\right]_{\bfw,\bfu}c_{\bfu}}^\ga\right)^\frac1\ga\\
	&\leq\left(\sum_{\bfw\in\La}\left(\sum_{\bfu\in\La}\abs{\left[M(\si)\right]_{\bfw,\bfu}}\abs{c_{\bfu}}\right)^\ga\right)^\frac1\ga\\
	&\leq\left(\sum_{\bfw\in\La}\sum_{\bfu\in\La}\abs{\left[M(\si)\right]_{\bfw,\bfu}}^\ga\abs{c_{\bfu}}^\ga\right)^\frac1\ga\\
	&=\left(\sum_{\bfw\in\La}\sum_{\bfu\in\La}\abs{V_{\Phi}\si\left(\sfT_0(\bfw,\bfu),\cJ(\bfu-\bfw)\right)}^\ga\abs{c_{\bfu}}^\ga\right)^\frac1\ga.
	\end{align*}
	Let us majorize each entry of the matrix as follows:
	\begin{align}
	\abs{V_{\Phi}\si\left(\sfT_0(\bfw,\bfu),\cJ(\bfu-\bfw)\right)}&\leq\sup_{\bfz\in\bU(\cG),\bfdel\in\bU(\hcG)}\abs{V_{\Phi}\si\left(\sfT_0(\bfw,\bfu)+\bfz,\cJ(\bfu-\bfw)+\bfdel\right)}\notag\\
	&=\overset{\bullet}{V}_\Phi \si\left(\left[\sfT_0(\bfw,\bfu)\right]^\bullet,\left[\cJ(\bfu-\bfw)\right]^\bullet\right)\label{Eq-STFT-sigma-almost-on-quotient},
	\end{align}
	where the function on the quotient group was introduced in \eqref{Eq-quotient-STFT}. Fix $\bfw,\bfu\in\La$ and consider $\bfx=(x,\xi),\,\bfy=(y,\eta)$ such that $\overset{\bullet}{\bfw}=\overset{\bullet}{\bfx}$ and $\overset{\bullet}{\bfu}=\overset{\bullet}{\bfy}$. Then there exist unique $\bfz=(z,\zeta)=((0,z_2),(0,\zeta_2)),\bfn=(n,\iota)=((0,n_2),(0,\iota_2))\in\bU(\cG)$ such that 
	\begin{equation*}
	\bfx=\bfw+\bfz,\qquad\bfy=\bfu+\bfn.
	\end{equation*} 
	Therefore
	\begin{align*}
	\sfT_0(\bfx,\bfy)&=\sfT_0(\bfw+\bfz,\bfu+\bfn)=((w_1,w_2+z_2),(\nu_1,\nu_2+\iota_2))\\
	&=\sfT_0(\bfw,\bfu)+((0,z_2),(0,\iota_2))
	\end{align*}
	where $((0,z_2),(0,\iota_2))\in\bU(\cG)$, so that we have shown
	\begin{equation}\label{Eq-T0-cosets}
	\overset{\bullet}{\bfw}=\overset{\bullet}{\bfx},\,\overset{\bullet}{\bfu}=\overset{\bullet}{\bfy}\qquad\Rightarrow\qquad\left[\sfT_0(\bfw,\bfu)\right]^\bullet=\left[\sfT_0(\bfx,\bfy)\right]^\bullet.
	\end{equation}
	Similarly,
	\begin{equation*}
	\cJ(\bfy-\bfx)=\cJ(\bfu+\bfn-\bfw-\bfz)=\cJ(\bfu-\bfw)+\cJ(\bfn-\bfz)
	\end{equation*}
	and being $\cJ(\bfn-\bfz)\in\bU(\hcG)$ we have proved
	\begin{equation}
	\overset{\bullet}{\bfw}=\overset{\bullet}{\bfx},\,\overset{\bullet}{\bfu}=\overset{\bullet}{\bfy}\qquad\Rightarrow\qquad\left[\cJ(\bfu-\bfw)\right]^\bullet=\left[\cJ(\bfy-\bfx)\right]^\bullet.
	\end{equation}
	Hence  the function in \eqref{Eq-STFT-sigma-almost-on-quotient} depends only on the cosets of $\bfw$ and $\bfu$, so that the application
	\begin{equation}
	\overset{\bullet}{H}(\overset{\bullet}{\bfu},\overset{\bullet}{\bfw})\coloneqq\overset{\bullet}{V}_\Phi \si\left(\left[\sfT_0(\bfw,\bfu)\right]^\bullet,\left[\cJ(\bfu-\bfw)\right]^\bullet\right)
	\end{equation} 
	is well defined. A sequence $c=(c_\bfw)_{\bfw\in\La}$ on the quasi-lattice $\La$ uniquely determines a sequence on $\sD(\cG)=\overset{\bullet}{\La}$ simply by
	\begin{equation}
	\overset{\bullet}{c}\coloneqq\left(c_{\overset{\bullet}{\bfw}}\coloneqq c_\bfw\right)_{\overset{\bullet}{\bfw}\in\sD(\cG)}
	\end{equation}
	with
	\begin{equation*}
	\norm{c}_{\ell^{q}(\La)}=\norm{\overset{\bullet}{c}}_{\ell^q(\sD(\cG))}.
	\end{equation*}
	Using H\"{o}lder's inequality in the $\overset{\bullet}{\bfu}$ variable (observe  $1/(p/\ga)+1/(q/\ga)=1$)  and the consideration above:
	\begin{align*}
	\norm{M(\si)c}_{\ell^\ga(\La)}&\leq\left(\sum_{\bfw\in\La}\sum_{\bfu\in\La}\overset{\bullet}{H}(\overset{\bullet}{\bfu},\overset{\bullet}{\bfw})^\ga\abs{c_{\bfu}}^\ga\right)^\frac1\ga\\
	&=\left(\sum_{\overset{\bullet}{\bfw}\in\sD(\cG)}\sum_{\overset{\bullet}{\bfu}\in\sD(\cG)}\overset{\bullet}{H}(\overset{\bullet}{\bfu},\overset{\bullet}{\bfw})^\ga\abs{c_{\overset{\bullet}{\bfu}}}^\ga\right)^\frac1\ga\\
	&\leq\left(\sum_{\overset{\bullet}{\bfw}\in\sD(\cG)}\left(\sum_{\overset{\bullet}{\bfu}\in\sD(\cG)}\overset{\bullet}{H}(\overset{\bullet}{\bfu},\overset{\bullet}{\bfw})^{\ga\frac{p}{\ga}}\right)^{\frac{\ga}{p}}\left(\sum_{\overset{\bullet}{\bfu}\in\sD(\cG)}\abs{c_{\overset{\bullet}{\bfu}}}^{\ga\frac{q}{\ga}}\right)^{\frac{\ga}{q}}\right)^\frac1\ga\\
	&=\norm{c}_{\ell^{q}(\La)}\left(\sum_{\overset{\bullet}{\bfw}\in\sD(\cG)}\left(\sum_{\overset{\bullet}{\bfu}\in\sD(\cG)}\overset{\bullet}{V}_\Phi \si\left(\left[\sfT_0(\bfw,\bfu)\right]^\bullet,\left[\cJ(\bfu-\bfw)\right]^\bullet\right)^{p}\right)^{\frac{\ga}{p}}\right)^\frac1\ga.
	\end{align*}
	Let us perform the following change of variables:
	\begin{equation}
	\overset{\bullet}{\bfth}\coloneqq\left[\cJ(\bfu-\bfw)\right]^\bullet\in\sD(\hcG)=\left[\cJ \La\right]^\bullet.
	\end{equation}
	Notice that $\cJ\La\subseteq\hcG\times\cG$ is a quasi-lattice. Then there exists $\bfdel\in\bU(\hcG)$ such that $\bfth+\bfdel=\cJ(\bfu-\bfw)$ and 
	\begin{equation*}
	\bfu-\bfw=\cJ^{-1}(\bfth+\bfdel)\,\Rightarrow\,\bfw=\bfu-\cJ^{-1}(\bfth)-\cJ^{-1}(\bfdel)\,\Rightarrow\,\overset{\bullet}{\bfw}=[\bfu-\cJ^{-1}(\bfth)]^\bullet,
	\end{equation*}
	since $-\cJ^{-1}(\bfdel)\in\bU(\cG)$. Recalling \eqref{Eq-T0-cosets} and writing $\bfth=(\theta,s)=((\theta_1,\theta_2),(s_1,s_2))\in\cJ\La$, we have
	\begin{align*}
	[\sfT_0(\bfw,\bfu)]^\bullet&=[\sfT_0(\bfu-\cJ^{-1}(\bfth),\bfu)]^\bullet=[\sfT_0((u-s,\nu+\theta),(u,\nu))]^\bullet\\
	&=[(u-s,\nu)]^\bullet=[\bfu-(s,\hat{e})]^\bullet.
	\end{align*}
	In the above calculation we can choose as representative of $\overset{\bullet}{\bfth}$ the only one in $\cJ\La$ without loss of generality. In fact,  write $\La=(\al\zd\times D_1)\times(\al\zd\times D_2)$, $\cJ\La=(\al\zd\times -D_2)\times(\al\zd\times D_1)$, and consider $\boldsymbol{\eta}=(\eta,l)=((\eta_1,\eta_2),(l_1,l_2))$  such that $\overset{\bullet}{\bfth}=\overset{\bullet}{\boldsymbol{\eta}}$ and $\boldsymbol{\eta}\notin\cJ\La$. Being $\bU(\hcG)=(\{0_{\rd}\}\times\cK^\perp)\times(\{0_{\rd}\}\times\cK)$, it necessarily follows that $\theta_1=\eta_1$ and $s_1=l_1$ in $\al\zd$, $[\theta_2]^\bullet=[\eta_2]^\bullet$ in $\hcG_0/\cK^\perp$, $[s_2]^\bullet=[l_2]^\bullet$ in $\cG_0/\cK$ and  $[(l,\hat{e})]^\bullet\in\overset{\bullet}{\La}$.\par
	Eventually we set 
	\begin{equation}
	\overset{\bullet}{\bfz}\coloneqq\overset{\bullet}{\bfu}-[(s,\hat{e})]^\bullet\in\sD(\cG)=\overset{\bullet}{\La}
	\end{equation} 
	and using Lemma \ref{Lem-quotient-coefficient-operator}
	\begin{align*}
	\Bigg(\sum_{\overset{\bullet}{\bfw}\in\sD(\cG)}\Bigg(\sum_{\overset{\bullet}{\bfu}\in\sD(\cG)}&\overset{\bullet}{V}_\Phi \si\Bigg(\left[\sfT_0(\bfw,\bfu)\right]^\bullet,\left[\cJ(\bfu-\bfw)\right]^\bullet\Bigg)^{p}\Bigg)^{\frac{\ga}{p}}\Bigg)^{\frac1\ga}\\
	&=\Bigg(\sum_{\overset{\bullet}{\bfth}\in\sD(\hcG)}\Bigg(\sum_{\overset{\bullet}{\bfz}\in\sD(\cG)}\overset{\bullet}{V}_\Phi\si(\overset{\bullet}{\bfz},\overset{\bullet}{\bfth})^{p}\Bigg)^{\frac{\ga}{p}}\Bigg)^\frac1\ga=\norm{\overset{\bullet}{\cC}_\Phi\si}_{\ell^{p,\ga}(\sD(\cG)\times\sD(\hcG))}\\
	&\lesssim \norm{\si}_{M^{p,\ga}(\cG\times\hcG)}<+\infty.
	\end{align*}
	\emph{Case $\ga>1$.} Observe that $p\geq\ga>1$ and $q\geq\ga>1$.
	Consider first $p\neq\infty$. The desired result is obtained by duality. By Proposition \ref{Pro-Duality-Mpq-Banach} $M^\ga(\cG)\cong(M^{\ga'}(\cG))'$, we hence show that if $f\in M^{q}(\cG)$ then $\Opz(\si)f$ is a continuous linear functional on $M^{\ga'}(\cG)$. Let $g\in M^{\ga'}(\cG)$, from the weak definition \eqref{Eq-KonNirenberg-weak} and the fact that $M^{p,1}(\cG\times\hcG)\cong(M^{p',\infty}(\cG\times\hcG))'$ we get:
	\begin{equation*}
	\abs{\la\Opz(\si)f,g\ra}=\abs{\la\si,R(g,f)\ra}\leq\norm{\si}_{M^{p,1}}\norm{R(g,f)}_{M^{p',\infty}}.
	\end{equation*}
	The  indexes' conditions in  \eqref{Eq-continuity-R-condition-1} and \eqref{Eq-continuity-R-condition-2}  become
	\begin{align}
	\ga',q&\leq\infty,\\
	\frac{1}{\ga'}+\frac1q&\geq\frac1p.
	\end{align} 
	The first one is trivial, the second follows from the assumption \eqref{Eq-cont-KN-indexes}. 
	Therefore
	\begin{equation*}
	\norm{R(g,f)}_{M^{p',\infty}}\lesssim \norm{g}_{M^{\ga'}}\norm{f}_{M^q}
	\end{equation*}
	and the boundedness of $\Opz(\si)$ from $M^q(\cG)$ into $M^\ga(\cG)$ follows.\\
	If $p=\infty$ the argument is similar,   we use the duality \eqref{Eq-duality-Minfty1-M1infty} between $M^{\infty,1}$ and $M^{1,\infty}$.
\end{proof}

\begin{proposition}\label{Pro-KN-eigenfunction}
	Consider a symbol $\si$ on the phase space such that for some $0<p<\infty$
	\begin{equation}
	\si\in\bigcap_{\ga>0}M^{p,\ga}(\cG\times\hcG).
	\end{equation}
	If $\lambda\in\si_P(\Opz(\si))\smallsetminus\{0\}$, then any eigenfunction $f\in L^2(\cG)$ with eigenvalue $\lambda$ satisfies  $f\in \bigcap_{\ga>0}M^\ga(\cG)$.
\end{proposition}
\begin{proof}
	We use Theorem \ref{Th-cont-KN-3-indici} and follow the proof pattern of \cite[Proposition 3.5]{BasCorNic20}.
\end{proof}

\section{Localization Operators on Groups}
The  aim of this section is to infer a result for $L^2$ eigenfunctions of localization operators which extends the one obtained in the Euclidean setting in \cite[Theorem 3.7]{BasCorNic20}. \par
 We address the reader to  Wong's book \cite{WongLocalization} for a detailed treatment of localization operators on locally compact Hausdorff groups and point out the recent works \cite{Luef1,Luef2}.  Let us recall their definition.
\begin{definition}
	Consider windows $\p_1,\p_2\in S_0(\cG)=\bG_1$ and symbol $a\in S'_0(\cG\times\hcG)$. Then the {\slshape localization operator with symbol $a$ and windows $\p_1,\p_2$} in $S_0(\cG)$ is  formally defined as
	\begin{equation}
		\A f(x)=\int_{\cG\times\hcG}a(u,\o)V_{\p_1}f(u,\o)M_\o T_u\p_2(x)\,dud\o.
	\end{equation}
	Equivalently, its  weak definition is
	\begin{equation}
		\la \A f,g\ra=\la a, \overline{V_{\p_1}f}V_{\p_2}g\ra\qquad\forall\,g\in S_0(\cG).
	\end{equation}
\end{definition} 
It is straightforward computation to check that
\begin{equation}
\A\colon S_0(\cG)\to S'_0(\cG)
\end{equation}
is well defined, linear and continuous  (cf. \cite[Theorem 5.3]{Jakobsen2018}).
Concretely, we shall mainly consider windows $\p_1,\p_2\in \cS_\cC(\cG)$ rather than in the whole Feichtinger algebra. Notice that if $a\in L^p(\cG\times\hcG)$, for any $1\leq p\leq\infty$,  then $\A\in B(L^2(\cG))$, cf.  \cite[Proposition 12.1, 12.2, 12.3]{WongLocalization}.\par
Given a function $F$ on $\cG\times\cG$, we introduce the operator $\Tb$:
\begin{equation}\label{EqDefTb}
	\Tb F(x,u)= F(x,u-x).
\end{equation}
Recall that $\cF_2$ stands for  the partial Fourier transform with respect to the second variable of measurable functions $\si$ defined on $\cG\times\hcG$. We shall consider $\cF_2 \si$ to be defined on $\cG\times\cG$, instead of $\cG\times\widehat{\hcG}$, due to the Pontryagin's duality. $\Tb$ and $\cF_2$ are automorphisms of $S_0(\cG\times \cG)$ and $S_0(\cG\times\hcG)$, respectively, which extend to automorphisms of $S'_0(\cG\times\cG)$ and $S'_0(\cG\times\hcG)$ by transposition. 

\begin{lemma}\label{LemKernelKN}
	Consider $\si\in S'_0(\cG\times\hcG)$ and  $f,g\in S_0(\cG)$. Then
	\begin{equation}
		\< \Opz(\si) f,g\>_{L^2(\cG)}=\< k_\si,g\otimes\overline{f}\>_{L^2(\cG\times\cG)}
	\end{equation} 
	where the kernel $k_\sigma$ is given by
	\begin{equation}
		k_\si(x,u)= \int_{\hcG}\si(x,\xi)\overline{\<u-x,\xi\>}\,d\xi=\Tb(\cF_2 \si(x,u)).
	\end{equation}
\end{lemma}
\begin{proof}
	The proof carries over from the Euclidean case almost verbatim, see e.g. \cite[formula (4.3)]{CordRod2020}.
\end{proof}

The following issue presents the connection between localization and Kohn-Nirenberg operators on LCA groups, extending the Euclidean case  proved in \cite[Proposition 2.16]{BasTeo21}.

\begin{proposition}\label{Pro-Loc-KN-form}
Consider windows  $\p_1,\p_2\in S_0(\cG)$ and a symbol $a\in S'_0(\cG\times\hcG)$. Then we have
	\begin{equation}
		\A=\Opz(a\ast R(\p_2,\p_1)).
	\end{equation}
\end{proposition}
\begin{proof} The proof is similar to the Eucliean case. We detail it for sake of clarity. We first compute the kernel  $k$ of  $\A$:
	\begin{align*}
		\<\A f,g\>&=\int_{\cG\times\hcG}a(x,\xi)\left(\int_{\cG}f(u)\overline{\pi(x,\xi)\p_1(u)}\,du\right)\left(\int_{\cG}\overline{g(y)}\pi(x,\xi)\p_2(y)\,dy\right)\,dxd\xi\\
		&=\int_{\cG\times\cG}f(u)\overline{g(y)}k(y,u)\,dydu,
	\end{align*}
	with
	\begin{equation*}
		k(y,u)=\int_{\cG\times\hcG}a(x,\xi)\overline{\pi(x,\xi)\p_1(u)}\pi(x,\xi)\p_2(y)\,dxd\xi.
	\end{equation*}
Using Lemma \ref{LemKernelKN}, we set $\Tb\circ\cF_2 (\si)=k$ and compute $\si$ using \eqref{Rtfs} as follows:
	\begin{align*}
		\cF^{-1}_2\circ\Tb^{-1}(k)&=\int_{\cG\times\hcG}a(x,\xi)\cF^{-1}_2\circ\Tb^{-1}\left(\pi(x,\xi)\p_2\otimes\overline{\pi(x,\xi)\p_1}(y,u)\right)\,dxd\xi\\
		&=\int_{\cG\times\hcG}a(x,\xi)\cF^{-1}_2\left(\pi(x,\xi)\p_2(y)\cdot\overline{\pi(x,\xi)\p_1(u+y)}\right)\,dxd\xi\\
		&=\int_{\cG\times\hcG}a(x,\xi)\pi(x,\xi)\p_2(y)\int_{\cG}\overline{\pi(x,\xi)\p_1(u+y)}\<\o,u\>\,du\,dxd\xi\\
		&=\int_{\cG\times\hcG}a(x,\xi)\pi(x,\xi)\p_2(y)\overline{\<\o,y\>}\overline{\cF(\pi(x,\xi)\p_1)(\o)}\,dxd\xi\\
		&=\int_{\cG\times\hcG}a(x,\xi)R(\pi(x,\xi)\p_2,\pi(x,\xi)\p_1)(y,\o)\,dxd\xi\\
		&=\int_{\cG\times\hcG}a(x,\xi)R(\p_2,\p_1)((y,\o)-(x,\xi))\,dxd\xi\\
		&=a\ast R(\p_2,\p_1)(y,\o).
	\end{align*}
	We then infer the thesis from the kernels' theorem \cite[Theorem B3]{Feichtinger-1980-kernel}.
\end{proof}

\begin{theorem}\label{Th-Loc-eigenfunction}
	Let $0<p<\infty$ and $a\in M^{p,\infty}(\cG\times\hcG)$. Consider $\p_1,\p_2\in\cS_\cC(\cG)\smallsetminus\{0\}$. Suppose that $\sigma_P(\A)\smallsetminus\{0\}\neq\varnothing$ and $\lambda\in\si_P(\A)\smallsetminus\{0\}$. Then any eigenfunction $f\in L^2(\cG)$ with eigenvalue $\lambda$ satisfies 
	\begin{equation}
	f\in\bigcap_{\ga>0}M^\ga(\cG).
	\end{equation}
\end{theorem}
\begin{proof}
	Observe that for $\p_1,\p_2\in\cS_\cC(\cG)$ we have $R(\p_2,\p_1)\in\sA_{\tilde{v}}(\cG\times\hcG)$, by Corollary \ref{Cor-Rfg-windows}. Therefore $R(\p_2,\p_1)$ belongs to every modulation space on the phase space; this is easily seen by using   \eqref{Eq-maximal-window-space-ALL-modulation-spaces}, the inclusion relations \eqref{Eq-Inclusions-Wiener-Voig-p113} and the inclusion between modulation spaces in Proposition \ref{Pro-Incl-Mpq}. Then the argument is the same as in \cite[Theorem 3.7]{BasCorNic20}: we write $\A$ in the Kohn-Nirenberg form (Proposition \ref{Pro-Loc-KN-form})
	\begin{equation}
	\A=\Opz(a\ast R(\p_2,\p_1)),
	\end{equation}
 use the convolution relations in Proposition \ref{Pro-convolution-Mpq} and infer the thesis applying Proposition \ref{Pro-KN-eigenfunction}.
\end{proof}


\appendix \label{appendix}
\section{}
We summarize the construction of coorbit spaces $\Co(Y)$, when $Y$ is a solid quasi-Banach function space on a locally compact Haudorff group $G$, even not abelian. This theory was first developed by Rauhut in \cite{Rauhut2007Coorbit} and technically fixed and deepened by Voigtlaender in his Ph.D. thesis \cite{Voig2015}. In the end we shall highlight the differences with the original theory for Banach spaces by Feichtinger and Gr\"ochenig, see \cite{FeiGro1988, FeiGro1989-I, FeiGro1989-II}.\\
 We mention that an exposition and treatment of the mentioned coorbit theory is now available also in the recent article \cite{VelthovenVoigtlaender2022} from van Velthoven and Voigtlaender, where the requirements of the weights are lightened up. However, due to the time when this work was written, we shall stick to the first version presented in \cite{Voig2015}. Moreover, on account of the objects of our particular setting, this makes no difference.
\par
In this section we deal with a locally compact, Hausdorff, $\sigma$-compact group $G$; the notation is intended to distinguish from the abelian case $\cG$. The group operation on $G$ will be expressed as multiplication; whenever a measure on $G$ is involved, it is understood to be the left Haar measure. We shall not list systematically the known properties for the spaces introduced in the sequel, but rather recall them when necessary. The reader is invited to consult \cite[Chapter 2]{Voig2015} for an exhaustive treatment. 
\par
Given $x\in G$ and a function $f$ on $G$, we define the {\slshape left} and {\slshape right translation operators} as 
\begin{equation}
L_x f(y)= f(x^{-1}y),\qquad R_x f(y)= f(yx).
\end{equation}
\begin{definition}\label{Def-QBF-space-Y}
	We say that $\left(Y,\norm{\cdot}_Y\right)$ is a {\slshape function space on $G$} if it is a quasi-normed space consisting of equivalence classes of measurable complex-valued functions on $G$, where two functions are identified if they coincide a.e.. \par
	A function space $\left(Y,\norm{\cdot}_Y\right)$ on $G$ is said to be {\slshape left invariant} if $L_x\colon Y\to Y$ is well defined and bounded for every $x\in G$, similarly we define the {\slshape right invariance}. We say that $Y$ is {\slshape bi-invariant} if it is both left and right invariant.\par
	A function space $(Y,\norm{\cdot}_Y)$ on $G$ is said {\slshape solid} if given $g\in Y$ and $f\colon G\to\bC$ measurable the following holds true:
	\begin{equation*}
	\abs{f}\leq\abs{g}\quad\text{a.e.}\qquad\Rightarrow\qquad f\in Y,\quad\norm{f}_Y\leq\norm{g}_Y;
	\end{equation*}
	$Y$ is called {\slshape quasi-Banach function (QBF) space on $G$} if it is complete.
\end{definition}
Without loss of generality, we can assume $\norm{\cdot}_Y$ to be a $r$-norm, $0<r\leq1$, i.e.
\begin{equation*}
\norm{f+g}^r_Y\leq \norm{f}^r_Y + \norm{g}^r_Y,\quad\forall f,g\in Y.
\end{equation*}
This is due to the Aoki-Rolewicz Theorem and the fact that equivalent quasi-norm induce the same topology (\cite[Theorem 2.1.4, Lemma 2.1.5]{Voig2015}).\par
It can be useful to describe Wiener Amalgam spaces, defined below, in terms of sequences. To this end, the so-called BUPUs and a particular space of sequences $Y_d$ associated to $Y$ are introduced. We present the space $Y_d$ under specific hypothesis fitting our framework, nevertheless a more general theory is possible, see \cite{Rauhut2007Winer} and \cite[Chapter 2]{Voig2015}.
\begin{definition}
	A family $X=\{x_i\}_{i\in I}$ in $G$ is called {\slshape relatively separated} if for all compact sets $K\subseteq G$ we have
	\begin{equation}\label{Eq-Constant-Rel-Sep-Fam}
	C_{X,K}\coloneqq\sup_{i\in I}\#\{j\in I\,|\,x_i K\cap x_j K\neq\varnothing\}<+\infty,
	\end{equation}
	where $\# S$ is the cardinality of a set $S$. Consider $X=\{x_i\}_{i\in I}$ relatively separated family in $G$, $Q\subseteq G$ measurable, relatively compact set of positive measure and $(Y,\norm{\cdot}_Y)$ solid QBF space on $G$. Then the {\slshape discrete sequence space associated to $Y$} is the set
	\begin{equation}
	Y_d(X,Q)=\left\{\left(\lambda_i\right)_{i\in I}\in \bC^I\,|\,\sum_{i\in I}\abs{\lambda_i}\chi_{x_i Q}\in Y\right\}
	\end{equation}
	endowed with the quasi-norm
	\begin{equation}
	\norm{\left(\lambda_i\right)_{i\in I}}_{Y_d(X,Q)}=\left\|\sum_{i\in I}\abs{\lambda_i}\chi_{x_i Q}\right\|_Y;
	\end{equation}
	$\bC^I$ is the space of functions from $I$ into $\bC$.
\end{definition}
If $G$ is $\sigma$-compact, then any relatively separated family $X$ is (at most) countable, see \cite[Lemma 2.3.10]{Voig2015}. In the setting presented so far, $Y_d(X,Q)$ is a quasi-Banach space. Moreover, if $Y$ is right invariant then $Y_d(X,Q)$ is independent of $Q$ in the sense that another $U\subseteq G$ measurable, relatively compact and with non empty interior yields the same space with an equivalent quasi-norm (cf. \cite[Lemma 2.2]{Rauhut2007Winer} and \cite[Lemma 2.3.16]{Voig2015}).
\begin{definition}\label{Def-BUPU}
	Let $U\subseteq G$ be a relatively compact, unit neighbourhood. A family $\Psi=\{\psi_i\}_{i\in I}$ of continuous functions on $G$ is called a {\slshape bounded uniform partition of unity of size $U$} ($U$-BUPU) if
	\begin{enumerate}
		\item[(i)] $0\leq \psi_i(x)\leq 1$ for all $x\in G$ and every $i\in I$;
		\item[(ii)] there exists $X=\{x_i\}_{i\in I}$ {\slshape $U$-localizing family for $\Psi$}, i.e., $X$ is a relatively separated family in $G$ such that
		\begin{equation*}
		\supp\psi_i\subseteq x_i U\qquad\forall i\in I;
		\end{equation*}
		\item[(iii)] $\sum_{i\in I}\psi_i\equiv 1$.
	\end{enumerate}
\end{definition}
Given any relatively compact unit neighbourhood $U$ in $G$, there  always exists a family $\Psi$ which is a  $U$-BUPU with some $U$-localizing family $X$ (\cite[Theorem 2]{Feichtinger_1981_Characterization}, \cite[Lemma 2.3.212]{Voig2015}) and, since $G$ is $\si$-compact, the indexes' set is (at most) countable. \par
We introduce the Wiener Amalgam spaces not in their full generality, but restrict ourselves to cases which ensure ``good" properties.
\begin{definition}\label{definizione-Wiener}
	Consider $Q\subseteq G$ measurable, relatively compact, unit neighbourhood and $f\colon G\to\bC$ measurable. We call {\slshape maximal function of $f$ with respect to $Q$} the following application
	\begin{equation}
	\sfM_Q f\colon G\to[0,+\infty], \quad x\mapsto \underset{y\in xQ}{\essupp}\abs{f(y)}.
	\end{equation}
	We fix a solid QBF space $(Y,\norm{\cdot}_Y)$ on $G$ and define the {\slshape Wiener Amalgam space with window $Q$, local component $L^\infty=L^\infty(G)$ and global component $Y$} as
	\begin{equation}
	W_Q(Y):= W_Q(L^\infty,Y)=\left\{f\in L^\infty_{loc}(G)\,|\,\sfM_Q f\in Y\right\}
	\end{equation}
	and endow it with
	\begin{equation}
	\norm{f}_{W_Q(Y)}\coloneqq\norm{f}_{W_Q(L^\infty,Y)}=\norm{\sfM_Q f}_Y.
	\end{equation}
\end{definition}
It was proven in \cite[Lemma 2.3.4]{Voig2015} that the maximal function $\sfM_Q f$ is measurable. Under the assumptions of the above definition, the Wiener Amalgam space\\ $(W_Q(Y),\norm{\cdot}_{W_Q(Y)})$ is a solid QBF space on $G$, in particular, $\norm{\cdot}_{W_Q(Y)}$ is a $r$-norm, $0<r\leq1$, if $\norm{\cdot}_Y$ is. For each $f\in L^\infty_{loc}(G)$ we have
\begin{equation}
\abs{f(x)}\leq \sfM_Q f(x)\quad\text{a.e.},
\end{equation}
which together with the solidity of $Y$ gives the continuous embedding
\begin{equation}
W_Q(L^\infty,Y)\hookrightarrow Y.
\end{equation}
In general the definition of $W_Q(Y)$ may depend on the chosen subset $Q$. However, we shall require some further properties in order to make the Wiener space independent of it. We collect some of the results of \cite[Lemma 2.3.16, Theorem 2.3.17]{Voig2015} in the following lemma (which holds under milder assumptions).
\begin{lemma}\label{Lem-WienerSpaceIndependence}
	Under the hypothesis presented so far, if the solid QBF space $Y$ on $G$ is right invariant, then the following equivalent facts hold true:
	\begin{itemize}
		\item[(i)] The Wiener Amalgam space $W_Q(L^\infty,Y)$ is right invariant for each measurable, relatively compact, unit neighbourhood $Q\subseteq G$;
		\item[(ii)] The Wiener Amalgam space $W_Q(L^\infty,Y)$ is independent of the choice of the measurable, relatively compact, unit neighbourhood $Q\subseteq G$, in the sense that different choices yield the same set with equivalent quasi-norms. The equivalence constants depend only on the two sets $Q,Q'\subseteq G$ and on $Y$.
	\end{itemize}
	If these conditions are fulfilled, $\Psi=\{\psi_i\}_{i\in I}$ is a $U$-BUPU for some localizing family $X=\{x_i\}_{i\in I}$ and $U\subseteq G$ relatively compact unit neighbourhood, then
	\begin{equation}\label{Eq-discrete-equiv-norm-Wiener}
	\norm{f}_{W_Q(L^\infty,Y)}\underset{X,Q,Y}{\asymp}\norm{\left(\norm{\psi_i \cdot f}_{L^\infty}\right)_{i\in I}}_{Y_d(X,Q)}
	\end{equation}
	for every $f\in W_Q(L^\infty,Y)$ and the constants involved in the above equivalence depend only on $X$, $Q$ and $Y$. 
\end{lemma}
We remark that the right invariance of $Y$ is sufficient for conditions (i) or (ii)  but not necessary; the existence of an $U$-BUPU $\Psi$ is always guaranteed. When one of the above conditions is satisfied, we suppress the index $Q$ in the Wiener space and simply write $W(L^\infty,Y)$ or $W(Y)$.\par
By considering $Qx$ instead of $xQ$ in the definition of the maximal function, we obtain the ``right-sided" version of the Wiener spaces.
So that we set the {\slshape right-sided maximal function } to be
\begin{equation}
\sfM^R_Q f\colon G\to[0,+\infty],\quad  x\mapsto \underset{y\in Qx}{\essupp}\abs{f(y)}
\end{equation}
and define the {\slshape right-sided Wiener Amalgam space} $W^R_Q(Y)$ similarly as before.
Analogous considerations   hold for $W^R_Q(Y)$, with the proper cautions about Lemma \ref{Lem-WienerSpaceIndependence}. In particular, the independence of $W^R_Q(Y)$ from $Q$ is guaranteed if $Y$ is left invariant, see \cite[Lemma 2.3.29]{Voig2015}.
\begin{definition}\label{Def-weights}
	A {\slshape weight on $G$} is a measurable function $m\colon G\to(0,+\infty)$. A weight $v$ is said to be {\slshape submultiplicative } if
	\begin{equation}
	v(xy)\leq v(x)v(y),\qquad\forall x,y\in G.
	\end{equation}
	Given two weights $m$ and $v$ on $G$, $m$ is said to be {\slshape left-moderate w.r.t. $v$} if
	\begin{equation}
	m(xy)\lesssim v(x)m(y),\qquad\forall x,y\in G,
	\end{equation}
	it is {\slshape right-moderate w.r.t. $v$} if 
	\begin{equation}
	m(xy)\lesssim m(x)v(y),\qquad\forall x,y\in G.
	\end{equation}
	If a weight $m$ is both left- and right-moderate w.r.t. $v$, we simply say that it is {\slshape moderate w.r.t. $v$} or {\slshape $v$-moderate.}\\
	Consider $v$ submultiplicative weight on $G$ which is also even, bounded from below and satisfies the Gelfand-Raikov-Shilov (GRS) condition, i.e. 
	\begin{equation}
	v(x)=v(-x)\quad\forall x\in G,\quad\exists c>0\,:\,v(x)\geq c\quad\forall x\in G,\quad\lim_{n\to+\infty}v(x^n)^{\frac1n}=1\quad\forall\,x\in G,
	\end{equation}  
	then the {\slshape class of weights on $G$ moderate w.r.t. $v$} is denoted as follows:
	\begin{equation}\label{Eq-Def-Class-Weights}
	\cM_v(G)=\left\{m\,\mathrm{weight\, on}\,G\,|\,m\,\mathrm{is}\,v\mathrm{-moderate}\right\}.
	\end{equation}
\end{definition}

\begin{remark}
	The GRS condition will be used in this paper only in the subsection dealing with Gabor frames, see Theorem \ref{Th-frame-L2}. In this framework, $v$ is a weight on the abelian group $\cG\times\hcG$, hence the GRS condition has the form
	\begin{equation*}
	\lim_{n\to+\infty}v(n\bfx)^{\frac1n}=1\quad\forall\,\bfx\in\cG\times\hcG.
	\end{equation*}
\end{remark}
We are now able to state the coorbit theory in \cite[Assumption 2.4.1]{Voig2015}, see  items  \textbf{A}--\textbf{G} and \textbf{H} --\textbf{J}. 
\begin{itemize}
	\item[\textbf{A}.] We assume $G$ to be a LCH, $\sigma$-compact group. We consider $\rho\colon G\to\cU(\cH)$ a strongly continuous, unitary, irreducible representation of $G$ for some nontrivial complex Hilbert space $\cH$. $\cU(\cH)$ denotes the group of unitary operators on $\cH$ (see e.g.,  \cite{Folland_AHA1995, WongLocalization} ).
	
	\item[\textbf{B}.] Given $f,g\in\cH$, we define the {\slshape (generalized) wavelet transform induced by $\rho$}, or voice transform, {\slshape of $f$ w.r.t. $g$} as
	\begin{equation}\label{Eq-DefWaveletGenericRho}
	W^\rho_g f\colon G\to\bC,\quad  x\mapsto\la f,\rho(x)g\ra_{\cH},
	\end{equation}
	where $\la\cdot{,}\cdot\ra_{\cH}$, also denoted by $\la\cdot{,}\cdot\ra$, is the inner product on $\cH$ supposed antilinear in the second component. $W^\rho_g f$ is always a continuous and bounded function on $G$, see \cite{WongLocalization}. We assume the representation $\rho$ to be {\slshape integrable}, i.e. there exists $g\in\cH\smallsetminus\{0\}$ such that $W^\rho_g g\in L^1(G)$; this implies that $\rho$ is also {\slshape square-integrable}: there exists $g\in\cH\smallsetminus\{0\}$ such that $W^\rho_g g\in L^2(G)$. Such a $g$ is said to be {\slshape admissible}.
	
	\item[\textbf{C}.] $(Y,\norm{\cdot}_Y)$ will be supposed to be a solid QBF space on $G$ with $\norm{\cdot}_Y$, or some equivalent quasi-norm, $r$-norm with $0<r\leq1$.
	
	\item[\textbf{D}.] The Wiener Amalgam space $W_Q(L^\infty,Y)$ is assumed right invariant for each measurable, relatively compact, unit neighborhood $Q\subseteq G$. We consider a submultiplicative weight $w\colon G\to(0,+\infty)$ such that for some (and hence each) measurable, relatively compact, unit neighborhood $Q\subseteq G$
	\begin{equation}\label{Eq-Condition-D-1}
	w(x)\underset{Q}{\gtrsim}\opnorm{R_x}_{W_Q(Y)\to W_Q(Y)}
	\end{equation} 
	and
	\begin{equation}\label{Eq-Condition-D_2}
	w(x)\underset{Q}{\gtrsim} \Delta(x^{-1})\opnorm{R_{x^{-1}}}_{W_Q(Y)\to W_Q(Y)},
	\end{equation}
	where $\Delta(x)$ is the modular function on $G$. We also require the weight $w$ to be bounded from below, i.e. there exists $c>0$ such that $w(x)\geq c$ for every $x\in G$.\\
\end{itemize}
If the condition on $W_Q(Y)$ in \textbf{D} is satisfied, then the Wiener space is independent of $Q$, so that we can omit the lower index. Moreover, this is ensured if $Y$ is right invariant (Lemma \ref{Lem-WienerSpaceIndependence}).
\begin{itemize}	
	\item[\textbf{E}.] We fix a submultiplicative weight $v\colon G\to(0,+\infty)$, which will be called {\slshape control weight for $Y$}, such that
	\begin{equation}
	v\geq w,\qquad v\geq w_{\vee,r},
	\end{equation}
	where $w$ is defined in \textbf{D} and
	\begin{equation}
	w_{\vee, r}(x)=  w(x^{-1})\left[\Delta(x^{-1})\right]^{1/r}.
	\end{equation}
	\item[\textbf{F}.] The {\slshape class of good vectors} is defined to be
	\begin{equation}
	\bG_v=\left\{g\in\cH\,|\,W^\rho_g g\in L^1_v(G)\right\}
	\end{equation}	
	and supposed nontrivial, $\{0\}\subsetneq\bG_v$.
	
	\item[\textbf{G}.] The {\slshape class of analyzing vectors} is defined as
	\begin{equation}
	\bA^r_v=\left\{g\in\cH\,|\,W^\rho_g g\in W^R(L^\infty,W(L^\infty,L^r_v))\right\}
	\end{equation}
	and supposed nontrivial, $\{0\}\subsetneq\bA^r_v$.
\end{itemize}

\begin{remark}\label{Rem-A-subset-G}(i) Observe that, since $v$ is submultiplicative, $L^r_v(G)$ is bi-invariant. This implies that $W(L^r_v)$ is independent of the window $Q$ and it is left invariant, hence also $W^R(W(L^r_v))$ is independent of the window subset. Concretely, this allow us to work with the same $Q$:
	\begin{equation}
	\norm{W^\rho_g g}_{W^R(W(L^r_v))}\asymp\norm{\sfM_Q \sfM^R_Q W^\rho_g g}_{L^r_v},
	\end{equation}
	(see in Lemma \ref{Lem-SC-window-space}).	(ii) From the continuous embeddings for $0<r\leq 1$
	\begin{equation}\label{Eq-Inclusions-Wiener-Voig-p113}
	W^R(L^\infty,W(L^\infty,L^r_v))\hookrightarrow W(L^\infty,L^r_v)\hookrightarrow W(L^\infty,L^1_v)\hookrightarrow L^1_v,
	\end{equation}
	see \cite[p. 113]{Voig2015}, follows the inclusion $\bA^r_v\subseteq \bG_v$.
\end{remark}

\begin{itemize}
	\item[\textbf{H}.] For a fixed $g\in\bG_v\smallsetminus\{0\}$, the {\slshape space of test vectors} is the set
	\begin{equation}
	\cT_v=\left\{f\in\cH\,|\,W^\rho_g f\in L^1_v(G)\right\}
	\end{equation}
	endowed with the norm
	\begin{equation}
	\norm{f}_{\cT_v}=\norm{W^\rho_g f}_{L^1_v}.
	\end{equation}
\end{itemize}
$(\cT_v,\norm{\cdot}_{\cT_v})$ is a $\rho$-invariant Banach space which embeds continuously and with density into $\cH$ and it is independent from the choice of the window vector $g\in\bG_v\smallsetminus\{0\}$, see \cite[Lemma 2.4.7]{Voig2015}. Recall that often the notation $\cH^1_v$ is used in place of $\cT_v$, see e.g. \cite{FeiGro1988,FeiGro1989-I,FeiGro1989-II,Rauhut2007Coorbit}.
\begin{itemize}
	\item[\textbf{I}.] We call {\slshape reservoir} the Banach space 
	\begin{equation}
	\cR_v\coloneqq\cT_v^\neg=\left\{f\colon\cH^1_v\to\bC\,|\, \text{antilinear and continuous}\right\}.
	\end{equation}
	\item[\textbf{J}.] We can extend the wavelet transform to $f\in\cR_v$ and $g\in\cT_v$:
	\begin{equation}
	W^\rho_g f\colon G\to\bC, \quad x\mapsto{}_{\cR_v}\la f,\rho(x)g\ra_{\cT_v},
	\end{equation}
	where ${}_{\cR_v}\la \cdot{,}\cdot\ra_{\cT_v}$ is the duality between $\cR_v$ and $\cT_v$ that will be denoted simply by $\la\cdot{,}\cdot\ra$. We have that $W^\rho_g f\in C(G)\cap L^\infty_{1/v}(G)$.
	\item[\textbf{K}.] For a fixed vector window $g\in\bA^r_v\smallsetminus\{0\}$, the {\slshape coorbit space on $G$ with respect to $Y$} is defined as 
	\begin{equation}
	\Co(Y)=\left\{f\in\cR_v\,|\,W^\rho_g f\in W(L^\infty,Y)\right\}
	\end{equation}
	endowed with the quasi-norm
	\begin{equation}
	\norm{f}_{\Co(Y)}=\norm{W^\rho_g f}_{W(L^\infty,Y)}.
	\end{equation}
\end{itemize}
The coorbit space $\Co(Y)$ is independent of $g\in\bA^r_v\smallsetminus\{0\}$, in the sense that different windows yield equivalent quasi-norms. Moreover, $(\Co(Y),\norm{\cdot}_{\Co(Y)})$ is a quasi-Banach space continuously embedded into $\cR_v$ and $\norm{\cdot}_{\Co(Y)}$ is a $r$-norm, $0<r\leq1$, if $\norm{\cdot}_Y$ is. We refer to \cite[Theorem 2.4.9]{Voig2015}.\par
In the following theorem we collect \cite[Theorem 2.4.19, Remark 2.4.20]{Voig2015}.
\begin{theorem}\label{Th-Voigt-Th-2.4.19}
	For every $g\in\bA^r_v\smallsetminus\{0\}$ there exists $U_0\subseteq G$ relatively compact unit neighbourhood such that for each $U_0$-BUPU $\Psi=\{\psi_i\}_{i\in I}$ with localizing family $X=\{x_i\}_{i\in I}$ the following hold true:
	\begin{itemize}
		\item[(i)] for each $i\in I$ there exists a continuous linear functional
		\begin{equation*}
		\lambda_i\colon\cR_v\to\bC
		\end{equation*}
		such that $\left(\lambda_i(f)\right)_{i\in I}\in Y_d(X)$ for every $f\in\cR_v$ and 
		\begin{equation}
		f=\sum_{i\in I}\lambda_i(f)\rho(x_i)g,,\qquad\forall\,f\in\Co(Y),
		\end{equation}
		where the sum converges unconditionally in the w-$\ast$-topology of $\cR_v$. If the finite sequences are dense in $Y_d(X)$, then the series converges unconditionally in $\Co(Y)$;
		\item[(ii)] for all $\lambda=\left(\lambda_i\right)_{i\in I}\in Y_d(X)$ the series 
		\begin{equation}
		\textsf{S}^X_g(\lambda)\coloneqq\sum_{i\in I}\lambda_i\rho(x_i)g
		\end{equation}
		is an element of $\Co(Y)$. The above sum converges unconditionally in the w-$\ast$-topology of $\cR_v$ (pointwise). If the finite sequences are dense in $Y_d(X)$, then the series converges unconditionally in $\Co(Y)$ and there exists $C>0$ such that
		\begin{equation}
		\norm{\textsf{S}^X_g(\lambda)}_{\Co(Y)}\leq C\norm{\left(\lambda_i\right)_{i\in I}}_{Y_d(X)},\qquad\forall\,\lambda\in Y_d(X);
		\end{equation}
		\item[(iii)] for $f\in\cR_v$ we have
		\begin{equation}
		f\in\Co(Y)\Leftrightarrow\left(\lambda_i(f)\right)_{i\in I}\in Y_d(X)
		\end{equation}
		and for every $f\in\Co(Y)$
		\begin{equation}
		\norm{f}_{\Co(Y)}\asymp\norm{\left(\lambda_i(f)\right)_{i\in I}}_{Y_d(X)}.
		\end{equation}
	\end{itemize}
\end{theorem}

\begin{remark}
	Let us remark the main differences with the Banach setting considered by Feichtinger and Gr\"ochenig \cite{FeiGro1988}:
	\begin{itemize}
		\item[(i)] in \cite{FeiGro1988} a solid Banach function space $Y$ on $G$ is considered and supposed continuously embedded in $L^1_{loc}(G)$. In particular, we observe how the condition $Y\hookrightarrow L^1_{loc}(G)$ is restrictive, in fact even if one would allow $Y$ to be quasi-Banach, all the spaces $L^p(G)$ with $0<p<1$ would be excluded;
		
		\item[(ii)] the window space considered in the construction of the coorbit space is larger than the one presented so far, namely it is sufficient a non-zero $g\in\cA_v\coloneqq\bG_v$ and 
		\begin{equation}
		\CoFG(Y)\coloneqq\left\{f\in\cR_v\,|\,W^\rho_g f\in Y\right\}
		\end{equation}
		with obvious norm. Hence $\CoFG(Y)$ is a Banach space independent of the chosen window $g\in\cA_v\smallsetminus\{0\}$
		.
	\end{itemize} 
\end{remark}
It is a natural question whether the two constructions coincide. In the Banach case the answer is positive, see \cite[Theorem 8.3]{FeiGro1989-II} and \cite[Theorem 6.1]{Rauhut2007Coorbit}.

\begin{theorem}\label{Th-Rauhut-Coo-Th6.1}
	Consider a solid Banach function space $Y$ such that it is bi-invariant and continuously embedded in $L^1_{loc}(G)$. Then
	\begin{equation*}
	\CoFG(Y)=\Co(Y)
	\end{equation*}
	with equivalent norms.
\end{theorem}


\section*{Acknowledgments}  
We thank F. Nicola for fruitful conversations on this topic and H. G. Feichtinger for valuable comments which contributed to improve our work.

\end{document}